\documentclass[11pt, letterpaper,english]{amsart}
\usepackage[left=1in,right=1in,bottom=1in,top=1in]{geometry}
\usepackage{amsmath,amsthm,amssymb}
\usepackage{mathrsfs}
\usepackage[all,cmtip]{xy}
\usepackage{cite}
\usepackage{hyperref}
\usepackage[shortlabels]{enumitem}
\usepackage{graphicx}
\usepackage{lipsum}
\usepackage[usenames,dvipsnames]{color}
\usepackage[normalem]{ulem}

\theoremstyle{plain}
\newtheorem{thmA}{Theorem}

\newtheorem{lem}{Lemma}
\newtheorem{lemma}[lem]{Lemma}
\newtheorem{thm}[lem]{Theorem}
\newtheorem*{thm*}{Theorem}

\newtheorem{prop}[lem]{Proposition}
\newtheorem{proposition}[lem]{Proposition}
\newtheorem{cor}[lem]{Corollary}
\newtheorem{corollary}[lem]{Corollary}
\newtheorem{conj}[lem]{Conjecture}

\theoremstyle{definition}
\newtheorem{defn}[lem]{Definition}

\newtheorem{remark}[lem]{Remark}
\newtheorem*{questions}{Question}

\numberwithin{equation}{subsection}
\numberwithin{lem}{subsection}

\newcommand{\mathfont}{\mathbb}

\newcommand{\C}{\mathfont C}
\newcommand{\Z}{\mathfont Z}
\newcommand{\ZZ}{\mathfont Z}

\newcommand{\Q}{\mathfont Q}
\newcommand{\QQ}{\mathfont Q}
\newcommand{\Qbar}{{\overline{\QQ}}}

\newcommand{\F}{\mathfont F}

\newcommand{\bA}{\mathbb{A}}

\newcommand{\bC}{\mathbb{C}}

\newcommand{\bE}{\mathbb{E}}
\newcommand{\bF}{\mathbb{F}}
\newcommand{\bG}{\mathbb{G}}

\newcommand{\bL}{\mathbb{L}}

\newcommand{\bP}{\mathbb{P}}
\newcommand{\bQ}{\mathbb{Q}}

\newcommand{\bX}{\mathbb{X}}

\newcommand{\bZ}{\mathbb{Z}}

\newcommand{\cC}{\mathcal{C}}

\newcommand{\cE}{\mathcal{E}}

\newcommand{\cH}{\mathcal{H}}
\newcommand{\cI}{\mathcal{I}}

\newcommand{\cM}{\mathcal{M}}

\newcommand{\cO}{\mathcal{O}}
\newcommand{\cP}{\mathcal{P}}

\newcommand{\cT}{\mathcal{T}}
\newcommand{\cU}{\mathcal{U}}
\newcommand{\cV}{\mathcal{V}}

\DeclareFontFamily{OT1}{rsfs}{}
\DeclareFontShape{OT1}{rsfs}{n}{it}{<-> rsfs10}{}
\DeclareMathAlphabet{\mathscr}{OT1}{rsfs}{n}{it}

\newcommand{\into}{\hookrightarrow}

\DeclareMathOperator{\Hom}{Hom}
\DeclareMathOperator{\Ker}{ker}

\DeclareMathOperator{\Aut}{Aut}
\DeclareMathOperator{\Gal}{Gal}
\DeclareMathOperator{\tr}{tr}

\newcommand{\Tr}{\mathrm{Tr}}
\DeclareMathOperator{\Sym}{Sym}

\DeclareMathOperator{\Frob}{Frob}
\DeclareMathOperator{\et}{\acute{e}t}

\DeclareMathOperator{\Spec}{Spec}

\DeclareMathOperator{\GL}{GL}
\DeclareMathOperator{\SL}{SL}
\DeclareMathOperator{\PSL}{PSL}
\DeclareMathOperator{\PGL}{PGL}

\newcommand{\git}{\mathbin{
  \mathchoice{\mkern-3mu/\mkern-6mu/\mkern-3mu}% \displaystyle
    {\mkern-3mu/\mkern-6mu/\mkern-3mu}% \textstyle
    {/\mkern-5mu/}% \scriptstyle
    {/\mkern-5mu/}}}% \scriptscriptstyle

\usepackage{tikz-cd}

\newcommand{\ps}[1]{[\![#1]\!]}

\newcommand{\rightiso}{\stackrel{\sim}{\longrightarrow}}

\renewcommand{\top}{{\operatorname{top}}}
\newcommand{\Stab}{\operatorname{Stab}}
\newcommand{\BP}{\textbf{P}}
\newcommand{\Mon}{\operatorname{Mon}}
\newcommand{\an}{{\operatorname{an}}}
\newcommand{\Frac}{\operatorname{Frac}}
\newcommand{\ol}[1]{\overline{#1}}
\newcommand{\ul}[1]{\underline{#1}}
\newcommand{\ab}{{\operatorname{ab}}}
\newcommand{\abs}{{\operatorname{abs}}}
\newcommand{\pre}{{\operatorname{pre}}}
\newcommand{\sh}{\text{sh}}

\newcommand{\Epi}{{\operatorname{Epi}}}
\newcommand{\ext}{{\operatorname{ext}}}
\newcommand{\Inn}{\operatorname{Inn}}
\newcommand{\spmatrix}[4]{\left[\begin{smallmatrix}#1&#2\\#3&#4\end{smallmatrix}\right]}
\newcommand{\spvector}[2]{\left[\begin{smallmatrix}#1\\#2\end{smallmatrix}\right]}
\newcommand{\mf}[1]{\mathfrak{#1}}
\usepackage{parskip}

\DeclareMathOperator{\Out}{Out}
\DeclareMathOperator{\rot}{rot}
\DeclareMathOperator{\Alt}{Alt}
\DeclareMathOperator{\Mod}{mod}
\DeclareMathOperator{\Soc}{Soc}

\newcommand{\Xbar}{\overline{X}}

\title[Tamely Ramified Covers of the Projective Line]{Tamely Ramified Covers of the Projective Line with Alternating and Symmetric Monodromy}
\date{March 2, 2022}

\author{Renee Bell}
\address{B\^at. 307,
Universit\'e Paris-Sud,
91405 Orsay Cedex,
France}
\email{renee.bell@universite-paris-saclay.fr}

\author{Jeremy Booher}
\address{School of Mathematics and Statistics, University of Canterbury, Private Bag 4800, Christchurch 8140, New Zealand}
\email{jeremy.booher@canterbury.ac.nz}

\author{William Y. Chen}
\address{Department of Mathematics, Columbia University, 2990 Broadway, New York, NY 10027}
\email{wchen@math.columbia.edu}

\author{Yuan Liu}
\address{Department of Mathematics, University of Michigan, 530 Church Street, Ann Arbor, MI 48104, USA}
\email{yyyliu@umich.edu}

\begin{document} 
\maketitle

\begin{abstract}
Let $k$ be an algebraically closed field of characteristic $p$ and $X$ the projective line over $k$ with three points removed.  
We investigate which finite groups $G$ can arise as the monodromy group of finite \'{e}tale covers of $X$ that are tamely ramified over the three removed points.  
 This provides new information about the tame fundamental group of the projective line. 
 In particular, we show that for each prime $p\ge 5$, there are families of tamely ramified covers with monodromy the symmetric group $S_n$ or alternating group $A_n$ for infinitely many $n$.  These covers come from the moduli spaces of elliptic curves with $\PSL_2(\F_\ell)$-structure, and the analysis uses work of Bourgain, Gamburd, and Sarnak, and adapts work of Meiri and Puder about Markoff triples modulo $\ell$. 
\end{abstract}

%Subject codes: 14H30, 11G20, 
%\blfootnote{Renee Bell was supported by European Research Council grant #715747 (to Fran\c{c}ois Charles)}

\tableofcontents

\section{Introduction}

\subsection{Three-Point Covers}
Let $k$ be an algebraically closed field and $X$ an algebraic curve over $k$.  If $k = \C$, it follows from the Riemann existence theorem that the \'{e}tale fundamental group of $X$, which we denote as $\pi_1(X)$, is the profinite completion of the topological fundamental group of the corresponding Riemann surface.  In this case, the \'{e}tale fundamental group is closely linked to topology.   
This connection is weaker in characteristic $p$, even for simple examples like $X = \bA^1_{\overline{\F}_p}$; Artin--Schreier theory implies that $\pi_1(\bA^1_{\ol{\bF}_p})$ is topologically infinitely generated, whereas the topological fundamental group of $(\bA^1_\C)^{\an}$ is trivial.  These Artin--Schreier covers of the affine line have no analog in characteristic zero, and in general 
covers of degree divisible by $p$ are responsible for much of the additional complexity that arises in characteristic $p$.

More precisely, suppose $X$ is a smooth affine curve over an algebraically closed field $k$ of characteristic $p$ obtained from a smooth projective connected curve $\Xbar$ of genus $g$ by removing $r$ points.  If $\widetilde{X}$ is a smooth lift of $X$ to an algebraically closed field of characteristic 0, then by the theory of specialization of the fundamental group, the maximal prime-to-$p$ quotients of the \'{e}tale fundamental groups are isomorphic, i.e. $\pi_1(X)^{(p')}\cong\pi_1(\widetilde{X})^{(p')}$ \cite[X, Cor 3.9]{SGA1}. 
This philosophy is extended in a conjecture of Abhyankar (now a theorem  of Harbater and Raynaud) which states that a finite group $G$ is a quotient of $\pi_1(X)$ if and only if its maximal prime-to-$p$ quotient is a quotient of $\pi_1(\widetilde{X})$ \cite{abhyankar57,harbater94,raynaud94}.  Since the fundamental group of the Riemann surface $\widetilde{X}(\bC)$ is free of rank $2g+r-1$, it follows that 
a finite group $G$ arises as the Galois group of a connected \'{e}tale cover of $X$ if and only if the maximal prime-to-$p$ quotient of $G$ is generated by $2g+r-1$ elements.

While Abhyankar's conjecture specifies the finite quotients of $\pi_1(X)$, this is not enough to determine $\pi_1(X)$ as it is not topologically finitely generated.  Furthermore, these results say nothing about the structure of the inertia groups of the covers under consideration.
By Grothendieck's theory of specialization \cite[XIII, Cor 2.12]{SGA1}, if $f : Y\rightarrow X$ is a $G$-Galois cover where $G$ is a quotient of $\pi_1(X)$ but not a quotient of $\pi_1(\widetilde{X})$, then $f$ is necessarily wildly ramified when extended to a branched cover of $\Xbar$.  Hence it is natural to study the difference between $\pi_1(\widetilde{X})$ and the tame fundamental group $\pi_1^t(X)$ which classifies \'{e}tale covers of $X$ that are tamely ramified when extended to a branched cover of $\Xbar$.  
So we are led to the following question:

 \begin{questions} \label{question}
 Which finite quotients of $\pi_1(\widetilde{X})$ are quotients of $\pi_1^t(X)$?  Equivalently, which finite groups generated by $2g+r-1$ elements arise as Galois groups of 
connected tamely ramified covers of $X$? 
 \end{questions}
 For examples of groups $G$ which are a finite quotient of $\pi_1(\widetilde{X})$ but not of $\pi_1^t(X)$, see \cite[p. 204]{kani86} and \cite[Proposition 4.3]{stevenson98}.  
  Since $\pi^t_1(X)$ is topologically finitely generated, a complete answer to this question would in some sense determine $\pi_1^t(X)$ (see \cite[Theorem 3.2.9]{rz10}).  
We are not aware of even a conjectural general answer, so we are interested in techniques for producing tamely ramified covers.

In this paper we will introduce a new technique for producing tamely ramified covers of $\bP^1_{\ol{\bF}_p} - \{0,1,\infty\}$ using moduli spaces of elliptic curves with $G$-structures. More precisely, by a three-point cover (in characteristic $p$) we mean a finite flat map of smooth curves $Y\rightarrow \bP^1_{\overline{\bF}_p}$ which is unramified away from $\{0,1,\infty\}$. A three-point cover is \emph{tame} if the ramification indices above $0,1,\infty$ are prime to $p$. One consequence of our methods is the following. 

\begin{thmA}[see Theorem \ref{thm:main_B_intro} below] \label{thm:main}
For any prime $p\ge 5$, there are infinitely many $n$ for which there exists a tame three-point cover defined over $\F_p$ which is Galois with Galois group isomorphic to $S_n$. The same is true for $A_n$, except we only show that the cover is defined over $\bF_{p^2}$.
\end{thmA}

This is a consequence of Theorem~\ref{thm:main_B_intro} below, which together with Theorem~\ref{thm:main_A} gives more precise information about which $n$ occur. Note that by Abhyankar's conjecture, every symmetric and alternating group is the Galois group of a three-point cover. The main point of our result is that our covers are \emph{tame}, even though in all but finitely many cases $p$ \emph{divides} the order of the Galois group.  

To date, the main tool for understanding good reduction of three-point covers is the following criterion of Obus--Raynaud, which comes from a detailed analysis of the stable reduction of three-point covers in characteristic 0:

\begin{thm*}[Obus, Raynaud \cite{obus17,raynaud99}] Let $G$ be a finite group with cyclic $p$-Sylow subgroup. Let $K_0 := \text{Frac}(W(k))$, where $k$ is an algebraically closed field of characteristic $p$. Let $K/K_0$ be a finite extension of degree $e(K)$, where $e(K)$ is less than the number of conjugacy classes of order $p$ in $G$. If $f : Y\rightarrow \bP^1_K$ is a three-point $G$-Galois cover defined over $K$ (as a $G$-Galois cover), then $f$ has potentially good reduction, realized over a tame extension $L/K$ of degree dividing the exponent of the center $Z(G)$ of $G$. In particular, if $Z(G)$ is trivial, then $f$ has good reduction.
\end{thm*}

In particular, one may attempt to use group theory to construct a three-point cover in characteristic zero with Galois group $G$ and with ramification indices coprime to $p$, reduce the cover modulo $p$, and then apply the theorem to deduce the reduction modulo $p$ is the desired cover.
 For an example with $G = \GL_m(\F_q)$ for appropriately chosen $m$ and $q$, see \cite[\S6]{obus17}.  
 Note that the theorem does not apply to almost all the groups in Theorem~\ref{thm:main}.  For example, it does not apply to the symmetric group $G = S_n$ if $n\geq 2p$, since in that case the $p$-Sylow subgroup of $S_n$ is not cyclic.

Our technique takes the opposite perspective to that of Obus--Raynaud. Instead of constructing covers in characteristic 0 and showing that they have good reduction, we consider covers which arise as maps between moduli spaces defined integrally which are automatically smooth and tamely ramified by virtue of the moduli problem we consider.  
Then working over the generic fiber, we study the monodromy for these maps of moduli spaces to determine 
exactly which groups we've managed to realize as quotients of $\pi_1^t(\bP^1_{\ol{\bF}_p} - \{0,1,\infty\})$. For this latter part, which in general is a difficult combinatorial problem\footnote{This is related to the question of Nielsen equivalence generating pairs in combinatorial group theory \cite[\S2]{Pak00}.}, we crucially rely on input from the work of Meiri--Puder \cite{mp18} and Bourgain--Gamburd--Sarnak \cite{bgs1,bgs16}.

More precisely, for a prime $\ell\ge 5$, let $\cM(\PSL_2(\bF_\ell))^{\abs}$ denote the moduli stack of elliptic curves equipped with a $\PSL_2(\bF_\ell)$-Galois cover defined \'{e}tale locally on the base (to be made more precise in \S\ref{section_moduli_of_G_structures}). Let $\cM(1)$ denote the moduli stack of elliptic curves. Its coarse moduli scheme $M(1)$ is isomorphic to the affine line $\Spec \bZ[j]$ with parameter given by the $j$-invariant. The forgetful map $\cM(\PSL_2(\bF_\ell))^\abs\rightarrow\cM(1)$ is finite \'{e}tale over $\bZ[1/|\PSL_2(\bF_\ell)|]$ and induces a map of coarse moduli schemes $M(\PSL_2(\bF_\ell))^\abs\rightarrow M(1)$.  For $p\nmid |\PSL_2(\bF_\ell)|$, its base change to $\ol{\bF}_p$ gives a smooth cover of $\bP^1_{\ol{\bF}_p}$, \emph{tamely ramified} above $j = 0,1728,\infty$. 

The scheme $M(\PSL_2(\bF_\ell))_{\Qbar}^\abs$ is never geometrically connected, but experimentally\footnote{We've checked by computer that every component of $M(\PSL_2(\bF_q))_{\Qbar}^\abs$ has alternating or symmetric monodromy over $M(1)_{\Qbar}$ for every prime power $q\le 43$.} every connected component we've computed has alternating or symmetric monodromy over $M(1)$. We cannot prove this in general, but by adapting the work of \cite{mp18}, we are able to establish this for a particular component. Specifically, we consider the open and closed substack
$$\cM(\PSL_2(\bF_\ell))^\abs_{-2}\subset\cM(\PSL_2(\bF_\ell))^\abs$$
corresponding to $\PSL_2(\bF_\ell)$-covers with ``trace invariant $-2$'' (see \S\ref{ss_trace_invariant}).
The degree of the induced map on coarse schemes is
$$n_\ell := \left\{\begin{array}{rl}
\frac{\ell(\ell+3)}{4} & \ell\equiv 1\mod 4 \\
\frac{\ell(\ell-3)}{4} & \ell\equiv 3\mod 4.
\end{array}\right.$$

In \S\ref{section_markoff_vs_moduli} we explain how any geometric fiber $\overline{F(\ell)}_{-2}$ of $\cM(\PSL_2(\bF_\ell))^\abs_{-2}$ over $\cM(1)$ can be identified with the set $Y^*(\ell)$ of equivalence classes of non-zero $\bF_\ell$-points of the affine surface $\bX$ defined by the Markoff equation
$$x^2 + y^2 + z^2 - xyz = 0;$$
two points are equivalent if one is obtained by negating two of the coordinates of the other.  (The $\bF_\ell$ points of $\bX$ also have a moduli interpretation as a fiber of $\cM(\SL_2(\bF_\ell))_{-2}^\abs$.) Under this identification, the monodromy action of $\pi_1(\cM(1)_{\Qbar})$ on the fiber $\ol{F(\ell)}_{-2}$ translates into the action of a certain group of automorphisms of $\bX$. This action was studied by Bourgain--Gamburd--Sarnak, who showed that for primes $\ell$ not in a density 0 ``exceptional'' set $\cE$, this action is transitive \cite{bgs1,bgs16}.  In particular, for any $\epsilon > 0$, the number of primes $\ell \le T$ with $\ell\in \cE$ is at most $T^\epsilon$ for $T$ large enough.  
They furthermore conjecture that this transitivity holds for all $\ell$.  
By Galois theory, this transitivity is the same as the connectedness of $\cM(\PSL_2(\bF_\ell))^\abs_{-2}$. 

We can understand the geometric monodromy of the map of coarse spaces $M(\PSL_2(\bF_\ell))^\abs_{-2} \to M(1)$ using work of Meiri and Puder \cite{mp18}.  Let
\begin{equation}\label{eq_property_P_ell}
    \BP(\ell) := \begin{array}{r}\text{The property that either $\ell\equiv 1 \mod 4$, or} \\
\text{the order of $\frac{3+\sqrt{5}}{2}\in\bF_{\ell^2}$ is at least $32\sqrt{\ell+1}$.}
\end{array}
\end{equation}
In the Appendix to \cite{mp18}, it is proven that $\BP(\ell)$ holds for a density $1$ set of primes $\ell$.  For a prime $\ell \not \in \cE$ for which $\BP(\ell)$ holds, the work of Meiri and Puder shows that the geometric monodromy group will contain the alternating group on the fiber.

We use these results to obtain the following asymptotic statement:

\begin{thmA}[see Theorem~\ref{thm:main_B}]\label{thm:main_B_intro}  For primes $\ell \geq 5$ outside of the exceptional set $\cE$, 
$M(\PSL_2(\bF_\ell))^\abs_{-2}$ is smooth and geometrically connected. 
If furthermore $\BP(\ell)$ holds, the geometric monodromy group over $M(1)$ satisfies
$$\Mon(M(\PSL_2(\bF_\ell))^\abs_{-2} /M(1))\cong \left\{\begin{array}{rl}
S_{n_\ell} & \ell\equiv 5,7,9,11 \mod 16 \\
A_{n_\ell} & \ell\equiv 1,3,13,15\mod 16.
\end{array}\right.$$
When $\ell\equiv 5,7,9,11 \mod 16$, the fiber at any prime $p \nmid |\PSL_2(\bF_\ell)| = \frac{\ell(\ell^2-1)}{2}$ is a tamely ramified three-point cover with monodromy group $S_{n_\ell}$ defined over $\bF_p$.  When $\ell\equiv 1,3,13,15\mod 16$, the fiber at any prime $p \nmid |\PSL_2(\bF_\ell)|$ is a tamely ramified three-point cover with monodromy group $A_{n_\ell}$ defined over $\bF_{p^2}$.
\end{thmA}

For a fixed prime $p\ge 5$, the set of $\ell$ for which the Theorem yields tame three-point covers in characteristic $p$ is the set of primes $\ell$ such that $\frac{\ell(\ell^2-1)}{2}\not\equiv0 \mod p$ (density $>0$), $\ell\notin\cE$ (density 1), and $\BP(\ell)$ holds (density 1). Thus, there is a positive density set of primes $\ell$ for which Theorem~\ref{thm:main_B_intro} yields the desired tame three-point cover mod $p$, from which we deduce Theorem~\ref{thm:main}.

We also prove a less precise result for sufficiently large $\ell$ which removes the restriction that $\ell \not \in \cE$.  
Specifically, let $M(\ell)$ denote the modular curve classifying elliptic curves with ``full level $\ell$ structure of determinant 1'' (see \S\ref{ss_rational_point} below). 
Let $\Gamma(1)'$ be the commutator subgroup of $\SL_2(\bZ)$. It can be checked that it is a torsion-free congruence subgroup of level 6 and index 12, and that the corresponding stack $\cM'$ over $\cM(1)$ is the complement of the zero section of an elliptic curve over $\bZ[1/6]$. The pullback
\begin{equation}\label{eq_mordell_covering_SL}
    \pi_\ell : M(\ell)' := M(\ell)\times_{\cM(1)} \cM'\longrightarrow\cM'
\end{equation}
is thus a $\SL_2(\bF_\ell)$-cover of a punctured elliptic curve and has good reduction at all $p\nmid 6\ell$. Taking the quotient by the center, we obtain a $\PSL_2(\bF_\ell)$-cover
\begin{equation}\label{eq_mordell_covering_PSL}
    \ol{\pi_\ell} : M(\ell)'/\{\pm I\}\longrightarrow\cM'
\end{equation}
with good reduction at all $p\nmid 6\ell$.  This cover defines a $\bQ$-point of $\cM(\PSL_2(\bF_\ell))^\abs$, and we let $\cM(\ol{\pi_\ell})$ denote the connected component containing $\ol{\pi_\ell}$.  Let $M(\ol{\pi_\ell})$ be its coarse moduli space.
One can compute that the covering $\ol{\pi_\ell}$ also has trace invariant $-2$, and hence $\cM(\ol{\pi_\ell})$ is a component of $\cM(\PSL_2(\bF_\ell))^\abs_{-2}$.  The conjecture of Bourgain, Gamburd, and Sarnak (see Conjecture~\ref{conj:markoff}) would imply that 
$\cM(\ol{\pi_\ell}) = \cM(\PSL_2(\bF_\ell))^\abs_{-2}$ (see Proposition~\ref{prop_Q_vs_Qplus}).

Let $d_\ell := \deg(M(\ol{\pi_\ell})/M(1))$.  The work of Bourgain, Gamburd, and Sarnak shows there is an integer $N_{1/2}$ such that $n_\ell - d_\ell \leq \ell^{1/2}$ when $\ell \geq N_{1/2}$.

\begin{thmA}\label{thm:main_A}
Fix a prime $\ell$ for which $\BP(\ell)$ holds.  If $\ell \equiv 1 \mod{4}$ assume that $\ell \geq \max(N_{1/2},13)$ while if $\ell \equiv 3 \mod{4}$ assume that $\ell \geq \max(N_{1/2},23)$.  
Then $M(\ol{\pi_\ell})$ is geometrically connected, smooth over $\bZ[1/|\PSL_2(\bF_\ell)|]$, and finite \'{e}tale over $M(1) - \{j = 0,1728\}$. The geometric monodromy group $\Mon(M(\ol{\pi_\ell})/M(1))$ is isomorphic to either $A_{d_\ell}$ or $S_{d_\ell}$. In either case, there is an at-most-quadratic extension $K_{\ell}$ of $\bQ$ with ring of integers $\cO_{K_\ell}$ such that the monodromy is defined over $\cO_{K_\ell}$.  Let $k_\ell$ be the residue field of a prime of $\cO_{K_\ell}$ lying above $p$ for $p \nmid |\PSL_2(\bF_\ell)|$.
Then the fiber over $k_\ell$  is  a smooth, geometrically connected, $\Mon(M(\ol{\pi_\ell})/M(1))$-cover of $\bP^1_{k_\ell}$ tamely ramified only over three points. Moreover, if $\Mon(M(\ol{\pi_\ell})/M(1))\cong S_{d_\ell}$, then we may take $K_\ell = \bQ$ and hence $k_\ell = \bF_p$.
\end{thmA}

\begin{remark}
The particular form of the monodromy groups of $M(\PSL_2(\bF_\ell))^{\abs}\rightarrow M(1)$ - in particular the fact that $|S_{n_\ell}|,|A_{n_\ell}|$ are divisible by many primes which don't divide $|\PSL_2(\bF_\ell)|$  - is crucial to ensure that our result is interesting; specifically, that we obtain many primes of good reduction that \emph{divide} the order of the monodromy group. 
It is a consequence of Belyi's theorem (see Theorem \ref{thm_faithful}) or \cite{ddh89} that every three-point cover can be realized as a map between moduli spaces. Thus, given an abstract three-point cover, one could try to find a moduli-interpretation of it, with the hopes of using that moduli-interpretation to deduce good reduction at ``interesting primes $p$''. A general procedure for finding a moduli-interpretation in terms of $G$-structures amounts to an effective version of Asada's theorem (see Theorem~\ref{thm_basic_properties}(7) below), which exists by the work of Bux--Ershov--Rapinchuk \cite{BER11} and a subsequent improvement by Ben-Ezra and Lubotzky \cite[Theorems 2.7, 2.9]{BL18}. However, their methods \emph{never} yield automatic good reduction at primes dividing the order of the monodromy group.
\end{remark}

\begin{remark} Let $F_2$ be a free group of rank 2. The connectedness of $M(\PSL_2(\bF_\ell))^\abs_{-2}$ is a consequence of the transitivity of the action of $\Aut^+(F_2)$ (the index two subgroup of $\Aut(F_2)$ consisting of automorphisms which induce automorphisms of determinant 1 on $F_2/[F_2,F_2]\cong\bZ^2$) on the set of equivalence classes $X := \{(a,b) \;|\; \text{$a,b$ generate $\SL_2(\bF_\ell)$ and $\tr[a,b] = -2\in\bF_\ell$}\}/\sim$, where two pairs are equivalent if they are conjugate by an element of $\SL_2(\ol{\bF}_\ell)$ (see \S\ref{section_markoff_vs_moduli}). On its face, this transitivity is a difficult problem in combinatorial group theory \cite[\S2]{LP01}. However, because $\SL_2(\bF_\ell)$ is the $\bF_\ell$-points of an \emph{algebraic group}, the set $X$ inherits an algebraic structure: it is the set of $\bF_\ell$ points of a certain \emph{character variety} on which $\Aut^+(F_2)$ acts via automorphisms of the variety. In the case of $\SL_2(\bF_\ell)$, this variety turns out to be an (affine) \emph{ruled surface}. The analysis of \cite{bgs1,bgs16} and \cite{mp18} make crucial use of this structure, especially the fact that $\Aut^+(F_2)$ is generated by the conjugates of an automorphism which correspond to ``rotations'' along the ruling. In general, if $G$ is a finite group of Lie type, this suggests that a deeper understanding of the stacks $\cM(G)$ may be obtained from the study of the associated character variety.
\end{remark}

\begin{remark}
We work with the moduli of elliptic curves with $G$-structures when $G=\SL_2(\bF_\ell)$ and $G = \PSL_2(\bF_\ell)$.  While the results of \cite{chen18} are applicable for any finite group $G$, our choice of $G$ allows us to leverage the work of \cite{bgs16} and \cite{mp18} to understand the connected components of $M(G)^{\abs}$ and the monodromy. Empirically, for other choices of $G$ there is more variability in the monodromy group.
For example, taking $G = A_7$ we compute that of the $17$ components of $M(A_7)^{\abs}$, two have symmetric monodromy, twelve have alternating monodromy, and three have smaller monodromy. While this diversity would complicate any analysis, it also provides opportunities to realize additional groups as quotients of $\pi_1^t(\bP^1_{\ol{\bF}_p}-\{0,1,\infty\})$. 
\end{remark}

\begin{remark}
Our results address a function field analog of a question articulated in \cite{rv15} about the existence of  number fields with little ramification and large Galois groups.  For a number field $K$ of degree $d$ over $\bQ$, say that $K$ is full if the associated Galois group is either $S_d$ or $A_d$.  
The question is whether given a set of places $\cP$ of $\bQ$, are there infinitely many full number fields unramified outside $\cP$?   

Roberts and Venkatesh construct examples by specializing maps between appropriately chosen Hurwitz spaces, but are unable to guarantee there are infinitely many specializations which produce non-isomorphic examples in order to prove \cite[Conjecture 8.1]{rv15}.  

Analogously, our results show that for any set of $\bF_p$-rational places $\cP$ of $\bF_p(t)$ with $\# \cP \geq 3$, there exist infinitely many non-isomorphic full extensions unramified outside of $\cP$.
\end{remark}

\subsection{Organization of the paper}
In \S\ref{section_moduli_of_G_structures}, we review the moduli of elliptic curves and the theory of $G$-structures, and explain how to use them to produce covers of the projective line.  We begin with the analytic theory in \S\ref{ss_analytic_theory}, which gives a concrete relation between the Galois theory for coverings of the moduli stack of elliptic curves and the Galois theory for coverings of its coarse moduli space (see Proposition \ref{prop_coarse_fibers} and Corollary \ref{cor_analytic_ramification}). While these results could have also been obtained algebraically, we find that the analytic perspective leads to a less technical exposition and naturally takes us through the necessary background to make clear the connection between our moduli stacks and the classical construction of modular curves as quotients of the upper half plane.  Next, in \S\ref{ss_arithmetic_moduli}, we recall the arithmetic theory of the moduli of elliptic curves with $G$-structures (see \cite{chen18}), whose moduli stacks $\cM(G)$ are finite \'{e}tale over the moduli stack of elliptic curves.  In \S\ref{ss_obtainingtame}, making crucial use of the algebraic theory, we explain how to use their coarse moduli schemes to obtain tame covers of the projective line in characteristic $p$. We then introduce the notion of absolute $G$-structures in \S\ref{ss_absolute_G_structures}, and study their Higman and trace invariants in \S\ref{ss_trace_invariant}.

In \S\ref{section_markoff_vs_moduli}, we prove Theorem~\ref{thm:main_B_intro}.  We review the work of Bourgain--Gamburd--Sarnak and Meiri--Puder on Markoff triples modulo $\ell$ in \S\ref{section_bgsmp}.  In \S\ref{ss_markoff_absolute}, 
we relate the geometric fibers of $\cM(\SL_2(\bF_\ell))_{-2}^{\abs}$ and $\cM(\PSL_2(\bF_\ell))_{-2}^{\abs}$ with Markoff triples modulo $\ell$ and relate the monodromy action on the fibers with the action of a group of automorphisms of the Markoff surface.  In \S\ref{ss_explicit_ramification}
we analyze the ramification of the coarse schemes of these stacks as covers of the $j$-line and produce a $\bQ$-rational point above $j = 0$. In \S\ref{ss_rational_point} we give a conceptual explanation of a rational point over $j=0$.  We give the proof of Theorem~\ref{thm:main_B_intro} in \S\ref{ss_proof_of_asymptotic_statement}.

Finally, in \S\ref{sec:largemarkoff}, we address the case that $\ell$ lies in the exceptional set $\cE$ of primes not covered by the arguments of \cite{bgs16}. By adapting the arguments of \cite{mp18} to apply to the largest orbit in $Y^*(\ell)$, we prove Theorem~\ref{thm:main_A} and obtain additional tamely ramified covers with symmetric or alternating Galois groups at the cost of only being able to bound the degree.

\subsection{Notation and Conventions}\label{ss_notations_conventions}
Throughout, $\Qbar$ will denote the algebraic closure of $\bQ$ in $\bC$.

Throughout this paper we will try to reserve the letter $p$ to refer to the characteristic of a field over which we are working, and $\ell$ will generally refer to a prime distinct from $p$.

Often script letters ``$\cM$'' will be used to denote a stack, in which case the corresponding Roman letter ``$M$'' will be used to denote its coarse space/scheme. 

For groups $F$ and $G$, $\Epi^{\ext}(F,G)$ is the set of equivalence classes of surjections $F\twoheadrightarrow G$ considered up to conjugation on $G$ (or equivalently in $F$).

When there is no risk of confusion, we will denote both the \'{e}tale fundamental group of a scheme and the topological fundamental group of an analytic space by $\pi_1$.  When there is risk of confusion, we use $\pi^{\et}_1$ and $\pi^{\top}_1$.

\subsection{Acknowledgements}  Bell was supported by European Research Council grant \#715747 (to Fran\c{c}ois Charles).  Booher was partially supported by the Marsden Fund Council administered by the Royal Society of New Zealand. Chen was partially supported by the National Science Foundation under Award No. DMS-1803357. We thank Piotr Achinger, Jordan Ellenberg, David Harbater, Julia Hartmann, Andrew Obus, and Rachel Pries for helpful conversations, and the American Mathematical Society and the organizers of the math research community ``Explicit Methods in Arithmetic Geometry in Characteristic $p$'' which was the genesis of this project. We also thank the anonymous referee for comments on an earlier draft.

\section{Moduli of elliptic curves with \texorpdfstring{$G$}{G}-structure}\label{section_moduli_of_G_structures}

Here we review the moduli stack of elliptic curves with $G$-structure and explain how to use their connected components to construct tamely ramified three-point covers in characteristic $p$.

\subsection{Analytic Moduli of Elliptic Curves}\label{ss_analytic_theory}

In this section we review the analytic theory of the moduli of elliptic curves. The main purpose is to explain the relationship between the geometric fiber of a finite \'{e}tale cover of the moduli stack with the geometric fiber of the corresponding map on coarse spaces (see Proposition \ref{prop_coarse_fibers} and Corollary \ref{cor_analytic_ramification}). By standard GAGA arguments, the same results will hold over algebraically over $\Qbar$ or $\bC$.

For an analytic elliptic curve $E$, a \emph{framing} on $E$ is a choice of basis $\mf{f} = (\mf{f}_1,\mf{f}_2)$ of $H_1(E,\Z)$ such that the intersection product $\mf{f}_1\cdot\mf{f}_2 =1$. Given a holomorphic family of elliptic curves over a complex analytic manifold, by Ehresmann's fibration theorem the family is topologically locally constant, and a framing on the family is defined to be a locally constant family of framings on the fibers. Let $\cT$ denote the moduli stack of framed elliptic curves, then $\cT$ is a complex manifold isomorphic to the upper half plane (see \cite[Proposition 2.4]{Hain11} and \cite[Proposition 10.1]{fm12}). A point of $\cT$ is thus an isomorphism class of framed elliptic curves.

Explicitly, this isomorphism can be defined as follows. Let $\cH := \{z\in\bC : \Im(z) > 0\}$ be the upper half plane. The action of $\bZ^2$ on $\cH\times\bC$ given by
$$(\tau,z)\cdot(n,m) = (\tau, z + n\tau + m)\qquad (\tau,z)\in\cH\times\bC,\; (n,m)\in\bZ^2$$
is free. Let $\bE := (\cH\times\bC)/\bZ^2$, then together with the zero section $\cH\times\{0\}$, $\bE$ is a family of elliptic curves over $\cH$. We give $\bE$ the structure of a framed family by specifying the ordered basis $(\mf{f}_1,\mf{f}_2)$ of $H_1(\bE_\tau,\Z)$ to be given by the straight-line paths $0\leadsto 1$ and $0\leadsto\tau$ respectively in $\bC$. This defines a framing $\mf{f}_{\bE}$ on $\bE$ and determines the isomorphism $\cH\rightiso\cT$.

Fix a framed elliptic curve $(E_0,\mf{f}_0)$ and let $\Gamma_{E_0}$ denote its (orientation-preserving) mapping class group (see \cite[\S2.1]{fm12}). Then $\Gamma_{E_0}$ acts on $\cT$ as follows. 
%Then $\Gamma_{E_0}$ acts on $\cT$ as follows: 
Given a framed elliptic curve $(E,\mf{f})$, up to homotopy there is a unique homeomorphism of elliptic curves $\phi_{\mf{f}_0,\mf{f}} : E_0\rightarrow E$ which respects the framings. Then $\Gamma_{E_0}$ acts on $\cT$ by the rule 
\begin{equation}\label{eq_action}
\gamma\cdot (E,\mf{f}) = (E,(\phi_{\mf{f}_0,\mf{f}}\circ f_\gamma\circ \phi_{\mf{f}_0,\mf{f}}^{-1})(\mf{f}))    
\end{equation}
where $f_\gamma$ is a self-homeomorphism of $E_0$ representing $\gamma$. This gives a \emph{right} action of $\Gamma_{E_0}$ on $\cT$. Since \eqref{eq_action} defines a free and transitive action of $\Gamma_{E_0}$ on the set of framings of $E$, the stack quotient $[\cT/\Gamma_{E_0}]$ is naturally the complex analytic moduli stack of elliptic curves, denoted $\cM(1)^\an$. Since $\cT$ is simply connected, relative to the base point $(E_0,\mf{f}_0)\in\cT$, the group $\Gamma_{E_0}$ acts (on the left) on the fiber functor associated to the Galois category of $\cM(1)^\an$ \cite[\S18]{noo05}. This action induces an isomorphism
\begin{equation}\label{eq_pi_1_isom}
\pi_1^\top(\cM(1)^\an,E_0)\stackrel{\sim}{\longrightarrow}\Gamma_{E_0}.
\end{equation}
The action of $\Gamma_{E_0}$ on $H_1(E_0,\bZ)$ also induces a canonical isomorphism \cite[Theorem 2.5]{fm12}
\begin{equation}\label{eq_MCG_on_homology}
\Gamma_{E_0}\stackrel{\sim}{\longrightarrow} \SL(H_1(E_0,\Z)).%\stackrel{\SL(\mf{f}_0^{-1})}{\longrightarrow}\SL_2(\bZ).
\end{equation}
We can relate the action of $\Gamma_{E_0}$ on $\cT$ to the usual action of $\SL_2(\bZ)$ on $\cH$ as follows. Let $\SL_2(\bZ)$ act on the set of framings of $E_0$ by
\begin{equation*}%\label{eq_SL2Z_on_framings}
\spmatrix{a}{b}{c}{d}\cdot\mf{f} := \spmatrix{0}{1}{1}{0}\spmatrix{a}{b}{c}{d}\spmatrix{0}{1}{1}{0}\spvector{\mf{f}_1}{\mf{f}_2} = \spvector{c\mf{f}_2 + d\mf{f}_1}{a\mf{f}_2 + b\mf{f}_1}\qquad\text{for $\spmatrix{a}{b}{c}{d}\in\SL_2(\bZ)$},
\end{equation*}
then this action is free and transitive and defines an action of $\SL_2(\bZ)$ on $\cT$.\footnote{We note that this action is \emph{not} the same as the one coming from the isomorphism $\bZ^2\rightiso H_1(E_0,\bZ)$ sending $(e_1,e_2)$ to $(\mf{f}_{0,1},\mf{f}_{0,2})$. This isomorphism would induce an isomorphism $\SL_2(\bZ)\cong\SL(H_1(E_0,\bZ))\cong\Gamma_{E_0}$ which would yield the action $\spmatrix{a}{b}{c}{d}\cdot\tau = \frac{b+d\tau}{a+c\tau}$ for $\tau\in\cH$.}
Via the isomorphism $(\bE,\mf{f}_\bE) : \cH\stackrel{\sim}{\longrightarrow}\cT$, one can check that this $\SL_2(\bZ)$-action on $\cT$ is transported to the usual $\SL_2(\bZ)$-action on $\cH$ given by fractional linear transformations: $\spmatrix{a}{b}{c}{d}\cdot\tau = \frac{a\tau + b}{c\tau + d}$. In particular, we see that the framed family $\bE$ induces an explicit universal cover $\cH\rightarrow\cM(1)^\an$, which induces an isomorphism $[\cH/\SL_2(\bZ)]\rightiso\cM(1)^\an$.

Similarly, the topological quotient $\cT/\Gamma_{E_0}\cong \cH/\SL_2(\bZ)$ is the \emph{coarse (analytic) moduli space} of elliptic curves, which we denote $M(1)^\an$. On $\cH$, the $j$-invariant is a holomorphic function which is invariant under $\SL_2(\bZ)$, and descends to an isomorphism $\cH/\SL_2(\bZ)\cong\bC$. Given an elliptic curve $E$, $E\cong\bE_\tau$ for some $\tau\in\cH$, and the $j$-invariant of $E$ is just $j(E) := j(\tau)$. We will often use the $j$ function to identify $\cH/\SL_2(\bZ)$ with $\bC$.

\begin{remark}\label{remark_modular_curves}  If we choose an isomorphism $\cM(1)^\an\cong[\cH/\SL_2(\bZ)]$ as above, then Galois theory gives a correspondence between finite index subgroups $\Gamma\le\SL_2(\bZ)$ and connected finite covers of $\cM(1)^\an \cong [\cH/\SL_2(\bZ)]$.  Explicitly, the cover corresponding to $\Gamma$ is simply the projection map $[\cH/\Gamma]\rightarrow[\cH/\SL_2(\bZ)]$, and its coarse space is the topological quotient $\cH/\Gamma$ (a ``modular curve''). From this perspective, all finite index subgroups of $\SL_2(\bZ)$ are treated equally. As discussed in \cite{chen18} and reviewed below, this allows us to describe a generalization of the notion of ``level structure'' for elliptic curves whose moduli spaces incorporate both congruence and noncongruence modular curves.
\end{remark}

For an analytic space $B$ and a family of elliptic curves $E_B\rightarrow B$, letting $E_{b_0}$ be the fiber of $E_B$ above $b_0 \in B$, there is a monodromy representation
$$\rho_{b_0} : \pi_1^\top(B,b_0)\longrightarrow \SL(H_1(E_{b_0},\Z)).$$
We may extend this construction to obtain a monodromy representation for the universal family over $\cM(1)^\an$
$$\rho_{E} : \pi_1^\top(\cM(1)^\an,E)\longrightarrow\SL(H_1(E,\bZ))$$
which is precisely the composition of the isomorphisms \eqref{eq_pi_1_isom} and \eqref{eq_MCG_on_homology}. If we view $E_B$ as a map $B\rightarrow\cM(1)^\an$, then $\rho_{b_0},\rho_{E_{b_0}}$ fit into a commutative diagram
\begin{equation}\label{eq_monodromy_vs_pi1}
\begin{tikzcd}
\pi_1^\top(B,b_0)\ar[r,"(E_B)_*"]\ar[rd,"\rho_{b_0}"'] & \pi_1^\top(\cM(1)^\an,E_{b_0})\ar[d, "\rho_{E_{b_0}}"] \\
 & \SL(H_1(E_{b_0},\Z))
\end{tikzcd}
\end{equation}

\subsubsection*{Coarse monodromy via stacky monodromy}
Now suppose $\cM\rightarrow\cM(1)^\an$ is any finite (\'{e}tale) cover, and let $M$ be the coarse moduli space of $\cM$; the universal property of coarse moduli spaces yields a commutative (but not cartesian!) diagram
\[\begin{tikzcd}
\cM\ar[d]\ar[r] & M\ar[d] \\
\cM(1)^\an\ar[r] & M(1)^\an
\end{tikzcd}\]
Let $\cM(1)^\circ\subset\cM(1)^\an$ denote the open substack parametrizing elliptic curves without fibers of $j$-invariant 0 or 1728. Its coarse space is just $M(1)^\circ := M(1)^\an - \{j = 0,1728\}$. Then the restricted cover $\cM^\circ\subset \cM$ is still a cover of $\cM(1)^\circ$, but moreover its coarse space $M^\circ$ is now also a finite cover of $M(1)^\circ$, corresponding to the fact that the $\SL_2(\bZ)$ action on $\cH$ descends to an action of $\PSL_2(\bZ)$ on $\cH$ which is \emph{free} on the complement of the orbits of $e^{2\pi i/6}$ and $i$ (where $j = 0,1728$). We wish to understand the monodromy of $M^\circ/M(1)^\circ$ via the monodromy of $\cM/\cM(1)^\an$.

Let $B := M(1)^\circ$ and let $E$ be an elliptic curve over $B$ with ``$j$-invariant $j$''. That is, it has no fibers of $j$-invariants $0,1728$ and the induced map $M(1)^\circ\stackrel{E}{\rightarrow}\cM(1)^\circ\rightarrow M(1)^\circ$ is the identity. For example, we may take $E$ to be given by
\begin{equation}\label{eq_j_invariant_j}
    y^2 + xy = x^3 - \frac{36}{j-1728}x - \frac{1}{j-1728}.
\end{equation}
Let $b_0\in B$ be a base point, then the induced map $B\rightarrow\cM(1)^\an$ sends $b_0$ to the elliptic curve $E_{b_0}$, and the fiber $\cM_{E_{b_0}}$ is by definition the underlying set of the fiber product $\{b_0\}\times_{\cM(1)^\an}\cM$. This fiber admits a natural action of $\Aut(E_{b_0}) = \{[\pm 1]\}$, as well as a natural monodromy action of $\pi_1^\top(\cM(1)^\an, E_{b_0})$. There is a natural map\footnote{If $x$ is the point of $\cM(1)$ corresponding to $E_{b_0}$, then in the language of \cite{noo04} and \cite{noo05}, this map is the canonical map from the ``inertial fundamental group'' or ``hidden fundamental group'' to the full fundamental group, and is denoted ``$\omega_x$''.}
$$\Aut(E_{b_0})\hookrightarrow \pi_1^\top(\cM(1)^\an,E_{b_0})$$
which is an isomorphism onto the center of $\pi_1^\top(\cM(1)^\an,E_{b_0})$, and the actions of $\Aut(E_{b_0})$ and $\pi_1^\top(\cM(1)^\an,E_{b_0})$ on $\cM_{E_{b_0}}$ are compatible with respect to this map. The map to the coarse space $\cM\rightarrow M$ induces a map $\cM_{E_{b_0}}\rightarrow M_{b_0}$ which is surjective and since $j(E_{b_0})\ne 0,1728$, it induces a bijection of sets
$$\alpha : \cM_{E_{b_0}}/\{[\pm1]\}\stackrel{\sim}{\longrightarrow} M_{b_0}.$$
%\yuan{Should all $\cM_{b_0}$ be $\cM_{E_{b_0}}$?}
We have a commutative diagram
\[\begin{tikzcd}
\cM^\circ_B\ar[d]\ar[r] & \cM^\circ\ar[r]\ar[d] & M^\circ\ar[d] \\
B\ar[r,"E"] & \cM(1)^\circ\ar[r] & M(1)^\circ
\end{tikzcd}\]
where the left square is cartesian and the right square is induced by the universal property of coarse spaces. Remembering that $B = M(1)^\circ$, the composition of the bottom row is the identity, and so the composition of the top row is a map of covers of $M(1)^\circ$ which induces the map $(\cM^\circ_B)_{b_0} = \cM_{E_{b_0}} \rightarrow M_{b_0}$. In particular, this map must be equivariant for the monodromy action of $\pi_1^\top(M(1)^\circ,b_0)$, and hence $\alpha$ is equivariant for $\pi_1^\top(M(1)^\circ,b_0)$ acting on $\cM_{E_{b_0}}$ via 
$$\pi_1^\top(M(1)^\circ,b_0)\stackrel{E_*}{\longrightarrow} \pi_1^\top(\cM(1)^\circ,E_{b_0})\longrightarrow\pi_1^\top(\cM(1)^\an,E_{b_0}).$$

\begin{prop}\label{prop_coarse_fibers} Let $\cM\rightarrow\cM(1)^\an$ be a finite cover and let $M\rightarrow M(1)^\an$ be the induced map on coarse spaces. Let $M^\circ\rightarrow M(1)^\circ$ be the restriction to the preimage over the complement of $j = 0,1728$. Let $E$ be any elliptic curve over $M(1)^\circ$ with ``$j$-invariant $j$''. 

\begin{enumerate}
\item \label{prop_coarse_fibers1} Let $b_0\in M(1)^\circ$, and let $\cM_{E_{b_0}}$ be the fiber of $\cM/\cM(1)^\an$ over $E_{b_0}$. Then the natural map $\cM\rightarrow M$ induces a bijection
$$\alpha : \cM_{E_{b_0}}/\{[\pm1]\}\stackrel{\sim}{\longrightarrow}M_{b_0}.$$

\item \label{prop_coarse_fibers2}
If the monodromy of $\cM/\cM(1)^\an$ is given by $\pi_1^\top(\cM(1)^\an,E_{b_0})\rightarrow\Aut(\cM_{E_{b_0}})$, and $Z$ is the center of $\pi_1^\top(\cM(1)^\an,E_{b_0})$, then the monodromy action of $\pi_1^\top(M(1)^\circ,b_0)$ on $M_{b_0}$ associated to the cover $M^\circ/M(1)^\circ$ is given by either path in the commutative diagram
\[\begin{tikzcd}
\pi_1^\top(M(1)^\circ,b_0)\ar[r,"E_*"] & \pi_1^\top(\cM(1)^\an,E_{b_0})\ar[r]\ar[d] & \pi_1^\top(\cM(1)^\an,E_{b_0})/ Z \ar[d] \\
 & \Aut(\cM_{E_{b_0}})\ar[r] & \Aut(\cM_{E_{b_0}}/\{[\pm1]\})\ar[r,"\alpha_*"] & \Aut(M_{b_0})
\end{tikzcd}\]
\end{enumerate}
\end{prop}
\begin{proof} 
The first statement was established in the discussion before the Proposition.  To check the second, it remains to prove that the composition $\pi_1^\top(M(1)^\circ,b_0)\rightarrow\Aut(M_{b_0})$ is independent of the choice of $E$.  This follows as the composition is simply the monodromy of $M^\circ/M(1)^\circ$, which makes no mention of $E$. It also follows from \S7 and \S8 of \cite{kod63} that the map $E_*$ is surjective, though we will not need this.
\end{proof}

For any elliptic curve $E/M(1)^\circ$ with ``$j$-invariant $j$'', the local monodromy on homology around $j = 0,1728$ or $\infty$ relative to some framing of a nearby fiber is conjugate to $\pm\spmatrix{1}{1}{-1}{0},\pm\spmatrix{0}{1}{-1}{0}$ or $\pm\spmatrix{1}{1}{0}{1}$ respectively. (See \cite[\S7-8]{kod63}, or we can also check this for our choice of $E$ as in (\ref{eq_j_invariant_j}) using Tate's algorithm.) By the commutativity of (\ref{eq_monodromy_vs_pi1}), the proposition also implies:

\begin{cor}\label{cor_analytic_ramification} Let $\cM\rightarrow\cM(1)^\an$ be any finite cover, let $M\rightarrow M(1)^\an$ be the corresponding map of coarse spaces, and let $\ol{M}\rightarrow\ol{M(1)^\an}$ denote the map of their smooth compactifications. Fix an isomorphism $\bZ^2\cong H_1(E_{b_0},\bZ)$ corresponding to some framing, and let $\psi : \pi_1^\top(\cM(1)^\an,E_{b_0})\rightiso\SL_2(\bZ)$ be the isomorphism induced by the framing via \eqref{eq_pi_1_isom} and \eqref{eq_MCG_on_homology}. Then letting $\SL_2(\bZ)$ act on $\cM_{E_{b_0}}$ via $\psi$, the fibers of $\ol{M}\rightarrow\ol{M(1)^\an}$ above $j = 0,1728,\infty$ are in bijection with the orbit spaces
$$(\cM_{E_{b_0}}/\{[\pm1]\})/\langle\spmatrix{1}{1}{-1}{0}\rangle, (\cM_{E_{b_0}}/\{[\pm1]\})/\langle\spmatrix{0}{1}{-1}{0}\rangle,  (\cM_{E_{b_0}}/\{[\pm1]\})/\langle\spmatrix{1}{1}{0}{1}\rangle$$
respectively. In each case, ramification indices correspond to orbit sizes of the corresponding monodromy matrix. In particular, all ramification indices of points above $j = 0$ must divide 3, all ramification indices above $j = 1728$ must divide 2, and all ramification indices above $j = \infty$ must divide the order of $\spmatrix{1}{1}{0}{1}$ acting on $\Aut(\cM_{E_{b_0}}/\{[\pm1]\})$.

\end{cor}
\begin{proof} The only thing to note is that all the indeterminacies due to the various ``up to conjugations'' that appear do not affect the final result.
\end{proof}

\subsection{Arithmetic moduli of elliptic curves with \texorpdfstring{$G$}{G}-structures}\label{ss_arithmetic_moduli}
Let $G$ be a finite group. In this section we will define the moduli stacks $\cM(G)$ which will later be used to obtain the desired three-point covers. Throughout this section we will work over the base $\bZ[1/|G|]$; we do not discuss what happens at primes dividing $|G|$.

A $G$-torsor over a scheme $X$ is a finite \'{e}tale morphism $f : Y\rightarrow X$ equipped with an $X$-linear action of $G$ on $Y$ which acts freely and transitively on geometric fibers. A morphism of $G$-torsors over $X$ is an $X$-linear, $G$-equivariant morphism. Such morphisms, if they exist, are necessarily isomorphisms. If $Y$ is connected then we will say that $f$ is a Galois cover (or a $G$-Galois cover or a $G$-cover if we wish to emphasize the Galois group).

Let $\cM(1)$ denote the moduli stack of elliptic curves. An object of $\cM(1)$ over a scheme $S$ is an elliptic curve $E/S$ (equipped with a zero section $O : S\rightarrow E$), and a morphism $E/S\rightarrow E'/S'$ is given by a cartesian diagram
\[\begin{tikzcd}
E\ar[r]\ar[d] & E'\ar[d] \\
S\ar[r] & S'
\end{tikzcd}\]
which respects the corresponding zero sections. Let $M(1)$ denote its coarse moduli scheme. It is well known that $M(1)$ is the $j$-line  $ \Spec (\bZ[1/|G|][j])$.

Let $\cT_G^{\pre} : \cM(1)\rightarrow\textbf{Sets}$ be the presheaf which associates to any object $E/S$ of $\cM(1)$ the set of isomorphism classes of $G$-torsors over the punctured elliptic curve $E^\circ := E - O$ with geometrically connected fibers over $S$. Let $\cT_G$ denote the sheafification of $\cT_G^{\pre}$ relative to the \'{e}tale topology\footnote{Here, the \'{e}tale topology on $\cM(1)$ is the inherited topology from the big \'{e}tale site $(\textbf{Sch}/\bZ[1/|G|])_{\et}$. That is to say, a family of maps in $\cM(1)$ with fixed target $E/S$ is a covering family if their images in $\textbf{Sch}/\bZ[1/|G|]$ is an \'{e}tale covering.}; a $G$-structure on $E/S$ is by definition an element of $\cT_G(E/S)$. In other words, a $G$-structure on $E/S$ is given by a collection of $G$-torsors defined \'{e}tale locally on $S$ whose common restrictions are isomorphic.   A more concrete combinatorial characterization of $G$-structures will be given in Theorem \ref{thm_basic_properties}(3).

Let $\cM(G)$ denote the category whose objects are pairs $(E/S,\alpha)$, where $\alpha\in\cT_G(E/S)$ is a $G$-structure, and morphisms are morphisms in $\cM(1)$ which respect the $G$-structure. There is a natural forgetful functor $\pi : \cM(G)\rightarrow\cM(1)$.

\begin{remark}\label{remarks} Here we record some technical remarks.
\begin{enumerate} 

\item Note that while the construction of $\cM(G)$ here differs from that given in \cite{chen18}, the resulting objects are isomorphic. To see this, one checks that by Galois theory, one obtains a natural map from the presheaf $\cT_G^{pre}$ to the presheaf of \cite[Definition 2.2.3]{chen18} which is locally an isomorphism. Thus, their sheafifications are isomorphic.
    
\item If $S = \Spec k$ where $k$ is a separably closed field, then there are no nontrivial \'{e}tale coverings of $\Spec k$, so in this case a $G$-structure on any elliptic curve $E$ over $k$ is the same as a connected $G$-torsor on $E^\circ$.

\item In the description of $G$-structures given above, note that 
there is no ``cocycle condition'' requiring that the isomorphisms are compatible.  If $G$ has trivial center, then geometrically connected $G$-torsors over $E^\circ/S$ have no nontrivial automorphisms (and hence the isomorphisms are automatically compatible), so by descent the set of $G$-structures on $E/S$ is precisely the set of isomorphism classes of geometrically connected $G$-torsors over $E^\circ/S$. However, if $G$ has nontrivial center, there can exist $G$-structures on $E/S$ which do not come from a $G$-torsor on $E^\circ/S$; this is reflected in the fact that $\cT_G$ is defined as a sheafification.  The benefit of setting things up this way is that the forgetful functor $\cM(G)\rightarrow\cM(1)$ is \emph{representable}, and in fact finite \'{e}tale.  If we had instead defined $\cT_G^{\pre}$ as a presheaf of \emph{groupoids} of $G$-torsors (which would implicitly include a cocycle condition), then $\cT_G^{\pre}$ would already be a sheaf (of groupoids), and this would lead to a Hurwitz stack in the style of \cite{br11} or \cite{ACV03}. Such stacks have a more natural ``moduli interpretation'', but the tradeoff is that its forgetful map to $\cM(1)$ is typically not representable, hence not finite.

    \item \label{analysisarithemtic} (Analysis to arithmetic) As seen in \S\ref{ss_analytic_theory}, the analytic theory of the moduli of elliptic curves yields very concrete descriptions of finite covers of the moduli stack of elliptic curves. Because of this concreteness, and to avoid introducing additional notation to deal with profinite groups, when discussing Galois theory we will sometimes prefer to state the analytic version of the corresponding algebraic statement over $\bC$. By the Riemann existence theorem for stacks (see \cite{noo05} Theorem 20.1), nothing is lost in this translation. Finally, to pass from algebraic stacks over $\bC$ to stacks over $\Qbar$, one should keep in mind the philosophy that ``a base change between algebraically closed fields of characteristic $0$ does not change the fundamental group''. For schemes, this follows from the K\"{u}nneth formula for fundamental groups (see \cite[Expos\'{e} XIII Proposition 4.6]{SGA1}) . For $\cM(1)$, one may use the fact that one may find a geometrically connected finite \'{e}tale Galois cover $U\rightarrow\cM(1)$ with $U$ a scheme. One then writes $\pi_1(\cM(1))$ as an extension of the Galois group by $\pi_1(U)$. The Galois group is not changed by algebraically closed base extension and by the Kunneth formula neither is $\pi_1(U)$. The ``short five lemma'' then yields the invariance of $\pi_1(\cM(1))$ under change of algebraically closed fields. In what follows we will use this result freely.

\end{enumerate}
\end{remark}

Next we record some of the salient properties of $\cM(G)$.

\begin{thm}\label{thm_basic_properties} Let $G$ be a finite group, and let $\pi : \cM(G)\rightarrow\cM(1)$ be the forgetful map. We will work universally over $\Spec\bZ[1/|G|]$ unless otherwise stated.
\begin{enumerate}
    \item \label{part_etale} (\'{E}taleness) The category $\cM(G)$ is a Deligne-Mumford stack and the forgetful functor $\pi : \cM(G)\rightarrow\cM(1)$ is finite \'{e}tale.
    \item \label{part_coarse} (Coarse moduli and ramification) $\cM(G)$ admits a coarse moduli scheme $M(G)$ which is a normal affine scheme finite over $M(1)\cong\Spec\bZ[1/|G|][j]$, and smooth of relative dimension $1$ over $\bZ[1/|G|]$. Moreover, $M(G)$ is \'{e}tale over the complement of the sections $j = 0$ and $j = 1728$ in $M(1)$. If either $6\mid |G|$ or $S$ is a regular Noetherian $\bZ[1/|G|]$-scheme, then $M(G)\times_{\bZ[1/|G|]} S$ is the coarse moduli scheme of $\cM(G)\times_{\bZ[1/|G|]} S$, and is normal.
    \item \label{part_combinatorial} (Combinatorial description of $G$-structures) Let $\bL$ be the set of prime divisors of $|G|$. For any profinite group $\pi$, let $\pi^\bL$ denote the maximal pro-$\bL$-quotient of $\pi$. Let $E$ be an elliptic curve over a scheme $S$. Let $x\in E^\circ$ be a geometric point, and let $s$ be its image in $S$. The sequence $E^\circ_s\hookrightarrow E\rightarrow S$ induces 
        an outer representation
    $$\rho_{E,x} : \pi_1^{\et}(S,s)\rightarrow\Out(\pi_1^\bL(E^\circ_s,x))$$
   which, by precomposition, gives a natural right action of $\pi_1(S,s)$ 
    on the set
    $$\Epi^\ext(\pi_1^\bL(E^\circ_s,x),G) := \Epi(\pi_1^\bL(E^\circ_s,x),G)/\Inn(G)$$
    of surjective morphisms $\pi_1^\bL(E^\circ_s,x)\rightarrow G$ up to conjugation in $G$. By the Galois correspondence, this action corresponds to a finite \'{e}tale morphism $F\rightarrow S$. There is a cartesian diagram:
    \[\begin{tikzcd}
    F\ar[d]\ar[r] & \cM(G)\ar[d,"\pi"]\\
    S\ar[r,"E/S"] & \cM(1).
    \end{tikzcd}\]
    In particular, since the primes dividing $|G|$ are in $\bL$, this diagram defines a bijection
    $$\cT_G(E/S)\stackrel{\sim}{\longrightarrow}\{\varphi\in\Epi^\ext(\pi_1(E^\circ_s,x),G) \;\big|\; \varphi\circ\rho_{E,x}(\sigma) = \varphi\quad\text{for all $\sigma\in\pi_1(S,s)$}\},$$
    where we recall that $\cT_G(E/S)$ is by definition the set of $G$-structures on $E/S$.
    \item \label{part_fibers} (Fibers) Let $E$ be an elliptic curve over an algebraically closed field $k$ of characteristic not dividing $|G|$, and let $x_0\in E^\circ(k)$. Let $x_E : \Spec k\rightarrow\cM(1)$ be the geometric point given by $E$ and let $\pi : \cM(G)\rightarrow\cM(1)$ be the forgetful map. Then the geometric fiber $\pi^{-1}(x_E)$ is in bijection with the set of isomorphism classes of connected $G$-torsors over $E^\circ$. By Galois theory, taking monodromy representations gives a bijection
    \begin{equation} \label{eq:epiext}
    \pi^{-1}(x_E)\stackrel{\sim}{\longrightarrow}\Epi^{\ext}(\pi_1^{\et}(E^\circ,x_0),G) := \Epi(\pi_1^{\et}(E^\circ,x_0),G)/\Inn(G),
    \end{equation}
    If $E$ is an elliptic curve over $\bC$ and $x_0\in E^\circ(\bC)$, by Galois theory taking monodromy representations gives a bijection
    $$\pi^{-1}(x_E)\stackrel{\sim}{\longrightarrow}\Epi^{\ext}(\pi_1^{\top}(E^\circ(\bC),x_0),G).$$
    In particular, if $G$ is not generated by two elements, then $\cM(G)$ is the empty stack.
    \item \label{part_monodromy} (Monodromy) Let $E$ be an elliptic curve over $\bC$, $x_0\in E^\circ(\bC)$, $\Pi := \pi_1^{top}(E^\circ(\bC),x_0)$, and let $x_E : \Spec\bC\rightarrow\cM(1)$ be the geometric point corresponding to $E$. Then $\Pi$ is a free group of rank $2$, and the canonical map $\Pi\rightarrow H_1(E,\bZ)$ induces an isomorphism $\Pi/[\Pi,\Pi]\cong H_1(E,\bZ)$. Let $\Gamma_E$ denote the orientation-preserving mapping class group of $E^\circ(\bC)$, and let $\Out^+(\Pi)$ be the preimage of $\SL(H_1(E,\bZ))$ under the canonical map
    $$\alpha : \Out(\Pi)\rightarrow\GL(H_1(E,\bZ)).$$
    The outer action of $\Gamma_E$ on $\Pi$ is faithful and identifies $\Gamma_E$ with $\Out^+(\Pi)$. As $\alpha$ is an isomorphism, it induces isomorphisms $\Gamma_E \stackrel{\sim}{\longrightarrow} \Out^+(\Pi) \stackrel{\sim}{\longrightarrow} \SL(H_1(E,\bZ))$.  The analytic theory identifies $\Gamma_E$ with the topological fundamental group of the analytic moduli stack of elliptic curves (with base point $E$), from which we obtain a canonical isomorphism $\Gamma_E^\wedge\stackrel{\sim}{\longrightarrow}\pi_1^{\et}(\cM(1)_{\Qbar},x_E)$ (where ${}^\wedge$ denotes profinite completion). In particular, there is a canonical injective map $\Gamma_E\hookrightarrow\pi_1^{\et}(\cM(1)_{\Qbar},x_E)$ with dense image. With respect to this map, the bijection
    $$\pi^{-1}(x_E)\stackrel{\sim}{\longrightarrow}\Epi^{\ext}(\Pi,G)$$
    of (4) is $\Gamma_E$-equivariant.
    To summarize, we have canonical isomorphisms
    $$\pi_1^{\top}(\cM(1)^\an,x_E) \cong\Gamma_E\cong \Out^+(\Pi)\cong \SL(H_1(E,\bZ))$$
    and
    $$\pi_1^{\et}(\cM(1)_{\Qbar},x_E) \cong \pi_1^{\top}(\cM(1)^\an,x_E)^\wedge.$$
    \item \label{part_functoriality} (Functoriality) Let $\cC$ denote the category whose objects are finite groups generated by two elements, and whose morphisms are surjective homomorphisms. If $f : G_1\twoheadrightarrow G_2$ is a morphism in $\cC$, then we obtain a map $\cT_f^\pre : \cT_{G_1}^\pre\rightarrow\cT_{G_2}^\pre$ defined by sending the $G_1$-torsor $X^\circ\rightarrow E^\circ$ to the $G_2$-torsor $X^\circ/\ker(f)\rightarrow E^\circ$ where the $G_2$-action is given by the canonical isomorphism $G_2\cong G_1/\ker(f)$. This induces a map $\cT_f : \cT_{G_1}\rightarrow\cT_{G_2}$, whence a map
$$\cM(f) : \cM(G_1)\rightarrow\cM(G_2).$$
The maps $\cM(f)$ make the rule sending $G\in\cC$ to the map $\cM(G)\rightarrow\cM(1)$ into an epimorphism-preserving functor from $\cC$ to the category\footnote{Here we mean the (1-)category associated to the (2,1)-category, see \cite[\S4]{noo04}.} of stacks finite \'{e}tale over $\cM(1)$. Let $E,\Pi,\Gamma_E$ be as in (4); then in terms of the Galois correspondence for covers of $\cM(1)$, given a surjection $f : G_1\rightarrow G_2$, the induced map $\cM(G_1)\rightarrow\cM(G_2)$ (of $\bZ[1/|G|]$-stacks) is given by the $\Gamma_E$-equivariant map of fibers
    $$f_* : \Epi^{\ext}(\Pi,G_1)\rightarrow\Epi^{\ext}(\Pi,G_2)$$ 
    obtained by post-composing every surjection with $f$. 
    
    \item \label{part_cofinality} (Cofinality --- Asada's theorem) For any stack $\cM$ finite \'{e}tale over $\cM(1)_{\Qbar}$, there is a finite group $G$ such that $\cM$ is dominated by some connected component of $\cM(G)_{\Qbar}$.
\end{enumerate}
\end{thm}

\begin{proof} Part \eqref{part_etale} is \cite[Proposition 3.1.4]{chen18} . Everything in \eqref{part_coarse} except for normality and \'{e}taleness is  \cite[Proposition 3.3.4]{chen18}.  The normality of $M(G)$ follows from the fact that $M(G)$ is the quotient of a smooth representable moduli problem by a finite group  (see \cite[\S3.3.3]{chen18}).  Let $U\subset M(1)$ be the complement of $j = 0,1728$; to see $M(G)$ is \'{e}tale over the complement of $U$, consider a finite \'{e}tale surjection $\cM\rightarrow\cM(G)$ with $\cM$ representable (we may for example take $\cM$ to be the fiber product over $\cM(1)$ of $\cM(G)$ with the moduli stack of elliptic curves with full level $p^2$ structure for some $p\mid |G|$).
By \cite[Corollary 8.4.5]{km85}, the map $\cM_U\rightarrow U$ is \'{e}tale. Since $\cM,M(G)$ and $M(1)$ are regular, by purity of the branch locus we are reduced to checking \'{e}taleness of extensions of complete discrete valuation rings. Since ramification indices of such extensions are multiplicative, \'{e}taleness of the composite $\cM_U\rightarrow\cM(G)_U\rightarrow U$ implies \'{e}taleness of $\cM(G)_U\rightarrow U$. Part \eqref{part_combinatorial} is \cite[Proposition 2.2.6(1,2)]{chen18}. Part \eqref{part_fibers} follows from part \eqref{part_combinatorial}, setting $S = \Spec k$.
For \eqref{part_monodromy}, a theorem of Nielsen  gives that  $\alpha$ is an isomorphism \cite[Theorem 3.1]{oz81}, and the isomorphism $\Gamma_E^\wedge\cong\pi_1(\cM(1)_{\Qbar},E)$ follows from the Riemann existence theorem for stacks \cite[Theorem 20.4]{noo05}. The rest of \eqref{part_monodromy} is simply an unfolding of definitions, part \eqref{part_functoriality} is \cite[Proposition 3.2.8]{chen18}, and part \eqref{part_cofinality} is Asada's theorem (see \cite[Theorem 3.4.2]{chen18}, \cite{BER11}, \cite[\S7]{asa01}). 
\end{proof}

Using Belyi's theorem together with the fact that $\cM(1)$ admits a finite \'{e}tale cover by $\bP^1 -  \{0,1,\infty\}$, it follows from Theorem~\ref{thm_basic_properties}\eqref{part_cofinality} that every algebraic curve over $\Qbar$ is the quotient of a component of $M(G)_{\Qbar}$. In fact, we may say even more:

\begin{thm}\label{thm_faithful} In the category of finite \'{e}tale covers of $\cM(1)_\Qbar$, let $\cC$ denote the full subcategory generated by the connected components of $\cM(G)_\Qbar$. Let $C$ denote the set of isomorphism classes in $\cC$. Since $\cM(G)$ is defined over $\bQ$, $\Gal(\Qbar/\bQ)$ acts on $\pi_0(\cM(G)_{\Qbar})$, and hence on $C$. This action is \emph{faithful}.

Moreover, let $X$ be a smooth projective curve over $\Qbar$. Then there is a finite group $G$ and a component $\cM\subset\cM(G)_{\Qbar}$ such that: 
\begin{enumerate}
    \item There is a unique minimal subfield $K\subset\Qbar$ such that $\cM$ is the base change of a geometrically connected component $\cM_K$ of $\cM(G)_K$.
    
    \item  Let $M_K$ be the coarse scheme of $\cM_K$, and let $\ol{M_K}$ be its smooth compactification. Then there is a finite group $H\subset\Aut_K(\ol{M_K})$ such that $\ol{M_K}/H\otimes_K \Qbar\cong X$. 
\end{enumerate}
\end{thm}

In short, every curve over $\Qbar$ admits a model as a quotient of a component of $M(G)$ \emph{over the field of definition of that component}.

\begin{proof} 
The faithfulness of the action follows from the rest: if the action had a non-trivial kernel $\Gal(\ol{\bQ}/L)$, then every component of $\pi_0(\cM(G)_{\Qbar})$ would arise from base change of a component over $L$.  But taking $X$ to be an elliptic curve over $\Qbar$ with $j$-invariant outside $L$, $X$ cannot admit a model as quotient of a component $M(G)$ over the field of definition of that component.

To prove (1) and (2), the idea is to combine Belyi's theorem with an explicit version of Asada's theorem due to Ellenberg--Mcreynolds \cite{EM12}.  Let $X$ be a smooth projective curve over $\Qbar$. Fix an elliptic curve $E$ over $\bC$ and a basis for $\Pi := \pi_1(E^\circ(\bC))$, which we use to identify $H_1(E(\bC),\bZ)\cong\bZ^2$ and  $\pi_1(\cM(1)_{\Qbar},E)\cong\widehat{\SL_2(\bZ)}$. Let $\Gamma(2)'\le\SL_2(\bZ)$ be the subgroup generated by $\langle\spmatrix{1}{2}{0}{1},\spmatrix{1}{0}{2}{1}\rangle$. Then $\Gamma(2)'$ is a free group which has index 2 inside $\Gamma(2) := \Ker(\SL_2(\bZ)\rightarrow\SL_2(\bZ/2\bZ))$, and we have $\Gamma(2) = \langle\Gamma(2)',-I\rangle$ where $-I := \spmatrix{-1}{0}{0}{-1}$. 
Let $V$ be the finite \'{e}tale covering of $\cM(1)_{\Qbar}$ corresponding to (the closure in $\widehat{\SL_2(\bZ)}$ of) $\Gamma(2)'$. Since $\Gamma(2)'$ is torsion-free, $V$ is a scheme \cite[Theorem 3.5.3]{chen18}, and by Riemann--Hurwitz applied to the ramification description given in Corollary~\ref{cor_analytic_ramification}, $V\cong \bP^1_{\Qbar} - \{0,1,\infty\}$.  By Belyi's theorem, there is an open $U\subset X$ together with a finite \'{e}tale map $U\rightarrow V$. Thus, $U$ is finite \'{e}tale over $\cM(1)_{\Qbar}$, and hence corresponds to a finite index subgroup $\Gamma_U\le\Gamma(2)'$. Let $\Gamma_U' := \langle\Gamma_U,-I\rangle$. Let $\cU'$ denote the cover of $\cM(1)_{\Qbar}$ corresponding to (the closure of) $\Gamma_U'$. Then by Proposition~\ref{prop_coarse_fibers}(\ref{prop_coarse_fibers1}) the natural map $U\rightarrow \cU'$ induces an isomorphism on coarse schemes. By \cite[Theorem 1.2]{EM12},  $\Gamma_U'$ is the \emph{Veech group of an origami}. From the description of origami Veech groups given in \cite{Sch04} (specifically Corollary 2.7 and Lemma 2.8(2)), this means that there is an integer $d$ and a permutation group $G\le S_d$ such that $\Gamma_U'$ is the stabilizer inside $\SL_2(\bZ)\cong\Out^+(\Pi)$ of an element of
$$\Epi^\ext(\Pi,G)/N_{S_d}(G)$$
where the normalizer $N_{S_d}(G)$ of $G$ in $S_d$ acts on $G$ by conjugation inside $S_d$. This implies that $\cU'$ is the quotient of a component $\cM\subset\cM(G)_{\Qbar}$ by the action of $H := N_{S_d}(G)/C_{S_d}(G)$, where $C_{S_d}(G)$ is the centralizer of $G$ in $S_d$.

Now we address (1). First we define the $\Gal(\Qbar/\bQ)$-action on $C$. Fix a base point $x : \Spec\Qbar\rightarrow\cM(1)_\Qbar$, there is a homotopy exact sequence
\begin{equation}\label{eq_homotopy_exact_sequence}
    1\rightarrow\pi_1(\cM(1)_\Qbar,x)\rightarrow\pi_1(\cM(1)_\bQ,x)\rightarrow\Gal(\Qbar/\bQ)\rightarrow 1.
\end{equation}
Thus the action of $\pi_1(\cM(1)_\Qbar)$ on the geometric fiber $\pi^{-1}(x)$ of $\pi : \cM(G)\rightarrow\cM(1)$ can be viewed as the restriction of the action of $\pi_1(\cM(1)_\bQ)$. Galois theory gives a bijection between $\pi_0(\cM(G)_\Qbar)$ and the set of orbits of the former action, and via the exact sequence (\ref{eq_homotopy_exact_sequence}), we obtain an action of $\Gal(\Qbar/\bQ)$ on the orbits and hence on $\pi_0(\cM(G)_\Qbar)$. For any other base point $x'$, any isomorphism of the fiber functors associated to $x$ and $x'$ differ by conjugation in $\pi_1(\cM(1)_\Qbar,x)$, so this action of $\Gal(\Qbar/\bQ)$ on $\pi_0(\cM(G)_\Qbar)$ is independent of the choice of base point.

We say that a subfield $L\subset\Qbar$ is a field of definition of $\cM/\cM(1)_\Qbar$ if $\cM\rightarrow\cM(1)_\Qbar$ is the base change of a cover $\cM_L\rightarrow\cM(1)_L$. We claim that $L$ is a field of definition of $\cM/\cM(1)_\Qbar$ if and only if $\Gal(\Qbar/L)$ fixes $\cM\in\pi_0(\cM(G)_\Qbar)$. Indeed, let $O\subset \pi^{-1}(x)$ be the $\pi_1(\cM(1)_\Qbar)$-orbit corresponding to $\cM$. If $L$ is a field of definition, then comparing the exact sequence \eqref{eq_homotopy_exact_sequence} with the analogous sequence for $\cM(1)_L$, we find that the action of $\pi_1(\cM(1)_\bQ)$ on $\pi^{-1}(x)$ restricts to a $\pi_1(\cM(1)_L)$-action which preserves $O$.  In other words % the definition of the $\Gal(\Qbar/\bQ)$-action is to say that 
$\Gal(\Qbar/L)$ fixes $\cM$. Conversely, if $\Gal(\Qbar/L)$ fixes $\cM$, then we obtain an action of $\pi_1(\cM(1)_L)$ on $O$ which by Galois theory corresponds to a finite \'{e}tale cover of $\cM(1)_L$ whose base change to $\cM(1)_\Qbar$ is $\cM$, so $L$ is a field of definition. This establishes our claim, and also shows that the intersection $K$ of all fields of definition (which is also the fixed field of the stabilizer of $\cM$ under the action of $\Gal(\Qbar/\bQ)$) is the unique minimal field of definition. This proves (1).

For (2), the action of $H$ is defined on $\cM(G)_\bQ$, and hence we may form the quotient $\cM_K/H$. It follows from Proposition~\ref{prop_coarse_fibers}(\ref{prop_coarse_fibers1}) that the coarse scheme of $\cM_K/H$ is precisely $M_K/H$, where $M_K$ is the coarse scheme of $\cM_K$. Taking compactifications, we find that $\ol{M_K}/H$ is a $K$-model of $X$.
\end{proof}

\subsection{Obtaining tamely ramified 3-point covers} \label{ss_obtainingtame}
Given any connected component $\cM$ of $\cM(G)_{\Qbar}$, by Theorem \ref{thm_basic_properties}\eqref{part_coarse}, we find that the map from its coarse moduli scheme $M$ to the $j$-line
$$\pi : M\rightarrow M(1)_{\Qbar} \cong \Spec\Qbar[j]$$
describes a three-point cover with good reduction at all $p\nmid |G|$.  The purpose of this section is to show that reduction of this cover to any $p\nmid |G|$ is tamely ramified, i.e. that all ramification indices of $\pi$ are coprime to $6|G|$. It is an often overlooked consequence of Abhyankar's lemma that when the branch divisor has normal crossings and is mixed characteristic, tameness of the restriction to special fibers is automatic, without having to know anything about the ramification indices of the generic fiber. Nonetheless, after explaining tameness, we also describe the ramification indices in Proposition \ref{prop_ramification_indices}.

To obtain tameness, we will need the following well-known lemma for which we do not know a reference.  (A slightly more general version is stated without proof in the paragraph before \cite[\S4.2.3]{ACV03}.)   
    \begin{lemma}\label{lemma_smoothness_of_normalization} Let $S$ be a regular Noetherian scheme. Let $f : X\rightarrow S$ be smooth proper morphism. For a normal crossings divisor $D\subset X$ that is smooth over $S$, let $U := X - D$.  Let  $\pi : V\rightarrow U$ be finite \'{e}tale, and let $Y$ be the normalization\footnote{This is equivalent to the normalization of $X$ inside the function field of $V$, which is a finite extension of that of $X$. See \cite[0BAK]{stacks} for the definition of relative normalization.} of $X$ inside $V$. Suppose for every maximal point\footnote{a maximal point is a generic point of an irreducible component} $\eta\in D$, the integral closure of $\cO_{X,\eta}$ inside the function field of $V$ is tamely ramified over $\Spec\cO_{X,\eta}$. 
    \begin{enumerate}
    \item \label{lemma_smoothness_of_normalization1} The natural diagram
\[\begin{tikzcd}
V\ar[r,hookrightarrow]\ar[d,"\pi"] & Y\ar[d,"\ol{\pi}"] \\
U\ar[r,hookrightarrow] & X
\end{tikzcd}\]
is cartesian, $\ol{\pi}$ is finite flat, $Y$ is smooth over $S$, and for every $s\in S$ the restriction $\ol{\pi}_{X_s}$ is tamely ramified over $D_s$.

\item \label{lemma_smoothness_of_normalization2} Let $\tilde{D} \subset Y$ be the reduced closed subscheme corresponding to $\ol{\pi}^{-1}(D)$, then $\tilde{D}$ is finite \'{e}tale over $D$. For any irreducible component $Z\subset\tilde{D}$ and any $z,z'\in Z$, the ramification indices of $\ol{\pi}|_{Y_{f(\ol{\pi}(z))}}$ and $\ol{\pi}|_{Y_{f(\ol{\pi}(z'))}}$ at $z,z'$ are the same. 

\item \label{lemma_smoothness_of_normalization3} The function 
 $n_{Y/S} : S\rightarrow\bZ$ counting the number of irreducible components in the geometric fibers of $Y/S$ is locally constant on $S$.
 \end{enumerate}
\end{lemma}

Note that if $S$ has generic characteristic 0, then the tameness condition (and hence also the consequent tameness in the special fiber) is automatic! 

\begin{proof} Since $U\subset X$ is the complement of a normal crossings divisor, $V\rightarrow X$ is affine \cite[07ZU]{stacks}, and hence the diagram is cartesian by Zariski's Main Theorem \cite[03GT(1)]{stacks} . Finiteness of $\ol{\pi}$ follows from \cite[Proposition 5.17]{am69}. Next, for a geometric point $\ol{x}$ lying over a point $x\in D$, let $A := \cO_{X,x}^\sh$ be the strict henselization of the local ring at $x$, and let $X_1 := \Spec A$. On $X_1$, $D$ is cut out by some $g\in A$.  By Abhyankar's lemma (see \cite[0EYG]{stacks} and also \cite[Corollary 2.3.4]{gm71}), the restriction $Y_{X_1}\rightarrow X_1$ is the disjoint union of covers of the form
$$Y' := \Spec A[T]/(T^e-g)\rightarrow X_1 = \Spec A$$
where $e$ is invertible on $X_1$. 
This immediately yields the flatness of $Y/X$, \'{e}taleness of $\tilde{D}/D$, the local constancy of ramification indices on $D$, and the tameness of the restrictions of $\ol{\pi}$ to fibers.
This proves all of (\ref{lemma_smoothness_of_normalization1}) and (\ref{lemma_smoothness_of_normalization2}) except the smoothness of $Y$ over $S$.  

Let $s := f(x)$, and let $k$ be a finite extension of the residue field $k(s)$ of $s$. Since $X/S$ is smooth, $A\otimes_S k$ is also a regular local ring. %\cite[05WS]{stacks}
Let $\mf{m}$ be its maximal ideal, and let $\ol{g}$ be the image of $g$ in $A\otimes_S k$. Since $D$ is $S$-smooth, the quotient $(A\otimes_S k)/(\ol{g})$ is also regular, so $\ol{g}$ is a member of a regular sequence of parameters for $A\otimes_S k$ \cite[00NR]{stacks}. It follows that the rings $(A\otimes_S k)[T]/(T^e-\ol{g})$ are regular where the regular parameter $\ol{g}$ is replaced by $T$ \cite[Lemma 1.8.6]{gm71}. Since $k$ was arbitrary, this shows that the normalization $Y$ has geometrically regular fibers. Since $Y\rightarrow S$ is flat and finitely presented, we find that $Y/S$ is also smooth \cite[01V8]{stacks}.

Finally, the local constancy of $n_{Y/S}$ follows from Stein factorization \cite[0E0N]{stacks}, noting that $Y/S$ is smooth, so every connected component is irreducible.
\end{proof}

Using the lemma, we obtain the following procedure for producing tamely ramified 3-point covers using connected components of $\cM(G)_{\overline{\Q}_p}$.  Given a finite \'{e}tale morphism $\pi : Y \to X$, we may form the Galois closure: this is a finite \'{e}tale morphism $\pi' : Y' \to X$ which is Galois \cite[03SF]{stacks}, factors through $\pi$, and which is universal with respect to this property. This generalizes the familiar operation of taking the Galois closure of a separable extension of fields.

\begin{prop}\label{prop_tame} Let $n\ge 1$ be an integer divisible by 6. Let $\cM'$ be a stack finite \'{e}tale over $\cM(1)_{\bZ[1/n]}$.
\begin{enumerate}
    \item For each connected component $\cM$  of $\cM'_{\Qbar}$, there is a connected finite \'{e}tale $\bZ[1/n]$-algebra $A$ such that $\cM'_A$ admits an $A$-section containing the image of $\cM$. For any such $A$, $\cM$ extends to a connected component $\cM_A$ of $\cM'_A$ with geometrically connected fibers. 
    \item Let $M_A$ be the coarse moduli scheme of $\cM_A$, then $M_A$ comes with a map $\pi : M_A\rightarrow M(1)_A\cong\Spec A[j]$, unramified away from $j = 0,1728$, such that the restrictions of $\pi$ to the special fibers of $M(1)_A$ are tamely ramified.
    \item Let $K := \Frac(A)$, let $M(1)_A^\circ := M(1)_A - \{j=0,1728\}$, and let $M_A^\circ = \pi^{-1}(M(1)^\circ_A)$.  For some finite extension $L/K$ the Galois closure $N_L$ of $M^\circ_L\rightarrow M(1)^\circ_L$ is geometrically connected. Let $B$ be the integral closure of $A$ in $L$, and let $N_B$ be the normalization of $M^\circ_B$ inside $N_L$. Then $N_B\rightarrow M(1)^\circ_B$ is the Galois closure of $M^\circ_B\rightarrow M(1)^\circ_B$ and has geometrically connected fibers. In particular, the Galois groups of the special fibers of $N_B$ are isomorphic to the Galois group of $N_L\rightarrow M(1)^\circ_L$.
    \item Let $H$ be the Galois group of $N_L\rightarrow M(1)^\circ_L$ and let $d := \deg(\pi)$. If $H \cong S_d$, then we may take $L = K$ and $B = A$. If $H \cong A_d$, then we may take $L$ to be an at most quadratic extension of $K$. 
\end{enumerate} 
\end{prop}

\begin{proof} Let $D$ be a $\bZ[1/n]$-section of $\cM(1)_{\bZ[1/n]}$. Since the map $\cM'_{\bZ[1/n]}\rightarrow\cM(1)_{\bZ[1/n]}$ is finite \'{e}tale, there is a connected finite \'{e}tale $\bZ[1/n]$-algebra $A$ such that $\cM'_A\times_{\cM(1)_A}D_A$ is totally split over $A$. Let $\cM_A$ be the connected component of $\cM'_A$ containing the image of $\cM$; then $\cM_A$ is connected and being \'{e}tale over $\cM(1)_A$, its generic fiber is also connected. Since it admits an $A$-section, its generic fiber is geometrically connected. To see that the closed fibers are geometrically connected, by Theorem \ref{thm_basic_properties}\eqref{part_coarse}, it suffices to show this for coarse schemes.
% Let $M(1)_A^\circ := M(1)_A - \{j=0,1728\}$ and $M_A^\circ = \pi^{-1}(M(1)_A)$. 
Let $\ol{M(1)}_A := \bP^1_A$ be the compactification of $M(1)_A$. Because $6\mid n$, $M(1)_A^\circ\subset\ol{M(1)}_A$ is the complement of a smooth divisor, and hence applying Lemma \ref{lemma_smoothness_of_normalization} to the map $M_A^\circ\rightarrow M(1)_A^\circ\subset \ol{M(1)}_A$ we find that the closed fibers are geometrically connected and tamely ramified over the corresponding closed fibers of $\ol{M(1)}_A$. This establishes (1) and (2).

Let $N_B$ be the normalization of $M(1)^\circ_B$ inside $N_L$. Let $k(N_L)$ be the function field of $N_L$ and $P$ be the Galois closure of $M^\circ_B/M(1)^\circ_B$.  Since $N_L$ is a Galois closure of the generic fiber, $k(P)$ is a finite extension of $k(N_L)$. Thus, since every codimension-$1$ point $x\in M(1)^\circ_B$ is unramified in $k(P)$, it is also unramified in $k(N_L)$. By purity, $N_B$ is \'{e}tale over $M_B^\circ$, which shows that $N_B$ is a Galois closure of $M^\circ_B\rightarrow M(1)^\circ_B$, so $N_B\cong P$. We are again in the situation of Lemma \ref{lemma_smoothness_of_normalization}, from which we find that the special fiber of $N_B$ is geometrically connected (and tamely ramified). This establishes (3).

Finally, we address (4). Let $x : \Spec K\rightarrow M(1)^\circ_K$ be a $K$-point, and let $\ol{x}$ be the corresponding geometric point with values in an algebraic closure $\ol{K}$ of $K$. Let $\Pi := \pi_1(M(1)^\circ_{K},\ol{x})$ and $\ol{\Pi} := \pi_1(M(1)^\circ_{\ol{K}},\ol{x})$.  The maps $M(1)^\circ_{\ol{K}}\rightarrow M(1)^\circ_K\rightarrow\Spec K$ induce an exact sequence (see \cite[Proposition 5.6.1]{sza09})
$$1\rightarrow\ol{\Pi}\rightarrow \Pi\rightarrow\Gal(\ol{K}/K)\rightarrow 1$$
which is split by $x$. Via this splitting, we obtain an isomorphism
\begin{equation}\label{eq_sdp}
    \Pi \cong \ol{\Pi}\rtimes\Gal(\ol{K}/K).
\end{equation}
Let $M_{\ol{x}}^\circ$ denote the geometric fiber of $M_K^\circ$ over $\ol{x}$, which has cardinality $\deg(\pi)$. Galois theory identifies the isomorphism class of the covering $M^\circ_K\rightarrow M(1)^\circ_K$ with the equivalence class of the homomorphism
$$\rho : \Pi\rightarrow\Sym(M_{\ol{x}}^\circ)$$
up to inner automorphisms of $\Sym(M_{\ol{x}}^\circ)$.  
(Here $\Sym(M_{\ol{x}}^\circ)$ is the symmetric group on the fiber $M_{\ol{x}}^\circ$.)  The image of $\ol{\rho} := \rho|_{\ol{\Pi}}$ is the geometric monodromy group of $\pi$, and is isomorphic to $H$. The kernel of $\rho$ corresponds to the Galois closure of $M^\circ_K$. If $\ol{\rho}$ surjects onto $\Sym(M_{\ol{x}}^\circ)$, then $[\Pi : \Ker(\rho)] = [\ol{\Pi} : \Ker(\ol{\rho})]$, and hence the Galois closure of $M^\circ_K$ is already geometrically connected, so we may take $L = K$. If the image of $\ol{\rho}$ is the alternating group, then using the decomposition (\ref{eq_sdp})
, we may take $L$ to be the at-most-quadratic extension corresponding to the (at-most-index-$2$) subgroup 
\[\rho^{-1}(\Alt(M^\circ_{\ol{x}}))\cap\Gal(\ol{K}/K)\subset\Gal(\ol{K}/K). \qedhere\]
\end{proof}

\begin{remark} \label{rmk:frobeniusoddness}
%\jeremy{Added this remark}
When $H \cong A_d$, the proof of the Proposition \ref{prop_tame}(4) gives a necessary and sufficient condition on when we may take $L=K$  in terms of the $\Gal(\overline{K}/K)$-action on the fiber $M_{\overline{x}}^\circ$: we must have $\rho^{-1}(\Alt(M^\circ_{\ol{x}}))\cap\Gal(\ol{K}/K) = \Gal(\ol{K}/K)$.
Thus when $K$ is a finite field, we may take $L =K$ when the action of the Frobenius on the fiber is even, while we must use a quadratic extension when the action of the Frobenius on the fiber is odd.  
\end{remark}

For a general finite \'{e}tale map of stacks $\cM\rightarrow\cM(1)_{\Qbar}$, the ramification of the map on coarse schemes is described by Corollary \ref{cor_analytic_ramification}. Here we spell out what it means combinatorially when $\cM = \cM(G)_{\Qbar}$ for a finite group $G$.

Let $E$ be an elliptic curve over $\bC$ with $j$-invariant not $0$ or $1728$. Fix an embedding $i : \Qbar\hookrightarrow\bC$. Using $i$, we will view $E$ as a geometric point of both $\cM(1)_{\Qbar}$ and $M(1)_{\Qbar}$. Let $x_0\in E^\circ(\bC)$ and $\Pi := \pi_1^\top(E^\circ(\bC),x_0)$. Let $x,y$ be a basis for $\Pi$ with intersection number $+1$. Consider the four automorphisms of $\Pi$:
$$\gamma_0 : (x,y)\mapsto (xy^{-1},x) \quad \gamma_{1728} : (x,y)\mapsto (y^{-1},x),\quad \gamma_\infty : (x,y)\mapsto (x,xy) \quad \gamma_{-I} : (x,y)\mapsto (x^{-1},y^{-1}).$$
Since $\gamma_{-I}$ is central in $\Out(\Pi)\cong\GL_2(\bZ)$, the natural action of $\Aut(\Pi)$ on $\Epi(\Pi,G)$ descends to an action of $\Aut(\Pi)$ on the set 
$$\tilde{F}(G) := \Epi^\ext(\Pi,G)/\langle \gamma_{-I}\rangle$$
which visibly factors through $\Out(\Pi)/\langle\gamma_{-I}\rangle\cong\PGL_2(\bZ)$. Let $\ol{M(G)}_{\Qbar}$ and $\ol{M(1)}_{\Qbar}$ be the smooth compactifications of $M(G)_{\Qbar}$ and $M(1)_{\Qbar}$ respectively. Using the bijection of Theorem \ref{thm_basic_properties}(4) and taking coarse schemes, we may identify the fiber of $\ol{M(G)}_{\Qbar}\rightarrow \ol{M(1)}_{\Qbar}$ above $E$ with $\tilde{F}(G)$, and the fibers above $j = 0,1728,\infty$ with the quotient sets
$$\tilde{F}(G)/\langle\gamma_0\rangle,\quad\tilde{F}(G)/\langle\gamma_{1728}\rangle,\quad\tilde{F}(G)/\langle\gamma_\infty\rangle$$
respectively, such that the ramification indices at each point correspond to the size of the its orbit under $\gamma_0,\gamma_{1728}$, or $\gamma_\infty$ respectively in $\tilde{F}(G)$.
Let $e(G)$ denote the least positive integer $n$ satisfying $g^n = 1$ for all $g\in G$. 

\begin{prop}\label{prop_ramification_indices} With the notation of the preceding paragraphs, the ramification indices of $\ol{M(G)}_{\Qbar}\rightarrow \ol{M(1)}_{\Qbar}$  all divide $6\cdot e(G)$. Specifically, the ramification indices above $j = 0$  all divide 3, and the ramification indices above $j = 1728$  all divide 2, and the ramification indices of $\pi$ above $j = \infty$  all divide $e(G)$.
\end{prop}

\begin{proof} First note that relative to the isomorphism $\Out(\Pi)\cong\GL_2(\bZ)$ induced by the basis $x,y$, the automorphisms $\gamma_0,\gamma_{1728},\gamma_\infty,\gamma_{-I}$ viewed inside $\Out(\Pi)$ correspond to the matrices 
$$\spmatrix{1}{1}{-1}{0},\spmatrix{0}{1}{-1}{0},\spmatrix{1}{1}{0}{1},\spmatrix{-1}{0}{0}{-1}.$$
Also note that the action of $\gamma_{-I}$ (resp. $\spmatrix{-1}{0}{0}{-1}$) on $\Pi$ (resp. $H_1(E(\bC),\bZ)$) is induced by the automorphism $[-1]$ of $E$. If every instance of $\Qbar$ in the statement is replaced by $\bC$, then using standard GAGA arguments (see \cite[Theorem 20.4]{noo05}), the descriptions of the fibers follow from Theorem \ref{thm_basic_properties}(4) together with Proposition \ref{prop_coarse_fibers} and Corollary \ref{cor_analytic_ramification}. Since $\gamma_0$ (resp. $\gamma_{1728}$) has order 3 (resp. 2) in $\Out(\Pi)/\langle\gamma_{-I}\rangle$, the ramification indices above $j = 0$ (resp. $j = 1728$) must divide 3 (resp. 2). Finally, note that $\gamma_\infty$ must act on $\tilde{F}(G)$ with order dividing $e(G)$.

To pass from $\bC$ to $\Qbar$ we use Remark~\ref{remarks}(\ref{analysisarithemtic}) plus ``coarse base change'' in characteristic $0$ \cite[Proposition 3.3.4]{chen18}. 
\end{proof}

\subsection{Absolute \texorpdfstring{$G$}{G}-structures}\label{ss_absolute_G_structures}
There is a natural action of $\Aut(G)$ on $\cM(G)$ which factors through a free action of $\Out(G)$.
 Thus the quotient $\cM(G)/\Out(G)$ is also finite \'{e}tale over $\cM(1)$ and we will denote it by\footnote{Here, the superscript $\abs$ is short for ``absolute'', and the distinction between $\cM(G)$ and $\cM(G)^\abs$ is  analogous to Fried's distinction between ``inner'' and ``absolute'' Hurwitz stacks (see \cite{ber13} Remark 4.56).}
$$\cM(G)^\abs := \cM(G)/\Out(G).$$
For an elliptic curve $E$ over a $\bZ[1/|G|]$-scheme $S$ corresponding to a map $E : S\rightarrow\cM(1)$, a section of the finite \'{e}tale $S$-scheme $S\times_{\cM(1)}\cM(G)^\abs$ is called an \emph{absolute $G$-structure} on $E$;  $\cM(G)^\abs$ is then the moduli stack of elliptic curves with absolute $G$-structures. In the notation of Theorem \ref{thm_basic_properties}\eqref{part_combinatorial}, the set of absolute $G$-structures is in bijection with the set:
$$\{\varphi\in\Epi^\ext(\pi_1^\bL(E^\circ_s,x),G)/\Out(G)\;\big|\;\varphi\circ\rho_{E,x}(\sigma) = \varphi\quad\text{for all $\sigma\in\pi_1(S,s)$}\}$$
As before, $\varphi$ is an equivalence class of surjections to $G$ (in this case, an $\Aut(G)$-orbit), and an absolute $G$-structure is given by an equivalence class which is stabilized by the action of $\pi_1(S,s)$.

%\will{Replaced the notation ``$(G)$-Galois'' with ``mere $G$-cover''} 
We may also give a geometric description of absolute $G$-structures. Given a scheme $S$ and an elliptic curve $E/S$ with zero section $O$, let $E^\circ := E - O$. A \emph{mere $G$-cover}\footnote{We use the word ``mere'' in a similar way as \cite{DD97}.} of $E^\circ$ is a finite \'{e}tale map $\pi : X\rightarrow E^\circ$ whose geometric fibers over $S$ are connected, and such that $\Aut(\pi)$ is isomorphic to $G$. A morphism of mere $G$-covers of $E^\circ$ is a morphism in the category $\textbf{Sch}/E^\circ$. In particular we do not specify a $G$-action on $X$ and hence we do not require that morphisms be $G$-equivariant. Let $\cT_G^{\abs,\pre} : \cM(1)\rightarrow\textbf{Sets}$ be the presheaf which associates to an elliptic curve $E/S$ the set of isomorphism classes of mere $G$-covers of $E^\circ$. Then $\cM(G)^\abs$ is isomorphic to the stack obtained from the sheafification of $\cT_G^{\abs,\pre}$ in the \'{e}tale topology. Thus, informally speaking an absolute $G$-structure on $E/S$ is given by the data of mere $G$-covers defined \'{e}tale locally on $S$ whose common restrictions are isomorphic, but with no requirement that the isomorphisms satisfy the cocycle condition.

\begin{remark} \label{remark:absolute_G_structure}
Note that if $\pi : X\rightarrow E^\circ$ is finite \'{e}tale, geometrically connected over $S$, and such that there is a surjective \'{e}tale map $S'\rightarrow S$ such that the pullback $X_{S'}\rightarrow E^\circ_{S'}$ is Galois, then $\pi$ defines an absolute $G$-structure on $E/S$ even though $\pi$ itself may not have enough automorphisms to be Galois.
\end{remark}

We now make this concrete in the case that $G = \SL_2(\bF_\ell)$ and $G = \PSL_2(\F_\ell) := \SL_2(\F_\ell) / \{ \pm I\}$.  
The main point is that as equivalence classes of $\text{SL}_2(\bF_\ell)$ or $\text{PSL}_2(\bF_\ell)$-structures, the equivalence relation is given by conjugation in $\text{SL}_2(\ol{\bF}_\ell)$ or $\text{PSL}_2(\ol{\bF}_\ell)$, and hence it will make sense to speak of the trace of the monodromy given by a generator of inertia in the fundamental group of a punctured elliptic curve (see \S\ref{ss_trace_invariant} below).

Fix a prime $\ell$, let $N_{\SL_2(\ol{\bF}_\ell)}(\SL_2(\bF_\ell))$ be the normalizer of $\SL_2(\bF_\ell)$ in $\SL_2(\ol{\bF}_\ell)$ and define
\begin{equation}
   D(\ell) := N_{\SL_2(\ol{\bF}_\ell)}(\SL_2(\bF_\ell))/\SL_2(\bF_\ell) .
\end{equation}
It can be shown that for $\ell=2$, $D(\ell)$ is trivial, and for odd $\ell$, $D(\ell)$ is a cyclic group of order two, generated by any matrix of the form $\spmatrix{u}{0}{0}{u^{-1}}$ where $u\in\ol{\bF}_\ell - \bF_\ell$ with $u^2\in\bF_\ell$. %\will{Does anyone have a reference for this?}
Let $F_2$ denote the free group of rank 2. Thus $D(\ell)$ acts on $\Epi^{\ext}(F_2,\SL_2(\bF_\ell))$ by conjugation, and its action commutes with the action of $\Out(F_2)$, so $\Out(F_2)$ also acts on 
\begin{equation} \label{eq:fl}
    F(\ell) := \Epi^{\ext}(F_2,\SL_2(\bF_\ell))/D(\ell).
\end{equation}

Since $\{\pm I\}$ is the center of $\SL_2(\bF_\ell)$, $\PSL_2(\bF_\ell)$ is a characteristic quotient and hence $D(\ell)$ acts on $\Epi^{\ext}(F_2,\PSL_2(\bF_\ell))$, and again its action commutes with the action of $\Out(F_2)$. Thus as before $\Out(F_2)$ also acts on 
\begin{equation} \label{eq:flbar}
    \overline{F(\ell)} := \Epi^{\ext}(F_2,\PSL_2(\bF_\ell))/D(\ell).
\end{equation}

It follows from \cite[Theorem 30]{Stein16} (noting that there are no field and graph automorphisms for $\SL_2(\bF_\ell)$) that $\Out(\SL_2(\bF_\ell))$ has order 2 for odd $\ell$ and is trivial otherwise. Thus, $D(\ell)$ induces the full outer automorphism group of $\SL_2(\bF_\ell)$. Indeed, if $\ell = 2$, $D(2) = \Out(\SL_2(\bF_2)) = 1$; if $\ell$ is odd, then if $A$ represents the nontrivial element of $D(\ell)$, then conjugation by $A$ cannot be an inner automorphism, for otherwise $AB = \pm I$ for some $B\in\SL_2(\bF_\ell)$, which is absurd. Similarly, $\Out(\PSL_2(\bF_\ell))$ also has order 2 for $\ell$ odd, so we find that $D(\ell)$ also induces the full outer automorphism group of $\PSL_2(\bF_\ell)$. Translating this into our geometric setting gives:

\begin{prop} \label{prop:identifyfibers}
We have that $\cM(\SL_2(\bF_\ell))/D(\ell) = \cM(\SL_2(\bF_\ell))^{\abs}$ and $\cM(\PSL_2(\bF_\ell))/D(\ell) = \cM(\PSL_2(\bF_\ell))^{\abs}$.  The characteristic quotient $\SL_2(\bF_\ell)\rightarrow\PSL_2(\bF_\ell)$ induces a commutative diagram of $\Out(F_2)$-equivariant maps
\begin{equation}\label{eq_cd}
\begin{tikzcd}
\Epi^\ext(F_2,\SL_2(\bF_\ell))\ar[r]\ar[d] & \Epi^\ext(F_2,\PSL_2(\bF_\ell))\ar[d] \\
F(\ell)\ar[r] & \ol{F(\ell)}
\end{tikzcd}
\end{equation}
For any elliptic curve $E$ over $\bC$, $x_0\in E$, fix an isomorphism $\pi_1^{top}(E^\circ(\bC),x_0)\cong F_2$, yielding isomorphisms $\pi_1(\cM(1)_{\Qbar},E)\cong\widehat{\Out^+(F_2)} \simeq \widehat{\SL_2(\ZZ)}$ as in Theorem~\ref{thm_basic_properties}\eqref{part_monodromy}.  Relative to these isomorphisms, 
the monodromy action of $\pi_1(\cM(1)_{\Qbar})$ on the geometric fiber of $\cM(\SL_2(\bF_\ell))^{\abs}$ (resp. $\cM(\PSL_2(\bF_\ell))^\abs$) over $E$ corresponds to the  $\Out^+(F_2)$-action on $F(\ell)$ (resp. $\ol{F(\ell)}$).

In particular, the diagram (\ref{eq_cd}) induces the following diagram of in the category of finite \'{e}tale stacks over $\cM(1)_{\bZ[1/|\SL_2(\bF_\ell)|]}$.
\[\begin{tikzcd}
\cM(\SL_2(\bF_\ell))\ar[r]\ar[d] & \cM(\PSL_2(\bF_\ell))\ar[d] \\
\cM(\SL_2(\bF_\ell))^\abs\ar[r] & \cM(\PSL_2(\bF_\ell))^\abs
\end{tikzcd}\]
\end{prop}

\begin{proof} This follows from the above discussion and Theorem \ref{thm_basic_properties}.
\end{proof}

\subsection{The Higman and trace invariants}\label{ss_trace_invariant}
In this section we will define the trace invariant of an absolute $\SL_2(\bF_\ell)$ (resp. $\PSL_2(\bF_\ell)$)-structure on an elliptic curve, which is derived from the finer Higman invariant. The Higman invariant can be defined for general $G$-structures, and in the case of a covering of an elliptic curve branched only over the origin, it is equivalent to the ``Hurwitz datum'' associated to a branched Galois covering of curves (see \cite[\S2.2]{br11}). %The Hurwitz datum is in some sense the most canonical object to work with. However, to keep the terminology  and development consistent with \cite{mp18} and \cite{mw13}, we will treat the Higman invariant as the key notion.

\subsubsection*{The Higman invariant geometrically}
Let $G$ be a finite group. The stacks $\cM(G)_{\bC}$ are generally not connected.  One can often distinguish connected components of $\cM(G)_{\bC}$ using the \emph{Higman invariant}. Analytically, given a $G$-torsor over a complex-analytic elliptic curve $E$, the monodromy around a small positively oriented loop on $E$ winding once around the puncture determines a conjugacy class of $G$, which we call the Higman invariant of the torsor. Note that such a loop represents the commutator of a positively oriented basis of $\pi_1(E^\circ)$. Since the mapping class group of the torus must preserve the homotopy class of this loop, the Higman invariant is locally constant on $\cM(G)_{\bC}$.

Algebraically, let $E$ be an elliptic curve over an algebraically closed field $k$ of characteristic prime to $|G|$. Let $f^\circ : X^\circ\rightarrow E^\circ$ be a $G$-torsor, and let $f : X\rightarrow E$ be the extension of $f$ to a finite map of smooth proper curves. Then the action of $G$ extends to $X$ and acts transitively on the fibers of $f$, in particular on the fiber $X_O$ above the puncture $O \in E$. For any point $x \in X_O$, the inertia group $G_x := \Stab_G(x)$ is cyclic of order equal to the ramification index at $x$. Let $T_x$ be the Zariski tangent space at $x$. Since $G_x$ acts freely on $\cO_{X,x}$, we obtain an injective local monodromy representation 
$$\rho_x : G_x\hookrightarrow\GL(T_x).$$
Since $T_x$ is a 1-dimensional $k$-vector space, $\GL(T_x)\cong k^\times$ canonically and hence $\rho_x$ identifies $G_x$ with the group of $e$-th roots of unity in $k$. If we make a choice of a primitive $e$-th root of unity $\zeta\in k$, then we may define the Higman invariant\footnote{By the Chevalley-Weil formula, this Higman invariant also determines the Hurwitz representations associated to the cover, and conversely by \cite[\S3.2]{br11}  the Hurwitz representations also determine the Higman invariant.} of $f$ (relative to $\zeta$) to be the set
$$\{\rho_x^{-1}(\zeta) \;\big|\; x\in X_O\}.$$
Since $G$ acts transitively on $X_O$, this set is a conjugacy class of $G$. In some situations it is useful to speak of the Higman invariant relative to a primitive $n$th root of unity $\zeta_n$ where $e\mid n$. In that case the Higman invariant of a torsor with ramification index $e$ relative to $\zeta_n$ is by definition the Higman invariant relative to $\zeta_n^{n/e}$.

The Higman invariant is locally constant. This is shown for example in \cite[Proposition 3.2.5]{br11} for Hurwitz data, but for completeness we give an argument here.

\begin{prop} Let $n := |G|$, and let $\zeta_n$ be any primitive $n$th root of unity in $\bC$. Then the Higman invariant relative to $\zeta_n$ is locally constant on $\cM(G)_{\bZ[1/n,\zeta_n]}$. 
\end{prop}
\begin{proof} For any geometric point $x : \Spec k\rightarrow\cM(G)_{\bZ[1/n,\zeta_n]}$, let $U\rightarrow\cM(G)_{\bZ[1/n,\zeta_n]}$ be a connected \'{e}tale neighborhood of $x$, corresponding to an elliptic curve $E/U$ with $G$-structure. Possibly passing to a further \'{e}tale localization, we may assume that the $G$-structure is given by a $G$-torsor $f^\circ : X^\circ\rightarrow E^\circ$. Since $U$ is regular, applying Lemma \ref{lemma_smoothness_of_normalization}, we may extend $f^\circ$ to a branched $G$-cover $f : X\rightarrow E$ of smooth proper curves, branched only above the zero section of $E$, such that the reduced ramification divisor is \'{e}tale over the zero section of $E$. \'{E}tale localizing even more, we may assume the reduced ramification divisor is a disjoint union of copies of $U$. Let $D\subset X$ be a component of the reduced ramification divisor, with ideal sheaf $\cI$. Then the conormal sheaf $\cI/\cI^2$ is invertible on $D$, and viewing $G$ as a constant group scheme over $U$, the action of $G_D := \Stab_G(D) \subset G$ on $D$ gives a homomorphism of constant $U$-group schemes
$$\rho : G_D\rightarrow\ul{\Aut}_D((\cI/\cI^2)^\vee)\cong\ul{\Aut}_U((\cI_U/\cI_U^2)^\vee)\cong\ul{\Aut}_U(\cO_U)\cong\cO_U^\times$$
where the isomorphisms are canonical. Thus the image of $\rho$ is necessarily contained in the subgroup scheme $(\mu_n)_U$. For geometric points $s\in U$, the restrictions $\rho|_{(G_D)_s}$ (taken up to conjugacy in $G_s$) give precisely the Hurwitz data of the restrictions $f_s : X_s\rightarrow E_s$, and the preimages $\rho_s^{-1}(\zeta_n)$ give precisely the Higman invariants of $f_s$. Since $\rho$ was a homomorphism of constant group schemes, this establishes the local constancy of the Higman invariant.
\end{proof}

\subsubsection*{Higman invariants on $\cM(G)$}
For a conjugacy class $C\subset G$, we wish to speak of the open and closed substacks of $\cM(G)$ having ``Higman invariant $C$''. For this to make sense, we must understand the ``ring of definition'' of a conjugacy class. %which we will describe as follows.

In this section let $k$ be any field of characteristic prime to $|G|$, not necessarily algebraically closed. Let $x : \Spec k\rightarrow\cM(G)$ be a point corresponding to a $G$-structure on an elliptic curve $E/k$. If $\iota : k\hookrightarrow\ol{k}$ is an algebraic closure, and $\zeta_e\in\ol{k}$ is a primitive $e$th root of unity, then the point
$$\ol{x} : \Spec\ol{k}\stackrel{\Spec{\iota}}{\longrightarrow}\Spec k\stackrel{x}{\longrightarrow}\cM(G)$$
corresponds to an actual $G$-torsor over an elliptic curve branched only over the origin. Thus, it makes sense to consider the Higman invariant of $\ol{x}$ relative to $\zeta_e^i$ for $i$ coprime to $e$. If the Higman invariant of $\ol{x}$ relative to $\zeta_e$ is $C$ (a conjugacy class of $G$), then the Higman invariant of $\ol{x}$ relative to $\zeta_e^i$ is $C^i$. 

Suppose $S\subset G$ is a conjugation-stable subset consisting of elements of the same order $e$ (a union of conjugacy classes of order $e$).  Let $\bF$ be the prime subfield of $k$, and let $\bF(e)$ denote the minimal extension of $\bF$ containing a primitive $e$th root of unity. Let $\chi_e : \Gal(\bF(e)/\bF)\rightarrow(\bZ/e\bZ)^\times$ be the cyclotomic character. 
For any integer $i$, let $S^i := \{s^i : s\in S\}$ and let $R(S) := \{S^i : \text{$i$ is coprime to $e$}\}$.  Note $R(S)$ is a set of conjugation-stable subsets of $G$, and that for any $i$, $S^i\cap S$ is either empty or a union of conjugacy classes. The group $(\bZ/e\bZ)^\times$ naturally acts on primitive $e$th roots of unity, and we similarly define an action on $R(S)$ by
\begin{eqnarray} \label{eq:rhoS}
\rho_S : (\bZ/e\bZ)^\times & \longrightarrow & \Sym(R(S)) \\
i & \longmapsto & \left(S\mapsto S^i\right). \nonumber
\end{eqnarray}
Define $\rho_{\bF,S} : \Gal(\bF(e)/\bF)\rightarrow\Sym(R(S))$ by $\rho_{\bF,S} := \rho_S\circ\chi_e$, and (for any prime field $\bF$) let $\bF(S)\subset\bF(e)$ be the fixed field of $\ker(\rho_{\bF,S})$.  Suppose we are given an embedding $\bF(S)\hookrightarrow k$; for example this will be the case if $x$ is a point of $\cM(G)_{\bF(S)}$.  

\begin{defn} Keeping the notation as above, let $\mu_e/\bF(S)$ denote the group scheme which is the kernel of the $e$th power map on $\bG_{m,\bF(S)}.$
Its automorphism group is canonically isomorphic to $(\bZ/e\bZ)^\times$. Let $\mu_{e,S,\bF}$ denote the set of $\ker(\rho_S)$-orbits of closed points of order $e$ in the group scheme $\mu_e$ over $\bF(S)$. We will think of $\mu_{e,S,\bF}$ as a collection of closed subschemes of the group scheme $\mu_e/\bF(S)$.
\end{defn}

\begin{remark}
Explicitly, a choice of $\omega \in \mu_{e,S,\bF}$ amounts to a $\ker(\rho_S)$-orbit of primitive $e$-th roots of unity in $\bF(e)$.  There are two extreme cases. If $\rho_S$ is faithful (the sets $S^i$ are all distinct for $i\in(\bZ/e\bZ)^\times$), then $\bF(S) = \bF(\zeta_e)$ and the closed points of $\mu_e$ are $\bF(S)$-points, and in this case choosing $\omega$ is equivalent to choosing a primitive $e$-th root of unity. In the other extreme, if $\rho_S$ is trivial ($S^i = S$ for every $i$ coprime to $e$), then $\bF(S) = \bF$, $\ker(\rho_S) = \Aut(\mu_e) = (\bZ/e\bZ)^\times$, and there is only one choice of $\omega$, namely the closed subscheme corresponding to the collection of all primitive $e$th roots of unity. We think of $\bF(S)$ as the ``field of definition'' of $S$.
\end{remark}

\begin{lemma} \label{lem:higequiv}
Suppose $S\subset G$ is a conjugacy-stable set consisting of elements of the same order $e$. Let $\omega \in \mu_{e,S,\bF}$ be an orbit, and let $x : \Spec k\rightarrow\cM(G)$ be a point. Suppose we are given an embedding $\bF(S)\hookrightarrow k$. The following are equivalent:
\begin{itemize}
    \item[(a)] For some algebraic closure $\iota : k\hookrightarrow\ol{k}$ and some morphism
    $$z : \Spec\ol{k}\rightarrow\omega$$
    of $\Spec\bF(S)$-schemes with corresponding primitive $e$th root of unity $\zeta_e\in\ol{k}$, the Higman invariant of $\ol{x} := x\circ\Spec(\iota)$ relative to $\zeta_e$ is contained in $S$.
    \item[(b)] For \emph{any} algebraic closure $\iota : k\hookrightarrow\ol{k}$ and any morphism
    $$z: \Spec\ol{k}\rightarrow\omega$$
    of $\Spec\bF(S)$-schemes with corresponding primitive $e$th root of unity $\zeta_e\in\ol{k}$, the Higman invariant of $\ol{x} := x\circ\Spec(\iota)$ relative to $\zeta_e$ is contained in $S$.
\end{itemize}
\end{lemma}

\begin{proof} The key point is that if $z'$ is another $\ol{k}$-point of $\omega$, corresponding to a root of unity $\zeta_e'\in\ol{k}$, then because $z,z'$ are morphisms of $\Spec\bF(S)$-schemes, they differ by some $\sigma\in\Gal(\ol{k}/\bF(S))$, so $\zeta_e' = \zeta_e^i$ for some $i\in\ker(\rho_S)$ by definition of $\bF(S)$.  
If $C\subset S$ is the Higman invariant of $\ol{x}$ relative to $\zeta_e$, then $C^i\subset S^i$ is the Higman invariant relative to $\zeta_e'$, but as $i\in\ker(\rho_S)$, %and hence
$S^i = S$, %but in both cases 
and hence the Higman invariant lies in $S$. A similar argument also implies independence of the choice of algebraic closure. 
\end{proof}

When any of the equivalent conditions in Lemma~\ref{lem:higequiv} are satisfied for a conjugation-stable set of elements of the same order $e$, we say that
\emph{$x$ has Higman invariant in $S$ relative to $\omega$}.

Now let
$\cO_{\bQ(S)}$ be the ring of integers of $\bQ(S)$; for any prime $\mf{p}\subset\cO_{\bQ(S)}[1/|G|]$ lying above $(p)\subset\bZ$, the residue field $k(\mf{p})$ is equal to $\bF_p(S)$. Moreover, for any $\omega \in \mu_{e,S,\bQ}$, $\omega$ uniquely extends to a closed subscheme of the group scheme $\mu_e$ over $\cO_{\bQ(S)}[1/|G|]$, and its fiber over a prime $\mf{p}\subset\cO_{\bQ(S)}[1/|G|]$ is an element of $\mu_{E,S,k(\mf{p})}$, which we will call the \emph{reduction of $\omega$ mod $\mf{p}$}. By the above discussion together with local constancy, the following definition makes sense:
\begin{defn} Let $S\subset G$ be a union of conjugacy classes of order $e$. Let $\omega\in\mu_{e,S,\bQ}$. Given a point $x : \Spec k \rightarrow \cM(G)_{\cO_{\bQ(S)}[1/|G|]}$ lying over a prime $\mf{p}\subset\cO_{\bQ(S)}[1/|G|]$, we say that $x$ has Higman invariant in $S$ relative to $\omega$ if it has Higman invariant in $S$ relative to the reduction of $\omega$ mod $\mf{p}$. Given a connected scheme $B$ and a map $f : B\rightarrow \cM(G)_{\cO_{\bQ(S)}[1/|G|]}$, we will say that it has Higman invariant in $S$ relative to $\omega$ if for some (equivalently any) point $b\in B$, the map $b\hookrightarrow B\stackrel{f}{\rightarrow}\cM(G)_{\cO_{\bQ(S)}[1/|G|]}$ has Higman invariant in $S$ relative to $\omega$. If we do not specify $\omega$, then it will be understood that we take $\omega$ to be the orbit of $\exp(2\pi i/e)$ relative to our fixed embedding of $\ol{\bQ}$ in $\bC$.
\end{defn}

In particular, $\bQ(S)\subset\bQ(\zeta_e)$ is an abelian number field, and the substack of $\cM(G)_{\cO_{\bQ(S)}[1/|G|]}$ consisting of objects with Higman invariant in $S$ relative to $\omega$ is open and closed.

\begin{remark} If $k = \bC$, since the local picture at a ramified point is given by $z\mapsto z^e$, we may check that in this case the Higman invariant defined analytically for $E(\bC)$ agrees with the Higman invariant defined for $E$ relative to $\exp(2\pi i/e)$.\footnote{In the analytic setting, the choice of root of unity was implicitly defined by an orientation, which can be thought of as a choice of generator of the fundamental group of a punctured neighborhood of $O\in E(\bC)$. Varying the group $G$, this orientation implicitly makes the choices of roots of unity $\zeta_e := \exp(2\pi i/e)$  which are compatible in the sense that $\zeta_e^r = \zeta_{e/r}$ for any $r\mid e$. If $k = \Qbar$, such a system can be thought of as an ``\'{e}tale orientation'' on $E/\Qbar$.}
\end{remark}

\subsubsection*{The Higman invariant algebraically}\label{sss_combinatorial_higman}
By Theorem \ref{thm_basic_properties}\eqref{part_fibers}, for any elliptic curve $E$ over $\Qbar$ with mapping class group $\Gamma_E := \Gamma_{E(\bC)}$ and base point $x_0\in E^\circ(\bC)$, the connected components of $\cM(G)_{\Qbar}$ are in bijection with the $\Gamma_E$-orbits on the set
$$\Epi^{\ext}(\pi_1^{top}(E^\circ(\bC),x_0),G).$$
Fixing a basis for the fundamental group with intersection number $+1$ (a ``positively oriented basis''), we may identify $\pi_1^{top}(E^\circ(\bC),x_0)$ with the free group $F_2$ on two generators $a,b$, and $\Gamma_E$ with $\Out^+(F_2)$ (see Theorem \ref{thm_basic_properties}\eqref{part_monodromy}). Under this identification, a positively oriented generator of inertia around $O$ is given by the commutator $[a,b]$. Thus, the Higman invariant of the $G$-torsor corresponding to (the conjugacy class of) a homomorphism $\varphi : F_2\rightarrow G$ is the conjugacy class of the commutator $[\varphi(a),\varphi(b)]$. This class is an invariant of the $\Out^+(F_2)$ action on the set
$$\Epi^{\ext}(F_2,G)$$
and correspondingly it does not depend on the choice of positively oriented basis $a,b$. Abstractly, if $F_2$ is the free group on $a,b$, then we define the Higman invariant of an element of $\Epi^\ext(F_2,G)$ to be the conjugacy class of $\varphi([a,b]) = [\varphi(a),\varphi(b)]$.

\subsubsection*{The trace invariant} We now specialize to the cases of $G=\SL_2(\bF_\ell)$ and $G=\PSL_2(\F_\ell)$ for $\ell$ prime.

\begin{defn} Let $F_2$ be the free group on generators $a,b$. Given a homomorphism $\varphi : F_2 \to \SL_2(\F_\ell)$, the \emph{trace invariant} of  $\varphi$ is $\tr([\varphi(a),\varphi(b)])$. For a homomorphism $\varphi: F_2 \to \PSL_2(\F_\ell)$, its trace invariant is defined to be $\tr([A,B])$ where $A,B \in \SL_2(\F_\ell)$ are any lifts of $\varphi(a)$ and $\varphi(b)$.
\end{defn}

\begin{remark}\label{remark_good} Since the kernel of $\SL_2(\bF_\ell)\rightarrow\PSL_2(\bF_\ell)$ is central, the trace invariant for $\varphi : F_2\rightarrow\PSL_2(\bF_\ell)$ is well-defined on $\Epi^{\ext}(F_2,\PSL_2(\bF_\ell))$ and is independent of the choice of lifts. Moreover, the trace invariant is invariant under the actions of $\Out^+(F_2)$ and $D(\ell)$ on $\Epi^{\ext}(F_2,\SL_2(\bF_\ell))$, and likewise for $\Epi^{\ext}(F_2,\PSL_2(\bF_\ell))$.
\end{remark}

Geometrically, we will define:
\begin{defn}\label{def_trace_invariant} Given a point $x : \Spec\Qbar\rightarrow\cM(\SL_2(\bF_\ell))_{\Qbar}$, its \emph{trace invariant} is the trace of its Higman invariant. For a point $y\in\cM(\PSL_2(\bF_\ell))_{\Qbar}$, its trace invariant is the trace of the Higman invariant of any lift of $y$ to a point of $\cM(\SL_2(\bF_\ell))_{\Qbar}$. Such lifts are guaranteed to exist by Theorem~\ref{thm_basic_properties}\eqref{part_functoriality}, and the resulting trace is well-defined by Remark \ref{remark_good}.

For $t\in\bF_\ell$, let 
\begin{equation}\label{eq_trace_as_higman}
    S(t) := \{A\in\SL_2(\bF_\ell) \;|\; \tr(A) = t, A\ne\pm I\}
\end{equation}
By \cite[Proposition 5.1]{mw13}, $S(t)$ is a union of conjugacy classes of the same order $e$. Let $n_\ell := |\SL_2(\bF_\ell)|$. We will say that an object of  $\cM(\SL_2(\bF_\ell))_{\cO_{\bQ(S(t))}[1/n_\ell]}$ has trace invariant $t$ if it has Higman invariant in $S(t)$. Given an object $B\rightarrow\cM(\PSL_2(\bF_\ell))_{\cO_{\bQ(S(t))}[1/n_\ell]}$, we say that it has trace invariant $t$ if for some geometric point $b\in B$, there is a lift of $b$ to $\cM(\SL_2(\bF_\ell))_{\cO_{\bQ(S(t))}[1/n_\ell]}$ which has trace invariant $t$. For the same reasons as above such lifts exist and the resulting trace is well-defined.
\end{defn}

The local constancy of the Higman invariant implies the local constancy of the trace invariant. Thus, for any $\cO_{\bQ(S(t))}[1/n_\ell]$-algebra $A$, we obtain open and closed substacks
\begin{eqnarray*}
\cM(\SL_2(\bF_\ell))_{t,A} & \subset & \cM(\SL_2(\bF_\ell))_A \\
\cM(\PSL_2(\bF_\ell))_{t,A} & \subset & \cM(\PSL_2(\bF_\ell))_A
\end{eqnarray*}
consisting of objects of trace invariant $t$. As with the Higman invariant, one may check that the algebraic version of the trace invariant agrees with the geometric version in characteristic $0$ if one fixes an isomorphism $F_2\cong \pi_1^{\top}(E^\circ(\bC),x_0)$ sending $a,b$ to a positively oriented basis.

\begin{remark} Note that since $\SL_2(\bF_\ell)$ is never abelian and the Higman invariant is the commutator of a generating pair, it cannot be contained in any proper normal subgroup. Thus in Definition \ref{def_trace_invariant}, nothing is lost by excluding $\pm I$ from $S(t)$.
\end{remark}

\section{Markoff triples as level structures on elliptic curves}\label{section_markoff_vs_moduli}

The \emph{Markoff surface} is the surface $\bX$ (over $\ZZ$) given by the equation
\begin{equation} \label{eq:markoff}
 x^2 + y^2 + z^2 - xyz = 0.
\end{equation}
A Markoff triple is a solution to this equation. 

For any prime $\ell$, let $X(\ell) := \bX(\F_\ell)$ and $X^*(\ell) := X(\ell) - \{(0,0,0)\}$.

In this section we will explain how for $\ell \ge 3$, the set $X^*(\ell)$ is in bijection with the set of absolute $\SL_2(\bF_\ell)$-structures on an elliptic curve with trace invariant $-2$. Under this correspondence (Proposition \ref{prop:fibermarkoff}), the natural monodromy action of $\pi_1(\cM(1)_{\Qbar})$ translates into an action on $X^*(\ell)$ given by automorphisms of the Markoff surface. This correspondence allows us to translate the work of Bourgain--Gamburd--Sarnak \cite{bgs1,bgs16} and Meiri--Puder \cite{mp18} into statements about the connectedness of the substack $\cM(\PSL_2(\bF_\ell))_{-2}^{\abs} \subset \cM(\PSL_2(\bF_\ell))^{\abs}$, and the explicit ``coordinatization'' it provides results in a remarkably simple calculation of the ramification behavior of $M(\PSL_2(\bF_\ell))^\abs_{-2}$ over $j = 0,1728$ in $M(1)$. The ramification calculation has the pleasant consequences of establishing the existence of a $\bQ$-rational point of $M(\PSL_2(\bF_\ell))^\abs$ (hence proving the $\bQ$-rationality of the substack), and determining the monodromy group of $M(\PSL_2(\bF_\ell))_{-2}^\abs$ over $M(1)$ assuming it is connected. While not necessary for our argument, in \S\ref{ss_rational_point} we show how this rational point arises from a natural covering of classical congruence modular curves.  Finally, in \S\ref{ss_proof_of_asymptotic_statement} we conclude by putting everything together and proving Theorem~\ref{thm:main_B_intro}.

\subsection{The work of Bourgain--Gamburd--Sarnak and Meiri--Puder} \label{section_bgsmp}

Let $\tau_{12}$ (resp. $\tau_{23}$) be the automorphism of $\bA^3$ by exchanging the first and second coordinates (resp. second and third coordinates), and $R_3$ be the ``Vieta involution'' given by $R_3 : (x,y,z) \mapsto (x,y,xy-z)$.
Let $\Gamma$ be the group of automorphisms of $\bA^3$ generated by $R_3$, $\tau_{12}$, and $\tau_{23}$. 
The group $\Gamma$ clearly preserves $\bX$, and hence induces a group of permutations $\Gamma$ on $X^*(\ell)$.

Markoff showed that $\Gamma$ acts transitively on the set of positive integer Markoff triples \cite{markoff79,markoff80}.  The analogous question about the action of $\Gamma$ on $X^*(\ell)$
was studied by Bourgain--Gamburd--Sarnak in \cite{bgs16,bgs1} because of its connection with strong approximation. There, they show that the action is transitive for a density 1 set of primes $\ell$.

    \begin{thm}[Bourgain, Gamburd, Sarnak \cite{bgs1}] \label{thm:BGS} 
		Let $\cE$ be the set of primes $\ell$ for which $\Gamma$ does not act transitively on $X^*(\ell)$. For any $\varepsilon >0$, the number of primes $\ell \leq T$ with $\ell \in \cE$ is at most $T^{\varepsilon}$, for $T$ large enough. Moreover, for any $\varepsilon>0$, the largest $\Gamma$-orbit in $X^*(\ell)$ is of size at least $|X^*(\ell)|- \ell^{\varepsilon}$, for $\ell$ large enough (whereas $|X^*(\ell)| \sim \ell^2$).
	\end{thm}
	
Moreover, they conjecture that transitivity holds for all $\ell$:

\begin{conj}[Bourgain, Gamburd, Sarnak] \label{conj:markoff}
Let $\ell$ be a prime.  The action of $\Gamma$ on $X^*(\ell)$ is transitive. 
\end{conj}

\begin{defn} \label{defn:markoffterminology}
Let $\cV$ denote the group of automorphisms of $\bA^3$ negating two of the coordinates.  It preserves $\bX$. Given $(x,y,z)\in X^*(\ell)$, its $\cV$-orbit is denoted $[x,y,z]$ and is called a \emph{block}. Let $Y^*(\ell) := X^*(\ell)/\cV$ denote the set of blocks of $X^*(\ell)$. Explicitly we have
$$[x,y,z] := \{(x,y,z),(x,-y,-z),(-x,y,-z),(-x,-y,z)\}.$$
Let $Q_\ell$ be the permutation group induced by the action of $\Gamma_{\ell}$ on $Y^*(\ell)$. This is the image of $\Gamma_\ell$ under the natural map $X^*(\ell)\rightarrow Y^*(\ell)$.
\end{defn}

Using the fact that $\bX$ is a ruled surface, the cardinality of $Y^*(\ell)$ can be computed as follows (see \cite[Lemmas 2.2, 2.3 ]{mp18} and \cite[Lemmas 3-5]{bgs1} )
\begin{equation}\label{eq_card}
    n_\ell := |Y^*(\ell)| = \left\{\begin{array}{rl}
	\frac{\ell(\ell+3)}{4} & \text{if $\ell\equiv 1\mod 4$} \\
	\frac{\ell(\ell-3)}{4} & \text{if $\ell\equiv 3\mod 4$}.
	\end{array}\right.
\end{equation}

Assuming the transitivity of $Q_\ell$ (equivalently $\Gamma_\ell$), Meiri and Puder are able to describe $Q_\ell$ in most cases:

	\begin{thm}[Meiri, Puder {\cite[Theorem~1.3, 1.4]{mp18}}] \label{thm:MP} Let $\ell\ge 5$ be a prime.
	\begin{enumerate}
		\item If $\ell \equiv 1 \mod 4$ and $Q_\ell$ is transitive, then $Q_\ell$ is the full alternating or symmetric group on $Y^*(\ell)$.
		\item If $\ell\equiv 3 \mod 4$, $Q_\ell$ is transitive, and the property $\BP(\ell)$ holds, then $Q_\ell$ is the full alternating or symmetric group on $Y^*(\ell)$. 
	\end{enumerate}
	\end{thm}
	
In the latter case, the condition in $\BP(\ell)$ is satisfied for a density one set of primes. Moreover, Meiri and Puder conjecture that $Q_\ell$ is always the full alternating or symmetric group:

\begin{conj}[{\cite[Conjecture 1.2]{mp18}} ] For $\ell\ge 5$, $Q_\ell$ is the full alternating or symmetric group on $Y^*(\ell)$.
\end{conj}
	
\begin{remark}
When $Q_\ell$ is the full alternating or symmetric group on $Y^*(\ell)$, it follows from \cite[Theorem 1.2]{cgmp} that $Q_\ell$ is the alternating group when $\ell \equiv 3 \mod{16}$ and the symmetric group otherwise.
\end{remark}

\subsection{Markoff triples as absolute \texorpdfstring{$G$}{G}-structures} \label{ss_markoff_absolute}
Since the trace invariant for both $\SL_2(\bF_\ell)$ and $\PSL_2(\bF_\ell)$ is invariant under $D(\ell)$ (see \S\ref{ss_absolute_G_structures}), it is well-defined on $\cM(\SL_2(\bF_\ell))/D(\ell)$. Since the Higman invariant is locally constant, so is the trace invariant. Thus, using the notation of Definition \ref{def_trace_invariant}, for any $t \in\F_\ell$ and any $\cO_{\bQ(S(t))}[1/n_\ell]$-algebra $A$, we may consider the open and closed substacks
\begin{align*}
    \cM(\SL_2(\bF_\ell))_{t,A}^\abs := \cM(\SL_2(\bF_\ell))_{t,A}/D(\ell) &\subset \cM(\SL_2(\bF_\ell))_{A}/D(\ell) = \cM(\SL_2(\bF_\ell))_{A}^{\abs}\\
    \cM(\PSL_2(\bF_\ell))_{t,A}^\abs := \cM(\PSL_2(\bF_\ell))_{t,A}/D(\ell) & \subset \cM(\PSL_2(\bF_\ell))_{A}/D(\ell) = \cM(\PSL_2(\bF_\ell))_{A}^{\abs}.
\end{align*}
consisting of absolute $\SL_2(\F_\ell)$ or $\PSL_2(\F_\ell)$-structures on elliptic curves with trace invariant $t$. 
Likewise we may consider $F(\ell)_{t}\subset F(\ell)$ and $\overline{F(\ell)}_t\subset\ol{F(\ell)}$ which are the geometric fibers with trace invariant $t$. By definition there are decompositions
\begin{equation}\label{eq_decomposition}
\cM(\SL_2(\bF_\ell))^\abs_{A} = \bigsqcup_{t\in\bF_\ell}\cM(\SL_2(\bF_\ell))^\abs_{t,A}\qquad \cM(\PSL_2(\bF_\ell))^\abs_{A} = \bigsqcup_{t\in\bF_\ell}\cM(\PSL_2(\bF_\ell))^\abs_{t,A}
\end{equation}
The traces for which the corresponding substack is nonempty is completely described in \cite{mw13} (in particular, for $\ell \ge 13$, every trace except 2 appears). We conjecture (see \ref{conj_complete_invariant} below) that the decompositions of \eqref{eq_decomposition} are decompositions into \emph{connected components}.

\begin{defn}
We define
\begin{eqnarray*}
\Tr : \Hom ( F_2 , \SL_2(\F_\ell)) & \longrightarrow & \F_\ell^3 \\
\varphi & \mapsto & (\tr(\varphi(a)),\tr(\varphi(b)),\tr(\varphi(ab))).
\end{eqnarray*}
The group $\cV$ acts on $\bF_\ell^3 = \bA^3(\bF_\ell)$ by negating two of the coordinates. We may similarly define
\begin{eqnarray*}
\overline{\Tr} : \Hom ( F_2 , \PSL_2(\F_\ell)) & \longrightarrow & \F_\ell^3 /\cV \\
\varphi & \mapsto & [(\tr(A),\tr(B),\tr(AB))]
\end{eqnarray*}
where $A, B \in \SL_2(\F_\ell)$ are any lifts of $\varphi(a)$ and $\varphi(b)$.
\end{defn}

It is again straightforward to verify that $\overline{\Tr}$ is well-defined. Since traces are conjugation invariant, $\Tr,\ol{\Tr}$ descend to maps on $F(\ell),\ol{F(\ell)}$ (defined in \eqref{eq:fl} and \eqref{eq:flbar}), giving a commutative diagram
\begin{equation}\label{eq_character_map}
\begin{tikzcd}
F(\ell)\ar[r,"\Tr"]\ar[d] & \bF_\ell^3\ar[d] \\
\ol{F(\ell)}\ar[r,"\ol{\Tr}"] & \bF_\ell^3/\cV.
\end{tikzcd}
\end{equation}
By the work of Macbeath, for any prime $\ell$, the horizontal maps in (\ref{eq_character_map}) are \emph{injective} \cite[Theorem 3]{mac69}.  (This injectivity holds even if $\ell$ is replaced by a prime power.)

\begin{remark}\label{remark_character_variety} This injectivity morally comes from the fact that $\bA^3$ is the \emph{character variety} for $\SL_2$-representations of $F_2$. Over $\bC$, this is just the classical Fricke-Vogt theorem giving an isomorphism
$$\Hom(F_2,\SL_2(\bC))\git\SL_2(\bC)\cong\bA^3_\bC$$
(see \cite[\S2]{gold03} , \cite{gold04}). Over $\bZ$, more precise statements can be found in \cite{bh95}.
\end{remark}

If $(x,y,z) = (\tr(\varphi(a)),\tr(\varphi(b)),\tr(\varphi(ab)))\in\bF_\ell^3$ is the image of $\varphi\in F(\ell)$, then the trace invariant of $\varphi$ can be expressed in terms of $(x,y,z)$ as follows (see \cite[\S1.3 (8)]{gold04})
$$\tr(\varphi([a,b])) = x^2 + y^2 + z^2 - xyz - 2.$$

Thus $X(\ell)\cap\Tr(F(\ell))\subset\bF_\ell^3$ are exactly the images of elements of $F(\ell)$ of trace invariant -2. Furthermore:

\begin{prop} \label{prop:fibermarkoff}
For $\ell\ge 3$, restricting to objects of trace invariant -2, the maps $\Tr$ and $\overline{\Tr}$ of \eqref{eq_character_map} restrict to give a commutative diagram
    \begin{equation}\label{eq_SL2PSL}
\begin{tikzcd}
F(\ell)_{-2} \ar[d]\ar[r,"\Tr"] & X^*(\ell) \ar[d] \\
\overline{F(\ell)}_{-2} \ar[r,"\ol{\Tr}"] & Y^*(\ell).
\end{tikzcd}
\end{equation}
\begin{enumerate}
    \item \label{prop:fibermarkoff1} The horizontal maps are bijections, and all fibers of the vertical maps have cardinality $4$.
    
\item \label{prop:fibermarkoff2} Under these bijections, the action of $\Out(F_2)$ on $F(\ell)_{-2}$ (resp. $\ol{F(\ell)}_{-2}$) translates into the action of $\Gamma$ on $X^*(\ell)$ (resp. $Y^*(\ell)$) from Definition~\ref{defn:markoffterminology}. Specifically, letting $a,b$ be a basis for $F_2$, then $\Out(F_2)$ is generated by the images of the automorphisms $r,s,t\in\Aut(F_2)$ described below. We record the corresponding automorphisms they induce on $X^*(\ell)$ via $\Tr$:
$$\begin{array}{rcl}
r : (a,b) & \mapsto & (a^{-1},b)\\
s : (a,b) & \mapsto & (b,a) \\
t : (a,b) & \mapsto & (a^{-1},ab)
\end{array}\quad\text{corresponds to (under $\Tr$)}\quad
\begin{array}{rcl}
R_3 : (x,y,z) & \mapsto & (x,y,xy-z) \\
\tau_{12} : (x,y,z) & \mapsto & (y,x,z) \\
\tau_{23} : (x,y,z) & \mapsto & (x,z,y)
\end{array}$$

\item  \label{prop:fibermarkoff3} For $\ell = 3$, the sets appearing in (\ref{eq_SL2PSL}) are all empty. For $\ell\ge 5$, these sets are all nonempty. For $\ell = 2$, $F(\ell)_{-2},\ol{F(\ell)}_{-2}$ are both empty but $X^*(\ell),Y^*(\ell)$ are not. 
\end{enumerate}
\end{prop}

\begin{remark} We note that whereas $\Aut(F_2)$ acts on the right on $F(\ell)_{-2}$ (resp. $\ol{F(\ell)}_{-2}$), the automorphisms $R_3,\tau_{12},\tau_{23}$ act on the left on $X^*(\ell)$ (resp. $Y^*(\ell)$). Accordingly, $\Tr$ induces an \emph{anti-homomorphism} $\Tr_* : \Aut(F_2)\rightarrow\Aut(\bA^3)$. In particular, $\Tr_*(r\circ s) = \Tr_*(s)\circ\Tr_*(r)$.
\end{remark}

\begin{proof} The bijectivity for (\ref{prop:fibermarkoff1}) is explained in \cite[\S6]{mp18} , but we will give a short argument here. The top row is injective by \cite[Theorem 3]{mac69}. It is surjective by the analysis of \cite[\S11]{mw13}. Since $\ell\ne 2$, $\cV$ acts freely on $X^*(\ell)$ and hence all fibers of $X^*(\ell)\rightarrow Y^*(\ell)$ have cardinality 4. On the other hand, given a generating pair $A,B$ of $\PSL_2(\bF_\ell)$, with lifts $\tilde{A},\tilde{B}$ to $\SL_2(\bF_\ell)$, the four lifts $\{(\pm\tilde{A},\pm\tilde{B})\}$ of $(A,B)$ have distinct images in $X^*(\ell)$, so all fibers of $F(\ell)_{-2}\rightarrow\ol{F(\ell)}_{-2}$ also have cardinality 4. Thus, $\ol{\Tr}$ is bijective as well.

For (\ref{prop:fibermarkoff2}), the automorphisms $r,s,t$ are generators of $\Aut(F_2)$ (see \cite{mw13} \S2), thus
 their images in $\Out(F_2)\cong\GL_2(\bZ)$ are also generators. Finally, it is also easy to check that $s$ (resp. $t$) corresponds to $\tau_{12}$ (resp. $\tau_{23}$). To see that $r$ corresponds to $R_3$, we use the Fricke identity 
$$\tr(AB) + \tr(A^{-1}B) = \tr(A)\tr(B)$$
valid for $A,B \in \SL_2(R)$ for any ring $R$.  (It follows from the Cayley Hamilton theorem, which implies $A+A^{-1} = \tr(A)I$, upon multiplying by $B$ and taking traces.)

For (\ref{prop:fibermarkoff3}), the nonemptiness of the sets when $\ell\ge 5$ follows from the Trace Theorem of \cite{mw13}. The statements for $\ell = 2,3$ are checkable by hand.
\end{proof}

\begin{prop} \label{prop_Q_vs_Qplus} Let $Q_\ell^+\le Q_\ell$ be the permutation image of $\Out^+(F_2)$ acting on $Y^*(\ell)$ via $\ol{\Tr}$. Then $Q_\ell^+$ is transitive if and only if $Q_\ell$ is transitive. In particular, Conjecture \ref{conj:markoff} is equivalent to the connectedness of $\cM(\SL_2(\bF_\ell))^\abs_{-2,\Qbar}$, which in turn implies the connectedness of $\cM(\PSL_2(\bF_\ell))^\abs_{-2,\Qbar}$.
\end{prop}

\begin{proof}
Certainly transitivity of $\Out^+(F_2)$ implies the transitivity of $\Out(F_2)$. Now suppose $\Out(F_2)$ acts transitively. If $\Out^+(F_2)$ does not act transitively, then since $\Out^+(F_2)$ is normal of index 2 inside $\Out(F_2)$, any representative for its nontrivial coset must act on $F(\ell)_{-2}$ without fixed points, exchanging the two $\Out^+(F_2)$-orbits. However, this is false, since
$$s : (a,b)\mapsto (a^{-1},ab)$$
corresponds to the permutation $\tau_{12}\in\Gamma$ having the fixed point $(3,3,3)\in X^*(\ell)$. Making the appropriate translations
using Propositions~\ref{prop:identifyfibers} and \ref{prop:fibermarkoff}, 
Galois theory then yields the equivalence of Conjecture \ref{conj:markoff} with the connectedness of $\cM(\SL_2(\bF_\ell))^\abs_{-2}$. Finally, note that the transitivity of the action of $\Out^+(F_2)$ on $F(\ell)_{-2}$ implies the transitivity of $\Out^+(F_2)$ on $\ol{F(\ell)}_{-2}$, which again by Galois theory implies the connectedness of $\cM(\PSL_2(\bF_\ell))^\abs_{-2,\Qbar}$.
\end{proof}

Finally, we end this subsection by showing that ``trace invariant -2'' is defined over $\bQ$:

\begin{prop}\label{prop_trace_n2_is_Q_rational} Let $m_\ell := |\SL_2(\bF_\ell)|$. Then in the notation of \eqref{eq_trace_as_higman}, %Definition \ref{def_trace_invariant},
$\bQ(S(-2)) = \bQ$. In other words, the property ``has trace invariant $-2$'' is well-defined on $\cM(\SL_2(\bF_\ell))^\abs_{\bZ[1/m_\ell]}$ and $\cM(\PSL_2(\bF_\ell))^\abs_{\bZ[1/m_\ell]}$. 
\end{prop}

The corresponding open and closed substacks are denoted by
\begin{equation}
    \cM(\SL_2(\bF_\ell))^\abs_{-2,\bZ[1/m_\ell]}\quad\text{and}\quad \cM(\PSL_2(\bF_\ell))^\abs_{-2,\bZ[1/m_\ell]}.
\end{equation}

\begin{proof} We begin by considering $\SL_2(\bF_\ell)$. In the notation of Definition~\ref{def_trace_invariant}, we wish to show that $\bQ(S(-2)) = \bQ$. This would follow from the stronger statement that for any $s\in S(-2)$, $s^i\in S(-2)$ for any $i$ coprime to $|s|$. Indeed, by \cite[Proposition 5.1]{mw13}, every element of $S(-2)$ is conjugate to one of
$$T := \spmatrix{-1}{1}{0}{-1},\text{ or } T' := \spmatrix{-1}{a}{0}{-1}$$
where $a\in\bF_\ell^\times - (\bF_\ell^\times)^2$ if $\ell$ is odd (for $\ell = 2$, $T'$ is not needed). If $\ell$ is odd, $T$ and $T'$ have order $2 \ell$ 
and if $\ell = 2$, then $T$ has order $2$. In either case, every other generator of the cyclic group $\langle T\rangle$ (resp. $\langle T'\rangle$) is an odd power of $T$ (resp. $T'$), and hence also lies in $S(-2)$.
%will follow from the following well-known lemma \jeremy{?}.

The case of $\PSL_2(\bF_\ell)$ is similar, where we note that the commutator trace of a generating pair of $\SL_2(\bF_\ell)$ cannot have trace $2$. Indeed, two elements with commutator trace $2$ must generate an affine subgroup (see \cite[\S 3]{mw13} ).
\end{proof}

\begin{remark} By Proposition \ref{prop_Q_vs_Qplus}, the Conjecture~\ref{conj:markoff} of Bourgain, Gamburd and Sarnak is equivalent to the connectedness of $\cM(\SL_2(\bF_\ell))_{-2}^\abs$. This connectedness would also follow from a special case of the ``Classification Conjecture'' of McCullough-Wanderley \cite{mw13}, which amounts to asserting that the union of the Higman invariant with its inverse completely characterizes the $\Out(F_2)$ orbits\footnote{The action of $\Out(F_2)$ does not preserve the Higman invariant. The ``orientation reversing'' automorphism $(x,y)\mapsto (y,x)$ replaces the Higman invariant with its inverse} on $\Epi^\ext(F_2,\SL_2(\bF_q))$. From our perspective, it is natural to ask if the Higman invariant is a complete invariant for the $\Out^+(F_2)$ orbits on $\Epi^\ext(F_2,\SL_2(\bF_q))$. We have computationally verified this for all $q\le 73$, except for $q = 9$, where it is false. Nonetheless, we will tentatively conjecture:
\end{remark}

\begin{conj}\label{conj_complete_invariant} For every prime power $q\ne 9$, the Higman invariant is a complete invariant for the orbits of $\Out^+(F_2)$ on $\Epi^\ext(F_2,\SL_2(\bF_q))$.
\end{conj}

Note that
for any conjugacy class $C$ of $\SL_2(\bF_q)$, the field $\bQ(C)$ is an abelian extension of $\bQ$ as $\rho_{\bQ,C}$ factors through an abelian group (see \eqref{eq:rhoS}).  Thus conjecture~\ref{conj_complete_invariant} would 
imply that the components of $\cM(\SL_2(\bF_q))_{\Qbar}$ are all defined over abelian number fields. Note that the conjecture cannot be true with $\SL_2(\bF_q)$ replaced by a general finite group $G$.  Indeed, by Theorem \ref{thm_faithful}, $\Gal(\Qbar/\bQ)$ acts faithfully on the set of components $\bigcup_G\pi_0(\cM(G)_{\Qbar})$, and hence the connected components cannot all be defined over abelian number fields.

\subsection{Explicit ramification behavior}\label{ss_explicit_ramification}
We now use the identifications given in Proposition~\ref{prop:identifyfibers} to analyze the ramification of $M(\SL_2(\bF_\ell))^\abs_{-2,\Qbar}$ and $M(\PSL_2(\bF_\ell))^\abs_{-2,\Qbar}$ over $M(1)$.
Let $a,b$ be a basis for $F_2$, and define the following automorphisms of $F_2$
$$\gamma_0 : (a,b)\mapsto (ab^{-1},a) \quad \gamma_{1728} : (a,b)\mapsto (b^{-1},a)\quad \gamma_{\infty} : (a,b)\mapsto (a,ab),\quad \gamma_{-I} : (a,b)\mapsto (a^{-1},b^{-1}).$$
Using the isomorphism $\Out^+(F_2)\cong\SL_2(\bZ)$, it is easy to check that $\Out^+(F_2)$ is generated by the images of $\gamma_0$ and $\gamma_{1728}$.
Since $\gamma_{-I}$ acts trivially on $X^*(\ell)$ and $Y^*(\ell)$, by Proposition \ref{prop_ramification_indices} the fibers of $M(\SL_2(\bF_\ell))^\abs_{-2}$ (resp. $M(\PSL_2(\bF_\ell))^\abs_{-2}$) above $j = 0,1728,\infty$ are precisely the quotients 
of $X^*(\ell)$ (resp. $Y^*(\ell)$) by the actions of the groups generated by $\gamma_0,\gamma_{1728},\gamma_\infty$ respectively, with ramification indices corresponding to orbit sizes. %Below, we use this to compute ramification indices above $j = 0,1728$.

\begin{remark}
The work in \cite{bgs1} studies the action of $\gamma_\infty$ on $X^*(\ell)$ in great detail. Moreover their ``rotations'' are all induced by a conjugate of $\gamma_\infty$ or $\gamma_\infty^{-1}$.  The geometric situation above the cusp $j = \infty$ is more delicate, and we do not treat it here.
\end{remark}

Above $j = 0,1728$, we have the following:
\begin{prop}\label{prop_explicit_ramification} Let $M$ be either $M(\SL_2(\bF_\ell))^\abs_{-2,\Qbar}$ or $M(\PSL_2(\bF_\ell))^\abs_{-2,\Qbar}$. The ramification indices of the map $M\rightarrow M(1)_{\Qbar}$ above $j = 0$ are all 3, except for a single unramified point which is $\bQ$-rational. For $M(\SL_2(\bF_\ell))^\abs_{-2,\Qbar}$, all ramification points above $j = 1728$ have index 2, and there are precisely two unramified points if $\ell\equiv 1,7\mod 8$, and no unramified points otherwise. In $M(\PSL_2(\bF_\ell))^\abs_{-2,\Qbar}$, again all ramified points above $j = 1728$ have index 2, and there is a unique unramified point if $\ell\equiv 1,7\mod 8$ and no unramified points otherwise.
\end{prop}
\begin{proof} We need to analyze the action of $\gamma_0,\gamma_{1728}$ on $X^*(\ell)$ and $Y^*(\ell)$. This is an explicit calculation which we have relegated to the Appendix (see Lemma \ref{lemma_explicit}). To see that the unique unramified points are rational, note that $\Gal(\Qbar/\bQ)$ acts on the fibers of $M\rightarrow M(1)_{\Qbar}$ above $j = 0,1728$, preserving ramification indices. If there is a unique unramified point, then it must be stabilized by $\Gal(\Qbar/\bQ)$, hence it must be $\bQ$-rational.
\end{proof}

Recall that $Q_\ell^+\le Q_\ell$ is the subgroup corresponding to the action of $\Out^+(F_2)$ on $Y^*(\ell)$ via $\ol{\Tr}$ (see  Propositions  \ref{prop:fibermarkoff}, \ref{prop_Q_vs_Qplus}),
and recall from \eqref{eq_card} that
\begin{equation}
    n_\ell = |Y^*(\ell)| = \left\{\begin{array}{rl}
\frac{\ell(\ell+3)}{4} & \ell\equiv 1\mod 4 \\
\frac{\ell(\ell-3)}{4} & \ell\equiv 3\mod 4. \\
\end{array}\right.
\end{equation}

Since $\Out^+(F_2)$ is generated by the images of $\gamma_0,\gamma_{1728}$, 
the above ramification description would allow us describe exactly when $Q_\ell^+$ is contained in the alternating group on $Y^*(\ell)$. 

\begin{prop} \label{prop:parity} For $\ell\ge 3$, on $Y^*(\ell)$, $\gamma_0$ always acts as an even permutation, and $\gamma_{1728}$ acts as an even permutation if and only if $\ell\equiv 1,3,13,15\mod 16$, and it is odd when $\ell \equiv 5,7,9,11\mod 16$.
\end{prop}
\begin{proof} From the ramification description given in Proposition \ref{prop_explicit_ramification} (or Lemma \ref{lemma_explicit}), we see that the parity of $\gamma_{1728}$ acting on $Y^*(\ell)$ is precisely the parity of the number of ramified points above $j = 1728$ in $M(\PSL_2(\bF_\ell))^\abs_{-2}$. This number is
$$\big(\text{\# ramified points above $j = 1728$}\big) = \left\{\begin{array}{ll}
\frac{\ell(\ell+3)-4}{8} & \text{ if }\ell\equiv 1\mod 8\\
\frac{\ell(\ell-3)}{8} & \text{ if }\ell\equiv 3\mod 8\\
\frac{\ell(\ell+3)}{8} & \text{ if }\ell\equiv 5\mod 8 \\
\frac{\ell(\ell-3)-4}{8} & \text{ if }\ell\equiv 7\mod 8.
\end{array}\right.$$
Computing the parity of this number forces us to consider $\ell\mod 16$, which gives us the desired result. We leave the details to the reader.
\end{proof}

\begin{cor}\label{cor_alternating_vs_symmetric} If $M(\PSL_2(\bF_\ell))^\abs_{-2,\Qbar}$ is connected and its monodromy group over $M(1)_{\Qbar}$ contains $A_{n_\ell}$, then the monodromy group is $A_{n_\ell}$ if $\ell\equiv 1,3,13,15\mod 16$, and is $S_{n_\ell}$ if $\ell\equiv  5,7,9,11\mod 16$.
\end{cor}

\subsection{Geometric construction of the \texorpdfstring{$\bQ$}{Q}-rational unique unramified point of \texorpdfstring{$M(\SL_2(\bF_\ell))^\abs_{-2,\Qbar}$}{the component} }\label{ss_rational_point}
Here we explain how the $\bQ$-rational unique unramified point of $M(\SL_2(\bF_\ell))^\abs_{-2}$ over $j=0$ from Proposition \ref{prop_explicit_ramification} can be obtained from a covering of congruence modular curves. The resulting point, relative to a certain isomorphism $F_2\cong\pi_1^{\top}(E^\circ(\bC),x_0)$ as in Proposition \ref{prop:identifyfibers}, will correspond to the point $(3,3,6)\in X^*(\ell)$. However, the connected component of $M(\SL_2(\bF_\ell))^\abs_{-2}$ containing that point is independent of the choice of isomorphism.

Let $\Gamma(1) := \SL_2(\bZ)$, then it can be checked (for example, in GAP) that its commutator subgroup $\Gamma(1)'$ is a torsion-free congruence subgroup of level 6 and index 12, with $\SL_2(\bZ)/\Gamma(1)'\cong\bZ/12\bZ$. In fact it can be verified in GAP that $\Gamma(1)'$ is free on the generators
\begin{equation}\label{eq_generators}
    \spmatrix{2}{-1}{-1}{1},\quad\spmatrix{1}{-1}{-1}{2}.
\end{equation}

Let $E$ be an elliptic curve over $\bQ$. In this subsection we will fix an isomorphism $\pi_1(\cM(1)_{\Qbar},E_{\Qbar})\cong\widehat{\SL_2(\bZ)}$ as in Theorem \ref{thm_basic_properties}, and using this isomorphism we will silently identify the two groups. 

In particular, $\widehat{\SL_2(\bZ)}$ is endowed with an action of $\Gal(\Qbar/\bQ)$ from the split exact sequence
$$1\rightarrow\widehat{\SL_2(\bZ)}\rightarrow\pi_1(\cM(1)_{\bQ},E_{\Qbar})\rightarrow\Gal(\Qbar/\bQ)\rightarrow 1$$
with splitting given by the section $E_* : \Gal(\Qbar/\bQ)\rightarrow\pi_1(\cM(1)_{\bQ},E_{\Qbar})$.

The closure $\ol{\Gamma(1)'}$ of $\Gamma(1)'$ inside $\widehat{\SL_2(\bZ)}$ corresponds to a degree 12 finite \'{e}tale cover $\cM'\rightarrow\cM(1)_{\Qbar}$ with $\pi_1(\cM'_{\Qbar})\cong\ol{\Gamma(1)'}$. Since $\Gamma(1)'$ is torsion free, $\cM'$ is a scheme.

The structure of $\cM'$ can be understood as follows. The map $\cM'\rightarrow M(1)_{\Qbar}$ is a branched covering of curves of degree 6. By Corollary \ref{cor_analytic_ramification}, it is \'{e}tale over the complement of $j = 0,1728,\infty$, with a unique totally ramified cusp (preimage of $\infty$), and 2 (resp. 3) points above $j = 0$ (resp. $j = 1728$) with ramification indices 3 (resp. 2).

By Riemann--Hurwitz, one finds that $\cM'$ is a once-punctured elliptic curve. By functoriality of coarse schemes, $\Gal(\cM'/\cM(1)_{\Qbar})\cong\bZ/12\bZ$ acts on the map $\cM'\rightarrow M(1)_{\Qbar}$. By Proposition \ref{prop_coarse_fibers}, one finds that this action yields a surjection of Galois groups $$\Gal(\cM'/\cM(1)_{\Qbar})\twoheadrightarrow\Gal(\cM'/M(1)_{\Qbar})$$
with kernel of order 2. Thus, $\Gal(\cM'/M(1)_{\Qbar})\cong\bZ/6\bZ$, and it preserves the cusp, so it acts (faithfully) as automorphisms of the punctured elliptic curve $\cM'$. Thus, we have
$$j(\cM') = 0.$$
Since $\ol{\Gamma(1)'}$ is characteristic inside $\widehat{\SL_2(\bZ)}$, it is preserved by the $\Gal(\Qbar/\bQ)$-action, and hence the map $\cM'\rightarrow\cM(1)_{\Qbar}$ has a $\bQ$-model $\cM'_\bQ\rightarrow\cM(1)_\bQ$.

\begin{remark} The stack $\cM'$ has a natural moduli interpretation over $\bZ[1/6]$. Given an elliptic curve $f : E\rightarrow S$ the sheaf $f_*\Omega^1_{E/S}$ is an invertible $\cO_S$-module. If it is free and $\omega\in H^0(S,f_*\Omega^1_{E/S})$ is a basis, and if moreover $6$ is invertible on $S$, then the basis $\omega$ determines a Weierstrass equation for $E$ inside $\bP^2_S$ (see \cite[\S2.2]{km85}) relative to which $\omega$ can be written as $\omega = \frac{-dx}{2y}$. Let $\Delta(E,\omega)$ be the discriminant of $E$ relative to the Weierstrass equation defined by $\omega$. By standard calculations (\cite[Table 3.1]{Sil09}), we have
$$\Delta(E,u\omega) = u^{-12}\Delta(E,\omega).$$
The moduli stack classifying pairs $(E,\omega)$ is naturally a $\bG_m$-torsor over $\cM(1)_{\bZ}$, and hence the moduli stack ``$\cM(\Delta = 1)$'' classifying pairs $(E,\omega)$ with $\Delta(E,\omega) = 1$ is a $\mu_{12}$-torsor over $\cM(1)_{\bZ[1/6]}$. In particular it is finite \'{e}tale. Since the abelianization of $\widehat{\SL_2(\bZ)} = \pi_1(\cM(1)_{\Qbar})$ is $\bZ/12\bZ$, $\cM(\Delta=1)$ is a moduli-theoretic model of $\cM'$ over $\bZ[1/6]$.
\end{remark}

We will obtain a $\bQ$-rational point of $\cM(\SL_2(\bF_\ell))^\abs$ by pulling back the modular curve corresponding to the principal congruence subgroup $\Gamma(\ell):=\ker(\SL_2(\bZ) \to \SL_2(\bZ/\ell\bZ))$ to $\cM'_\bQ$.
This modular curve is a geometric component of the classical moduli problem classifying elliptic curves with full level $\ell$ structure. This moduli problem is defined over $\bQ$, but over $\bQ$ it is not geometrically connected. Its geometrically connected components (ie, modular curves) classify elliptic curves $E$ equipped with a basis for $E[n]$ with a fixed Weil pairing; these components are defined over cyclotomic field $\bQ(\zeta_\ell)$. However, we may obtain a $\bQ$-model of the modular curve by twisting the definition of a full level $\ell$ structure as in  Deligne-Rapoport \cite[\S V4.1-2]{dr75} . We briefly recall the definition here.

For an integer $n\ge 1$ and an elliptic curve $E$ over a $\bZ[1/n]$-scheme $S$, the Weil pairing on $E[n]$ is a map of finite \'{e}tale $S$-schemes
$$e_n(*,*) : E[n]\times E[n]\longrightarrow\mu_n$$ which on geometric fibers is alternating and nondegenerate. Let
$$\omega(*,*) : (\mu_n\times\bZ/n\bZ)\times (\mu_n\times\bZ/n\bZ) \longrightarrow \mu_n$$
be defined on geometric fibers as the unique alternating pairing satisfying 
$$\omega((\zeta,0),(1,k)) = \zeta^k\qquad\text{for } \zeta\in\mu_n, k\in\bZ/n\bZ.$$
Then $\omega$ is also nondegenerate.  With $E$ as above, a \emph{full level $n$ structure of determinant $1$} on $E$ is an isomorphism of finite \'{e}tale group schemes $\alpha : \mu_n\times\bZ/n\bZ \overset{\sim}{\rightarrow} E[n]$ compatible with the pairings $e_n$ and $\omega$. Equivalently, we require that the following diagram commutes:
\[\begin{tikzcd}
(\mu_n\times\bZ/n\bZ) \times (\mu_n\times\bZ/n\bZ) \ar[rd,"\omega"']\ar[r,"\alpha\times\alpha"] & E[n]\times E[n]\ar[d,"e_n"] \\
 & \mu_n.
\end{tikzcd}\]
Let $\cM(\ell)_{\bQ}$ be the moduli stack classifying elliptic curves with full level $\ell$-structure of determinant 1 in the sense described above. Over $\bQ(\zeta_\ell)$, if we choose a primitive $\ell$th root of unity $\zeta\in\bQ(\zeta_\ell)$, then to any ``full level $n$ structure of determinant 1'' $\alpha : \mu_\ell\times\bZ/\ell\bZ\rightarrow E[\ell]$  as given above, we may associate the ``classical'' full level $\ell$ structure given by $(\alpha(\zeta,0), \alpha(1,1))$. Compatibility with $\omega,e_\ell$ implies that over $\bQ(\zeta_\ell)$, the set of full level $\ell$ structures of determinant 1 is in bijection with the set of full level $\ell$ structures with Weil pairing $\zeta$, so $\cM(\ell)$ is a $\bQ$-model of the classical modular curve $Y(\ell)$.

The finite \'{e}tale cover $\cM(\ell)_\bQ\rightarrow\cM(1)_\bQ$ is not Galois, but its base change to $\bQ(\zeta_\ell)$ is Galois with Galois group isomorphic to $\SL_2(\bZ/\ell\bZ)$ (using an isomorphism $(\bZ/\ell\bZ)^2\cong\mu_\ell\times\bZ/\ell\bZ$ over $\bQ(\zeta_\ell)$). Let $\cM(\ell)'_\bQ := (\cM(\ell)\times_{\cM(1)}\cM')_{\bQ}$, then we obtain the map
$$\pi_\ell : \cM(\ell)_{\bQ}'\longrightarrow\cM_{\bQ}'$$
appearing in \eqref{eq_mordell_covering_SL}. This map is finite \'{e}tale, and becomes $\SL_2(\bF_\ell)$-Galois over $\bQ(\zeta_\ell)$  
and thus it determines a $\bQ$-rational point of $\cM(\SL_2(\bF_\ell))^\abs$ using Remark~\ref{remark:absolute_G_structure}. 

To see that it has trace invariant $-2$, it suffices to work analytically. Note that the cover $\cM(\ell)^\an\rightarrow\cM(1)^\an$ corresponds to the congruence subgroup $\Gamma(\ell)$. With suitable choices of base points, the monodromy of $\cM(\ell)^\an\rightarrow\cM(1)^\an$ is given by the map
$$\SL_2(\bZ)\longrightarrow\SL_2(\bZ)/\Gamma(\ell)\cong\SL_2(\bZ/\ell\bZ)$$
and the monodromy of the pullback $\cM(\ell)'\rightarrow\cM'$ is just the composition
$$\varphi : \Gamma(1)'\hookrightarrow\SL_2(\bZ)\longrightarrow\SL_2(\bZ)/\Gamma(\ell)\cong\SL_2(\bZ/\ell\bZ).$$
Now one can compute, using the generators of $\Gamma(1)'$ given in (\ref{eq_generators}), that the trace invariant of $\varphi$ is precisely
$$\tr\left(\left[\spmatrix{2}{-1}{-1}{1},\spmatrix{1}{-1}{-1}{2}\right]\right) = \tr\left(\spmatrix{5}{6}{-6}{-7}\right) = -2.$$
Thus, $\pi_\ell$ determines a $\bQ$-rational point of $\cM(\SL_2(\bF_\ell))^\abs_{-2}$, and hence its image in $\cM(\PSL_2(\bF_\ell))^\abs_{-2}$ yields a $\bQ$-point of $\cM(\PSL_2(\bF_\ell))^\abs_{-2}$ as desired. Moduli theoretically, this point is given by the cover
\begin{equation}\label{eq_ol_pi_ell}
    \ol{\pi}_\ell : \cM(\ell)'_\bQ/\{\pm I\}\rightarrow\cM'_\bQ
\end{equation}
where $-I\in\SL_2(\bF_\ell)$ acts 
by negating the full level $\ell$-structure.

Finally, note that relative to our free basis $\spmatrix{2}{-1}{-1}{1},\spmatrix{1}{-1}{-1}{2}$ of $\Gamma(1)'$, the covering $\ol{\pi}_\ell$ corresponds to the point $[3,3,6]\in Y^*(\ell)$. By Proposition \ref{prop_explicit_ramification} and the calculation in Lemma \ref{lemma_explicit}, this implies that $\ol{\pi}_\ell$ indeed corresponds to the unique unramified point above $j = 0$ in $\cM(\PSL_2(\bF_\ell))^\abs_{-2}$.

\begin{remark} \label{rmk:ringofdef}
In fact, the analysis shows that this point is defined over $\ZZ[1/|\SL_2(\F_\ell)|]$.
\end{remark}

\begin{remark} We will conclude this section by describing how the above geometric construction which yields the unique unramified point of $M(\PSL_2(\bF_\ell))^\abs_{-2}$ may be deduced by ``pure thought''. We work over $\Qbar$, and assume $\ell\ge 5$.

The completion of the \'{e}tale local ring at a geometric point of a Deligne-Mumford stack is the universal deformation ring of that geometric point. Using the \'{e}tale local structure of Deligne-Mumford stacks \cite[Lemma 2.2.3]{AV02}, this implies that for a smooth separated 1-dimensional Deligne-Mumford stack $\cM$ with coarse scheme $c : \cM\rightarrow M$, then for any geometric point $x\in \cM$, the map induced by $c$ on \'{e}tale local rings at $x$ is a totally ramified extension of discrete valuation rings with ramification index $|G_x/K_x|$, where $G_x := \Aut_\cM(x)$ and $K_x\subset G_x$ is the subgroup consisting of automorphisms which extend to the universal deformation of $x$. This implies in particular that at the level of \'{e}tale local rings, the map $\cM(1)\rightarrow M(1)$ has ramification index 3 at $j = 0$.

Now let $x\in M(\PSL_2(\bF_\ell))_{-2}$ be a geometric point which is unramified over $j = 0$ in $M(1)$. Then the map $\cM(\PSL_2(\bF_\ell))_{-2}\rightarrow M(\PSL_2(\bF_\ell))_{-2}$ also has ramification index 3 above $x$. Such an $x$ must correspond to a $\PSL_2(\bF_\ell)$-torsor $p : X^\circ\rightarrow E^\circ$ where $E$ is an elliptic curve over $\Qbar$ with $j$-invariant 0. Viewing $x$ as a geometric point of $\cM(\PSL_2(\bF_\ell))_{-2}$, the fact that $\gamma_{-I}$ acts trivially on $Y^*(\ell)$ implies that $[-1]\in\Aut(E)$ extends to the universal deformation of $x$.\footnote{We elaborate on this: By \cite[\S5]{mw13}, there are two conjugacy classes in $\SL_2(\bF_\ell)$ with trace $-2$, and they are swapped by the action of $D(\ell)$. This implies that $\cM(\PSL_2(\bF_\ell))_{-2}$ is a disjoint union of two open and closed substacks (corresponding to different Higman invariants) which are mapped isomorphically onto each other by the nontrivial element of $D(\ell)$. Thus, the degree 2 finite \'{e}tale map $\cM(\PSL_2(\bF_\ell))_{-2}\rightarrow\cM(\PSL_2(\bF_\ell))^\abs_{-2}$ admits a section, hence induces isomorphisms on automorphism groups of geometric points. By Proposition \ref{prop_coarse_fibers}(\ref{prop_coarse_fibers1}), $\gamma_{-I}$ acting trivially on $Y^*(\ell)$ implies that $[-1]$ fixes the absolute $\PSL_2(\bF_\ell)$-structure determined by $p$, and we will say ``$[-1]\in\Aut_{\cM(\PSL_2(\bF_\ell))^\abs_{-2}}(x)$''. By the above, this also means that ``$[-1]\in\Aut_{\cM(\PSL_2(\bF_\ell))_{-2}}(x)$'', so $[-1]$ fixes the $\PSL_2(\bF_\ell)$-structure determined by $p$. Using Theorem \ref{thm_basic_properties}(\ref{part_combinatorial}), one can check that $[-1]$ must also fix the $\PSL_2(\bF_\ell)$-structure attached to the universal deformation of $x$ defined over $\Qbar\ps{t}$, which has trivial $\pi_1$.} The connection with deformation theory then implies that because $\cM(\PSL_2(\bF_\ell))_{-2}\rightarrow M(\PSL_2(\bF_\ell))_{-2}$ has ramification index 3 at $x$, $x$ must have an automorphism group of order 6. Since $\Aut(E) \cong \mu_6$, this implies that if $\alpha\in\Aut(E)$ is a generator, then $\alpha$ lifts via $p$ to a $\PSL_2(\bF_\ell)$-equivariant automorphism $\tilde{\alpha}$ of $X^\circ$. Since $\PSL_2(\bF_\ell)$ has trivial center, any such lifting is unique, and hence we obtain an action of $\mu_6$ on $X^\circ$, commuting with the $\PSL_2(\bF_\ell)$-action.

Next, let $W$ be a cyclic group of order 12, and fix a surjection $f : W\rightarrow \mu_6$. Then $W$ acts (non-faithfully) on $X^\circ$ and on $E^\circ$ via $f$, commuting with the $\PSL_2(\bF_\ell)$-action on $X^\circ$. Taking stacky quotients we obtain a diagram
\[\begin{tikzcd}
X^\circ\ar[r]\ar[d,"p"] & {[X^\circ/W]}\ar[d,"\ol{p}"] \\
E^\circ\ar[r] & {[E^\circ/W]}
\end{tikzcd}\]
where every map is finite \'{e}tale. Since all objects in the diagram are connected, comparing degrees we find that the diagram is cartesian. Since the action of $W$ on $E^\circ$ is isomorphic to the action of $\Gal(\cM'/\cM(1))$ on $\cM'$, it follows that $[E^\circ/W]\cong\cM(1)$ and that the monodromy representation of the $\PSL_2(\bF_\ell)$-torsor $\ol{p}$ (viewed as a covering of $\cM(1))$ is a surjection
$$\rho : \widehat{\SL_2(\bZ)}\rightarrow\PSL_2(\bF_\ell).$$
At this point, if we take $\rho$ to be the obvious surjection given by reduction mod $\ell$, then pulling back the $\PSL_2(\bF_\ell)$-cover of $\cM(1)$ corresponding to $\rho$ gives a cover of $E^\circ$ which by deformation theory must correspond to an ``unramified point'' in the sense described above. This establishes the existence of the unramified point.

To establish uniqueness, note that since the trace invariant of $p$ is -2, all ramification indices of $p$ and $\ol{p}$ above the puncture are $\ell$. It is a group-theoretic fact that any surjection $\SL_2(\bZ)\twoheadrightarrow\PSL_2(\bF_\ell)$ which sends $\spmatrix{1}{1}{0}{1}$ (a generator of inertia around the cusp) to an element of order $\ell$ must have kernel the projective principal congruence subgroup $\langle -I,\Gamma(\ell)\rangle$. This implies that $\ol{p}$ is isomorphic to the map $\cM(\ell)/\{\pm I\}\rightarrow\cM(1)$, and $p$ is isomorphic to the map $\ol{\pi}_\ell$ of \eqref{eq_ol_pi_ell}. This proves uniqueness.

\end{remark}

\subsection{Proof of Theorem~\ref{thm:main_B_intro}}\label{ss_proof_of_asymptotic_statement}
At this point we are ready to prove the ``asymptotic part'' of the main theorem, Theorem~\ref{thm:main_B_intro}. We begin by recalling the setup. Let $F_2$ be the free group on the generators $a,b$. For a prime $\ell\ge 3$, we have a diagram
\begin{equation} \label{asymptoticdiagram}\begin{tikzcd}
\Epi^\ext(F_2,\SL_2(\bF_\ell))/D(\ell)\ar[d]\ar[r,hookleftarrow] & F(\ell)_{-2}\ar[d]\ar[r,"\Tr"] & X^*(\ell)\ar[d] \\
\Epi^\ext(F_2,\PSL_2(\bF_\ell))/D(\ell)\ar[r,hookleftarrow] & \ol{F(\ell)}_{-2}\ar[r,"\ol{\Tr}"] & Y^*(\ell)
\end{tikzcd}\end{equation}
where $\Tr$ and $\ol{\Tr}$ are bijections.  Recall from \eqref{eq_card} that
$$ |Y^*(\ell)|  = n_\ell = \left\{\begin{array}{rl}
\frac{\ell(\ell+3)}{4} & \ell \equiv 1\mod 4 \\
\frac{\ell(\ell-3)}{4} & \ell \equiv 3\mod 4.
\end{array}\right.$$

There are natural actions of $\Out(F_2)$ on each object on the left square, compatible with all the morphisms, which descends via $\Tr$ to the action of $\Gamma$ on $X^*(\ell),Y^*(\ell)$ (see Proposition \ref{prop:fibermarkoff}). We recall that the permutation images of the $\Gamma$ action on $X^*(\ell)$ (resp. $Y^*(\ell)$) are denoted $\Gamma_\ell$ (resp. $Q_\ell$).
%on each object in the left square, which corresponds to the action of $\Gamma_\ell$ on $X^*(\ell)$ (see \ref{defn_markoff}) and $Q_\ell$ on $Y^*(\ell)$ (see  \ref{defn:markoffterminology}).
Let $E$ be an elliptic curve over $\bQ$, $x\in E^\circ(\bQ)$, and $x_\bC : \Spec\bC\rightarrow\Spec\bQ\rightarrow E^\circ$ be the corresponding $\bC$-point. Fixing an isomorphism $F_2\cong\pi_1^{\top}(E^\circ(\bC),x_0)$, then $F(\ell)_{-2}$ (resp. $\ol{F(\ell)}_{-2}$) is identified with the geometric fiber of $\cM(\SL_2(\bF_\ell))^\abs_{-2}$ (resp. $\cM(\PSL_2(\bF_\ell))^\abs_{-2}$) over $\cM(1)_{\Qbar}$ by Proposition~\ref{prop:identifyfibers}.
 Moreover, since $E$ is defined over $\bQ$ we also obtain an action of $\Gal(\Qbar/\bQ)$ on these fibers. For any $p\nmid |\PSL_2(\bF_\ell)|$ and any Frobenius element $\Frob_p\in\Gal(\Qbar/\bQ)$, $\Frob_p$ acts on $Y^*(\ell)$. The parity of the $\Frob_p$-action on $Y^*(\ell)$ does not depend on the choice of Frobenius element.

Let $\cE$ be the set of ``exceptional'' primes $\ell$ for which $\Gamma_\ell$ does not act transitively on $X^*(\ell)$. We will refer to $\cE$ as the ``exceptional set''. By Proposition \ref{prop_trace_n2_is_Q_rational}, it makes sense to ask if a $\PSL_2(\bF_\ell)$-torsor over a punctured elliptic curve over a $\bZ[1/|\PSL_2(\bF_\ell)|]$-scheme has trace invariant $-2$.

Let $c_\ell := |\PSL_2(\bF_\ell)| = \frac{\ell(\ell^2-1)}{2}$, and note that $6\mid c_\ell$. Recall that we defined $\BP(\ell)$ as:
\begin{equation*}%\label{eq_property_P_ell}
    \BP(\ell) := \begin{array}{r}\text{The property that either $\ell\equiv 1 \mod 4$, or} \\
\text{the order of $\frac{3+\sqrt{5}}{2}\in\bF_{\ell^2}$ is at least $32\sqrt{\ell+1}$.}
\end{array}
\end{equation*}

\begin{thm}\label{thm:main_B}  Let $\ell\ge 3$ be a prime such that $\ell\notin\cE$ and the condition $\BP(\ell)$ holds.  
\begin{enumerate}
    \item \label{thm:main_B1} The substack $\cM(\PSL_2(\bF_\ell))^\abs_{-2,\bZ[1/c_\ell]}\subset\cM(\PSL_2(\bF_\ell))^\abs_{\bZ[1/c_\ell]}$ of objects of trace invariant $-2$ from Proposition \ref{prop_trace_n2_is_Q_rational} is finite \'{e}tale over $\cM(1)_{\bZ[1/c_\ell]}$. 
    
    \item \label{thm:main_B2} Its coarse scheme $M_\ell = (M(\PSL_2(\bF_\ell))^\abs_{-2})_{\bZ[1/c_\ell]}$ is smooth and geometrically connected over $\bZ[1/c_\ell]$.
    
    \item \label{thm:main_B3} Let $M_\ell^\circ\subset M_\ell$ be the preimage of $M(1)_{\bZ[1/c_\ell]}^\circ := M(1)_{\bZ[1/c_\ell]} - \{j = 0,1728\}$. Then there is an at most quadratic extension $L$ of $\bQ$ such that if $B$ is the integral closure of $\bZ[1/c_\ell]$ in $L$, then the Galois closure $N_B$ of $M_{\ell,B}^\circ\rightarrow M(1)^\circ_B$ is geometrically connected over $B$ with Galois group
$$\Gal(N_B/M(1)^\circ_B)\cong\left\{\begin{array}{rl}
S_{n_\ell} & \text{if }\ell\equiv 5,7,9,11 \mod 16 \\
A_{n_\ell} & \text{if }\ell\equiv 1,3,13,15\mod 16.
\end{array}\right.$$

\item \label{thm:main_B4} If $\Gal(N_B/M(1)^\circ_B) \cong S_{n_\ell}$, then we may take $L = \bQ$ and $B  = \bZ[1/c_\ell]$, and the fiber of $N_B \to M(1)_B^\circ$ at $p \nmid c_\ell$ is an $S_{n_\ell}$-Galois cover of $\bP^1_{\bF_p} - \{0,1728,\infty\}$. 

\item \label{thm:main_B5} If $\Gal(N_B/M(1)^\circ_B)\cong A_{n_\ell}$, then the fiber at any prime above $p \nmid c_\ell$ will yield an $A_{n_\ell}$-cover of $\bP^1_k - \{0,1728,\infty\}$, where $k = \bF_p$ if the action of $\Frob_p$ on $Y^*(\ell)$ is even, and $k = \bF_{p^2}$ if the $\Frob_p$-action is odd.
\end{enumerate}
\end{thm}

Theorem~\ref{thm:main_B_intro} is a less precise version of Theorem~\ref{thm:main_B}.

\begin{proof} 
The first statement follows from Theorem~\ref{thm_basic_properties}\eqref{part_etale}. Statements (\ref{thm:main_B2}) and the existence of the at most quadratic $L$ such that $N_B$ is geometrically connected follow from Proposition~\ref{prop_tame}.
The description of the Galois groups in (\ref{thm:main_B3}) follows from Corollary \ref{cor_alternating_vs_symmetric}, and the exact fields of definitions for the fibers in (\ref{thm:main_B4}) and (\ref{thm:main_B5}) are a consequence of Proposition~\ref{prop_tame} and Remark~\ref{rmk:frobeniusoddness}.
\end{proof}

\begin{remark} A subgroup $A$ of a group $B$ is $G$-defining if $A$ is normal and $B/A\cong G$. Let $E$ be an elliptic curve over $\bQ$, and $x_0\in E^\circ(\Qbar)$ a geometric point.  Recall that
$$\ol{F(\ell)} := \Epi^\ext(\pi_1(E^\circ_{\Qbar},x_0),\PSL_2(\bF_\ell))/D(\ell)$$
is in bijection with the set of $\PSL_2(\bF_\ell)$-defining subgroups of $\pi_1(E^\circ_{\Qbar},x_0)$. Via this bijection, we may speak of the trace invariant of a $\PSL_2(\bF_\ell)$-defining subgroup. 
In the notation of Theorem \ref{thm:main_B}, if $[L:\bQ]=2$, then it follows from the Chebotarev density theorem that for a fixed $\ell \equiv 1,3,13,15 \mod{16}$, the set of primes $p$ such that $\Frob_p$ acts as an even (resp. odd) permutation on the set of $\PSL_2(\bF_\ell)$-defining subgroups of $\pi_1(E^\circ_{\Qbar},x_0)$ of trace invariant $-2$ each have density $\frac{1}{2}$.
\end{remark}

\section{Large Markoff orbits over finite fields} \label{sec:largemarkoff}

For a prime $\ell$, we will adapt Meiri--Puder's method to apply to the maximal $Q_\ell$-orbit of $Y^*(\ell)$, and show that $Q_\ell$ is the full alternating or symmetric group on the maximal orbit.  In contrast to Theorem~\ref{thm:main_B}, this will remove the restriction coming from \cite{bgs1} that $\ell$ lies outside the small exceptional set $\mathcal{E}$ of primes. % (those for which $p^2-1$ is not very smooth). 
Since $Q_\ell$ acts transitively on the maximal orbit, many arguments from \cite{mp18} can be adapted.  This analysis will give Theorem~\ref{thm:main_A}. Throughout this section, we assume that $\ell$ is odd.

\subsection{The Large Orbit}
	\begin{defn}
		We let $\cO(\ell)$ denote a $Q_\ell$-orbit of maximal size in $Y^*(\ell)$, and denote by $\overline{Q}_\ell$ the permutation group induced by the $Q_\ell$-action on $\cO(\ell)$.
	\end{defn}

Although we do not know the size of a maximal orbit, we have a nice lower bound of its size given by Theorem~\ref{thm:BGS} and the easy computation of $|Y^*(\ell)|$ recorded in \cite[Lemmas 2.2 and 2.3]{mp18}.  In particular, for any $\varepsilon>0$ there is a smallest integer $N_{\varepsilon}$ such that for prime $\ell\geq N_{\varepsilon}$ we have that 
\begin{equation}
|\cO(\ell)| \geq |Y^*(\ell)|-\ell^{\varepsilon} = \begin{cases} 
\frac{\ell(\ell+3)}{4} - \ell^{\varepsilon} & \text{ if } \ell \equiv 1 \mod{4} \\
\frac{\ell(\ell-3)}{4} - \ell^{\varepsilon} & \text{ if } \ell \equiv 3 \mod{4}.
\end{cases}
\end{equation}
In particular, this shows there is a unique maximal orbit for sufficiently large $\ell$.

The main theorem in this section is the following.
	
	\begin{thm} \label{thm:max-orbit}
	Fix a prime $\ell$ for which $\BP(\ell)$ holds.  If $\ell \equiv 1 \mod{4}$ then assume that $\ell \geq \max(N_{1/2},13)$, while if $\ell \equiv 3 \mod{4}$ then $\ell \geq \max(N_{1/2},23)$.
	%and the order of $\frac{3+\sqrt{5}}{2}\in \F_{\ell^2}$ is at least $32\sqrt{\ell+1}$.
	 Then $\overline{Q}_\ell$ is the full alternating or symmetric group on $\cO(\ell)$.
	\end{thm}

	The proof of this theorem is very similar to the proofs of \cite[\S4]{mp18}.  
	The strategy is to show that the permutation action of $\overline{Q}_\ell$ on $\cO(\ell)$ is primitive, and then use a group-theoretic result classifying primitive permutation groups with extra conditions.  
	The case that $\ell \equiv 1 \mod{4}$ is significantly easier, and will be dealt with at the end in \S\ref{sec:p1mod4}. 
	When $\ell \equiv 3 \mod{4}$, we first establish analogues of \cite[Proposition 4.2 and 4.3]{mp18} that work for the maximal $\overline{Q}_\ell$-orbit $\cO(\ell)$, and use them to show the permutation action is primitive in \S\ref{sec:primitivity}. Then in \S\ref{sec:permutation} and \S\ref{sec:wreath} we use a classification of primitive 
	permutation groups containing an element with a large number of fixed points, which is a consequence of the classification of finite simple groups.  Our argument is significantly harder than that of \cite[\S4.3]{mp18} as we do not know the exact size of $\cO(\ell)$; this makes it more complicated to eliminate cases and conclude that
$\overline{Q}_\ell$ contains the alternating group of permutations of the set $\cO(\ell)$.

\subsection{Primitivity for \texorpdfstring{$\ell \equiv 3 \mod 4$}{l=3 mod 4}}\label{sec:primitivity}
We briefly recall some further background about Markoff triples modulo $\ell$.  The conic $C_1(\pm a) \subset Y^*(\ell)$ is the set of elements $[x_1,x_2,x_3] \in Y^*(\ell)$ with $x_1=a$.  There is a very important rotation element $\rot_1:=R_3\circ \tau_{23}$ sending $(x,y,z)$ to $(x,z,xz-y)$ which visibly acts on $Y^*(\ell)$ and preserves the conics $C_1(\pm a)$. 

Recall \cite[Definition 2.1]{mp18} that $x \in \F_{\ell}$ is called \emph{hyperbolic}, \emph{elliptic}, or \emph{parabolic} if $x^2 -4$ is a square in $\F_\ell^\times$, a non-square in $\F_\ell^\times$, or zero, respectively.   This categorization is invariant under sign change.  The cycle structure of $\rot_1$ on $C_1(\pm a)$ depends on whether $a$ is hyperbolic, elliptic, or parabolic: this is conveniently summarized in \cite[Lemma 2.2 and 2.3, Tables 1 and 2]{mp18}; we urge the reader to use them as reference.

We now begin our analysis of the maximal orbit $\cO(\ell)$ when $\ell \equiv 3 \mod 4$.

\begin{lemma} \label{lem:inOp}
Let $\ell \equiv 3 \mod{4}$ be prime.  If $\ell\geq N_{1/2}$ and $\BP(\ell)$ holds, then $[3,3,3]$ is in $\cO(\ell)$.
\end{lemma}	

\begin{proof}
When $\ell \equiv 3 \mod{4}$, it is shown in the proof of \cite[Thm.~4.1]{mp18} that assuming the order of $\frac{3+\sqrt{5}}{2}\in \F_{\ell^2}$ is at least $32\sqrt{\ell+1}$, the element $[3,3,3]$ belongs to a $\rot_1$-cycle of length at least $16\sqrt{\ell+1}$.  
As $\ell \geq N_{1/2}$, we know that $|Y^*(\ell) \backslash \cO(\ell)| \leq \ell^{1/2}$ so this cycle must lie in $\cO(\ell)$.
\end{proof}

Recall that a $\overline{Q}_\ell$-block is a subset $B \subseteq \cO(\ell)$ such that for every $g\in \overline{Q}_\ell$, either $gB=B$ or $g B\cap B=\emptyset$.  We say a coordinate $j \in \{1,2,3\}$ is homogenous in a block $B$ if the $j$th coordinate of every triple in $B$ has the same type (all hyperbolic, elliptic, or parabolic).

Remember that we say the $\overline{Q}_\ell$ action is \emph{primitive} if the only blocks are singletons and $\cO(\ell)$.  
	
\begin{proposition}[Analogue of {\cite[Prop~4.2]{mp18}}] \label{prop:4.2}
		Let $\ell \equiv 3 \mod 4$ be a prime such that $\ell \geq N_{1/2}$ and $\BP(\ell)$ holds, and let $B\subsetneq \cO(\ell)$ be a proper $\overline{Q}_\ell$-block. 
		Then at least two of the coordinates $\{1,2,3\}$ are homogeneous in $B$. 
	\end{proposition}
	
	\begin{proof}
		In order to follow the proof of \cite[Prop.~4.2]{mp18}, we use Lemma~\ref{lem:inOp} to see that $[3,3,3]$ is in the maximal orbit $\cO(\ell)$.  The rest of the argument is the same.
	\end{proof}

	  Recall that $d_\ell(\pm x)$ denotes the length of the cycles of $\rot_1$ on $C_1(\pm x)$.  
	
	\begin{proposition}[Analogue of {\cite[Prop.~4.3]{mp18}}] \label{prop:4.3}
		Let $\ell\equiv 3 \mod 4$ be a prime such that $\ell\geq N_{1/2}$ and $\BP(\ell)$ holds. Let $x\in \F_\ell\backslash \{0, \pm 2\}$ satisfy $d_\ell(\pm x) \geq 16 \sqrt{\ell+1}$. 
		Then for every $j\in \{1,2,3\}$, every proper $\overline{Q}_\ell$-block $B\subsetneq \cO(\ell)$ contains at most one solution with $j$-th coordinate $\pm x$. 
	\end{proposition}
	\begin{proof}
		The proposition follows from the proof of \cite[Prop.~4.3]{mp18} and Proposition~\ref{prop:4.2}.
	\end{proof}
	
	Using Proposition~\ref{prop:4.3}, the analogue of \cite[Cor.~4.4]{mp18} immediately follows. That is, under the conditions in Proposition~\ref{prop:4.3}, if $B\subsetneq \cO(\ell)$ is a proper $\overline{Q}_\ell$-block containing some solution with first coordinate $\pm x$, and another solution with first coordinates $\pm x'$, then $d_{\ell}(\pm x)= d_{\ell}(\pm x')$. With these ingredients, we obtain the analogue of \cite[Thm.~4.1]{mp18} by replacing $Y^*(\ell)$ by $\cO(\ell)$ and remembering that $[3,3,3]\in \cO(\ell)$ (Lemma~\ref{lem:inOp}).

	\begin{thm}[Analogue of {\cite[Thm.~4.1]{mp18}}]
		Let $\ell\equiv 3 \mod 4$ be a prime with $\ell\geq N_{1/2}$ and the property $\BP(\ell)$. 
		%Assume that the order of $\frac{3+\sqrt{5}}{2}\in \F_{\ell^2}$ is at least $32\sqrt{\ell+1}$. 
		Then the $\overline{Q}_\ell$-action on $\cO(\ell)$ is primitive.
	\end{thm}

\subsection{Analyzing the Permutation Group for \texorpdfstring{$\ell \equiv 3 \mod 4$}{l = 3 mod 4}}\label{sec:permutation}
In this section, we prove that $\overline{Q}_\ell$ contains the alternating group of permutations of the set $\cO(\ell)$. 
This is the analog of \cite[Proposition 4.15]{mp18}, but it is significantly harder to adapt the proof as we do not know the exact size of $\cO(\ell)$.  It relies on the following classification of primitive permutation groups containing an element which fixes at least half of the elements of the set it is acting on; Guralnick and Magaard obtain this result as a consequence of the classification of finite simple groups  \cite{gm98}.  We record the convenient formulation of \cite[Theorem~4.13]{mp18}.

\begin{thm} \label{thm:choices}
Let $G \subset S_n$ be a primitive permutation group, and let $x \in G$ have at least $n/2$ fixed points.  Then one of the following holds:
\begin{enumerate}
    \item \label{case1}$G = \operatorname{Aff}^2(k)$ is the affine group acting on $\F_2^k$ and  $x$ is a transvection (and so in particular is an involution with $n/2$ fixed points).

\item \label{ourcase} There are integers $r \geq 1$, $m\geq 5$, and $1 \leq k \leq m/4$ such that $n = \binom{m}{k}^r$, the group $S_m$ acts on the set $\Delta$ of $k$-element subsets of $\{1,\ldots, m\}$ in the natural way, $G \subset S_m \wr S_r$ acts on $\Delta^r$, and the socle $\Soc(G)$ of $G$ is $A_m^r$.

\item  \label{case3} For some integer $r \geq 1$, $n = 6^r$, the group $S_6$ acts on $\Delta = \{1,\ldots,6\}$ by applying an outer automorphism, $G \subset S_6 \wr S_r$ acts on $\Delta^r$, and $\Soc(G) = A_6^r$.

\item  \label{case4} The group $G$ is some variant of an orthogonal group over the field of two elements acting on some collection of $1$-spaces of hyperplanes, and the element $x$ is an involution.
\end{enumerate}
\end{thm}

Our ultimate goal is to show that $\overline{Q}_\ell$ occurs as Case \eqref{ourcase} with $r=1$ and $k=1$; this shows that $\overline{Q}_\ell$ is the full alternating or symmetric group on $\cO(\ell)$.  To do so requires some non-trivial group theory involving wreath products; we briefly recall some background now, and then proceed with the proof of  Theorem~\ref{thm:max-orbit}, making use of several technical results whose statements and proofs are deferred to \S\ref{sec:wreath}.

Let $G$ be a group acting on $\Delta = \{1 , \ldots, m\}$, and $r$ a positive integer.  Let $\Omega = \{1 ,\ldots ,r\}$, which has a natural action of $S_r$.  Recall that the \emph{wreath product} $G \wr S_r$ is defined to be the semi-direct product $G^r \rtimes S_r$, where $S_r$ acts on $G^r = \prod_{i \in \Omega} G$ by permuting the coordinates.  In particular, elements $\pi \in G \wr S_r$ are represented by pairs $(\sigma, \tau)$ where $\sigma = (\sigma_{1}, \ldots, \sigma_{r}) \in G^r$ and $\tau \in S_r$.  There is a natural action of $G \wr S_r$ on $\Delta^r$, given by
\begin{equation}\label{eq:wreathaction}
\pi  (x_1,\ldots,x_r) = (\sigma_{1}(x_{\tau^{-1}(1)}) , \ldots ,\sigma_{r}(x_{\tau^{-1}(r)})).
\end{equation}
This gives a natural embedding $\iota: G \wr S_r \hookrightarrow S_n$ where $n = m^r$. 

\begin{proof}[Proof of Theorem~\ref{thm:max-orbit} for $\ell\equiv 3 \mod 4$]
 We will follow the outline of \cite[Prop.~4.15]{mp18}.  At the beginning, it shows that the permutation $\pi=\rot_1^{(\ell+1)/2}$ fixes exactly $\frac{(\ell+1)(\ell-3)}{8}$ elements of $Y^*(\ell)$. Therefore, since $\ell\geq N_{1/2}$, we have $\pi$ fixes at least $\frac{(\ell+1)(\ell-3)}{8}-\ell^{1/2} >\frac{\ell(\ell-3)}{8}-\frac{\ell^{1/2}}{2} \geq \frac{|\cO(\ell)|}{2}$ elements of $\cO(\ell)$. Thus the $\overline{Q}_\ell$-action on $\cO(\ell)$ together with the permutation $\pi$ satisfies the assumptions in Theorem~\ref{thm:choices}, and we need to rule out all options except for \eqref{ourcase} with $k= r= 1$, so that $\overline{Q}_\ell=\Alt(\cO(\ell))$ or $\Sym(\cO(\ell))$.
			
 Let $q$ be a prime factor of $\frac{\ell-1}{2}$, and let $s$ be a prime factor of $\frac{\ell+1}{2}$.    By \cite[Table~2]{mp18}, $\rot_1$ contains $\frac{(\ell-1)(q-1)}{4q}$ cycles of size $q$ and $\frac{(\ell+1)(s-1)}{4s}$ cycles of size $s$ in $Y^*(\ell)$. So in $\cO(\ell)$, $\rot_1$ contains at least $\lceil\frac{(\ell-1)(q-1)}{4q} -\frac{\ell^{1/2}}{q}\rceil\geq 1$ cycles of size $q$ and at least $\lceil\frac{(\ell+1)(s-1)}{4s}-\frac{\ell^{1/2}}{s}\rceil \geq 1$ cycles of size $s$.
 Because $\frac{\ell-1}{2}$ is odd and $\pi$ has a cycle of size $q$, $\pi$ is not an involution so the cases \eqref{case1} and \eqref{case4} cannot occur.
 
Again using \cite[Table~2]{mp18}, notice that $\rot_1$ does not contain any cycle of size divisible by $q s$; every cycle length divides either $\frac{\ell-1}{2}$ or $\frac{\ell+1}{2}$ which are relatively prime. There would be cycles of length divisible by $qs$ if we were in case \eqref{ourcase} with $k \geq 2$; 
\cite[Lem.~4.14]{mp18} shows that the hypotheses of Corollary~\ref{cor:embedding} are satisfied (note the latter is applied with $\Delta$ the set of $k$ element subsets of $\{1,\ldots,m\}$, so the $m$ in Corollary~\ref{cor:embedding} is $\binom{m}{k}$ in the notation of Theorem~\ref{thm:choices}). Thus we can rule out case \eqref{ourcase} with $k\geq 2$.
 
Similarly, we will rule out case \eqref{case3}.  	
 Note that $\frac{\ell-1}{2}$ is odd as $\ell\equiv 3 \mod{4}$.  If $\frac{\ell+1}{2}$ is a power of $2$, then by the assumption $\ell\geq 23$ we have $\frac{\ell+1}{2} \equiv 0 \mod{8}$, and hence $\frac{\ell-1}{2} \equiv 7 \mod{8}$ cannot be a power of $3$.  Thus we may assume that $\{q,s\} \neq \{2,3\}$.  Then $S_6$ acting on $\{1, \ldots, 6\}$ satisfies the condition for $G$ in Corollary~\ref{cor:embedding} (no element of $S_6$ can contain a $q$-cycle and an $s$-cycle when $q + s > 6$) which implies there should be cycles of length divisible by $qs$.  But there aren't, so case \eqref{case3} can be ruled out. 
			
Finally, we need to rule out case \eqref{ourcase} with $r\geq 2$ and $k=1$.  So $\overline{Q}_\ell$ is a subgroup of $S_m \wr S_r$, and inherits the action of $S_m \wr S_r$ on $\{1, \ldots, m \}^r$.  We may identify $\cO(\ell)$ with a subset of $r$-tuples of elements of $\{1, \ldots, m\}$, and represent $\rot_1$ by a tuple $(\sigma,\tau)$ where $\sigma$ is a permutation of $\{1,\ldots,m\}$ and $\tau$ is a permutation of $\{1,\ldots,r\}$.   We have seen that for any primes $q$ and $s$ with $q \mid\frac{\ell-1}{2}$ and $s \mid A:=\frac{\ell+1}{2}$ there are $\rot_1$-cycles of size $q$ and of size $s$, but none of size divisible by $qs$.  Notice that the $\rot_1$-cycle of size $A $ lies in $\cO(\ell)$ as $|Y^*(\ell)\backslash \cO(\ell)|\leq \ell^{1/2}$.

Fix a prime $q$ dividing $\frac{\ell-1}{2}$ together with an element $(a_1,\ldots,a_r) \in \{1, \ldots, m \}^r = \cO(\ell)$.
Likewise let $(b_1,\ldots,b_r) \in \{1, \ldots, m \}^r$ be an element in the cycle of size exactly $A$.  Then by Corollary~\ref{cor:eliminatepower}, 
there exists $i \in \{1,\ldots, r\}$ such that $\tau(i)=i$ and such that the $j$th component of $\rot_1 (b_1,\ldots,b_r)$ is $b_j$ for any $j \neq i$.

In light of \eqref{eq:wreathaction}, the elements in the $\rot_1$-cycle containing $(b_1, \ldots, b_r)$ only differ in the $i$-th component.  Since the cycle has length $\frac{\ell+1}{2}$, we must have  $m\geq\frac{\ell+1}{2}$.  But $n = m^r$, and hence $n \geq \frac{(\ell+1)^2}{4}$.  This contradicts the fact that $n=|\cO(\ell)|\leq |Y^*(\ell)|=\frac{\ell(\ell-3)}{4}$. 

Thus the only possibility is case \eqref{ourcase} with $r=1$ and $k=1$, which shows that $\overline{Q}_\ell$ is the full alternating or symmetric group on $\cO(\ell)$.
\end{proof}

\subsection{Arguments with Wreath Products} \label{sec:wreath}
We continue the notation for wreath products introduced before the proof of Theorem~\ref{thm:max-orbit}, with $G \subset S_m$ and with $G \wr S_r$ acting on $\Delta^r = \{1,\ldots, m\}^r$.   The action gives a natural inclusion $\iota: G \wr S_r \into S_n$, where $n = m^r$.  We will need to work with powers in the wreath product: notice
we have that $\pi^i = (\sigma_i, \tau^i)$ for some $\sigma_i = (\sigma_{i,1} , \ldots, \sigma_{i,r}) \in G^r$.  Furthermore, we have
\begin{equation} \label{eq:wreathpower}
\pi^i (x_1,\ldots,x_r) = (\sigma_{i,1}(x_{\tau^{-i}(1)}) , \ldots ,\sigma_{i,r}(x_{\tau^{-i}(r)})).
\end{equation}

Our main goal is to establish the following:

\begin{prop} \label{prop:wreaths}
	Let $\pi \in G\wr S_r$. Assume $(a_1, \ldots, a_r), (b_1, \ldots, b_r) \in \Delta^r$ belong to $\iota(\pi)$-cycles of size divisible by distinct primes $q$ and $s$ respectively.
		Then at least one of the following holds:
		\begin{enumerate}
			\item \label{item:eb-1} $\iota(\pi)$ has a cycle of size divisible by $qs$;
			\item \label{item:eb-2} There exist $i\in\{1, \ldots, r\}$ and an integer $t$ relatively prime to $qs$ such that $\tau^t(i)=i$ and for any $j\neq i$ we have that the $j$-th component of $\pi^t (a_1, \ldots, a_r)$ is $a_j$, and that the $j$-th component of $\pi^t (b_1, \ldots, b_r)$ is $b_j$.
		\end{enumerate}
\end{prop}

In the proof, for convenience we will omit $\iota$ and also use $\pi$ to denote the image of $\pi$ in $S_n$.  

\begin{defn}
Given a subset $S \subset \{1,\ldots, r\}$ and $(x_1,\ldots, x_r), (y_1,\ldots, y_r) \in \Delta^r$ we say that
\[
 (x_1,\ldots, x_r ) = (y_1,\ldots, y_r) \text{ at } S
\]
provided that for all $j \in S$ we have $x_j = y_j$.
Given a cycle $C$ of an element of $S_r$, we say that two elements of $\Delta^r$ are equal at $C$ if the elements are equal at elements appearing in the cycle. 

We say that two elements of $\Delta^r$ are equal away from $S$ if they are equal on $\{1,\ldots,r\} - S$, and likewise for cycles.  

For $\pi \in G \wr S_r$, we say that $\pi$ fixes $(x_1,\ldots,x_r)$ at $S$ if $\pi (x_1,\ldots , x_r)$ equals $(x_1,\ldots, x_r)$ at $S$, and similarly for cycles and fixing away from.
\end{defn}

For $\pi^i= (\sigma_i,\tau^i) \in G \wr S_r$ and $(x_1,\ldots, x_r) \in \Delta^r$, notice that
\[
\pi^i (x_1,\ldots, x_r) = (x_1, \ldots, x_r)  \text{ at } S
\]
provided for all $j \in S$ 
\begin{equation} \label{eq:action}
\sigma_{i,j}( x_{\tau^{-i}(j)} ) = x_j.
\end{equation}

We begin with a reduction.

\begin{lemma} \label{lem:reduction1}
It suffices to prove Proposition~\ref{prop:wreaths} when $(a_1, \ldots, a_r)$ and $(b_1, \ldots, b_r)$ belong to cycles with length a power of $q$ and $s$ respectively, and when, writing $\pi = (\sigma,\tau)$, there exists a cycle $C$ of $\tau$ whose length is $d = q^{n_1} s^{n_2}$ and such that $\pi$ fixes $(a_1,\ldots, a_r)$ and $(b_1,\ldots, b_r)$ away from $C$.
\end{lemma}

\begin{proof}
We adopt the hypotheses of Proposition~\ref{prop:wreaths}.
Assume that $(a_1, \ldots, a_r)$ (resp. $(b_1, \ldots, b_r)$) belongs to a $\pi$-cycle of size $Aq^{*}$ (resp. $Bs^{*}$), where $A$ is prime to $q$ (resp. $B$ is prime to $s$) and $q^*$ is a power of $q$ (resp. $s^*$ is a power of $s$). If $q \mid B$ or $s \mid A$, then the statement \eqref{item:eb-1} in Proposition~\ref{prop:wreaths} holds.  Otherwise, $\gcd(q, B)=\gcd(s, A)=1$ and we set $\varpi:=\pi^{AB}$, so $(a_1, \ldots, a_r)$ and $(b_1, \ldots, b_r)$ belong to $\varpi$-cycles of size exactly $q^{*}$ and $s^{*}$ respectively. 

Write $\varpi = (\sigma',\tau')$.
Suppose that there are disjoint cycles $C_1$ and $C_2$ for $\tau'$ such that $\varpi$ does not fix $(a_1,\ldots ,a_r)$ at $C_1$ and $\varpi$ does not fix $(b_1,\ldots,b_r)$ at $C_2$.  Then consider the element $(x_1,\ldots, x_r) \in \Delta^r$ defined by $x_i = a_i $ if $i$ appears in $C_1$, and $x_i = b_i$ otherwise.  
Looking at \eqref{eq:wreathpower}, it is clear that the $j$-th entry of $\varpi^i (x_1,\ldots, x_r)$ will depend only on $x_k$ such that $k$ is in the same $\tau'$-orbit as $j$.  In particular, as $(x_1,\ldots ,x_r)$ equals $(a_1,\ldots, a_r)$ at $C_1$, as $\varpi$ does not fix $(a_1,\ldots,a_r)$ at $C_1$, and as the $\varpi$-orbit of $(a_1,\ldots,a_r)$ has size $q^*$, we deduce that
\[
\varpi^i (x_1,\ldots,x_r) = (x_1,\ldots,x_r)
\]
implies that $q \mid i$.  Likewise, working with $(b_1,\ldots,b_r)$ and $C_2$ we deduce that $s\mid i$.  Thus the $\varpi$-orbit of $(x_1,\ldots, x_r)$ is has size a multiple of $q s$.  This establishes statement Proposition~\ref{prop:wreaths}\eqref{item:eb-1} in this case.
%when there are disjoint cycles $C_1$ and $C_2$ for $\tau'$ such that $\varpi$ does not fix $(a_1,\ldots ,a_r)$ on $C_1$ and $\varpi$ does not fix $(b_1,\ldots,b_r)$ on $C_2$.

It remains to consider the case that there is a cycle $C$ for $\tau'$ such that $\varpi$ fixes both $(a_1,\ldots,a_r)$ and $(b_1,\ldots, b_r)$ away from $C$.  Let $e$ be the maximal divisor of $|C|$ that is not divisible by $q$ or $s$.  Then $C$ breaks into $e$ disjoint $\tau'^e$-cycles of size $|C|/e$, and we will study the action of $\varpi^e$.  If there are two cycles $C_1$ and $C_2$ for $\tau'^e$ on which $\varpi^e$ doesn't fix $(a_1,\ldots,a_r)$ and $(b_1,\ldots,b_r)$ respectively, then we can repeat the argument of the previous paragraph, establishing statement Proposition~\ref{prop:wreaths}\eqref{item:eb-1}. 

Otherwise, there is a unique cycle $C$ for $\tau'^e$ away from which $\varpi^e$ fixes $(a_1,\ldots,a_r)$ and $(b_1,\ldots,b_r)$.  The length of this cycle is $d:=|C|/e =q^{n_1} s^{n_2}$.  The sizes of the orbits of $(a_1,\ldots,a_r)$ and $(b_1,\ldots, b_r)$ under $\varpi^e$ are a power of $q$ and a power of $s$ respectively, because $e$ is relatively prime to $q$ and $s$.  A cycle for $\varpi^e$ of size divisible by $qs$ gives the same for $\pi$, and statement Proposition~\ref{prop:wreaths}\eqref{item:eb-2} for $\varpi^e = \pi^{ABe}$ gives the same for $\pi$ as $A$, $B$, and $e$ are coprime to $q$ and $s$.  This completes the reduction.
\end{proof}

\begin{proof}[Proof of Proposition~\ref{prop:wreaths}]
We may assume we are in the special case described in Lemma~\ref{lem:reduction1}.  Let $\pi = (\sigma,\tau)$.  Without loss of generality, the cycle $C$ consists of the integers $\{1,2,\ldots, d \} \subset \Delta$ and that $\tau$ acts on $C$ via
\[
\tau (j) = (j \Mod{d})+1 \text{ for } 1 \leq j \leq d = q^{n_1} s^{n_2}.
\]
Here $(j \Mod{d}) \in \{0,1,\ldots, d-1\}$ indicates the remainder when $j$ is divided by $d$.
Furthermore, since $\pi$ fixes $(a_1,\ldots,a_r)$ and $(b_1,\ldots,b_r)$ away from $C$, but the tuples lie in cycles of different lengths, they must differ at some element of  $C$.  So assume $a_1 \neq b_1$.

\begin{description}
			\item[Case 1] Suppose $d=1$. As $\pi$ fixes $(a_1,\ldots,a_r)$ away from $C = (1)$, by \eqref{eq:action} we see that
\[
\sigma_{1,j}( a_{\tau^{-1}(j)} ) = a_j \text{ for } j >1.
\]			
Because $\tau(1) =1$ and $\tau$ permutes $\{2,3,\ldots, r\}$, we see that
\[
\pi (a_1,a_2,\ldots, a_r) = (\sigma_{1,1}(a_1), \sigma_{1,2}( a_{\tau^{-1}(2)}) , \ldots, \sigma_{1,r}(a_{\tau^{-1}(r)})) = (\sigma_{1,1}(a_1), a_2, \ldots , a_r).
\]
We can make a similar calculation with $(b_1,b_2,\ldots, b_r)$, so statement Proposition~\ref{prop:wreaths}\eqref{item:eb-2} with $i=1$ and $t=1$ holds in this case.
		
			\item[Case 2] 
	Suppose $d$ is divisible by only one of $q$ and $s$. Without loss of generality, we assume $d$ is a power of $q$. We let $q^*$ and $s^*$ denote the size of the $\pi$-cycles that $(a_1, \ldots, a_r)$ and $(b_1, \ldots, b_r)$ belong to, and denote $q':=\max(d, q^*)$. Note that by construction we have $\tau^{q'}(j)=j$ for $1\leq j \leq d$.
	
	We write $\pi^i = ((\sigma_{i,1},\ldots, \sigma_{i,r}),\tau^i)$, and first suppose $\sigma_{q',1}(b_1)=b_1$.  We know that $\pi^{q's^*}$ fixes $(a_1, \ldots, a_r)$ and $(b_1, \ldots, b_r)$, and hence in light of \eqref{eq:wreathpower}	it fixes $(a_1, b_2, \ldots, b_r)$ since $\tau^{q's^*}(1)=1$. As $\pi^{q'}$ does not fix $(b_1, \ldots, b_r)$ and 
	\[
		\pi^{q'}(b_1, \ldots, b_r) = (\sigma_{q', 1}(b_{\tau^{-q'}(1)}), \ldots, \sigma_{q', r}(b_{\tau^{-q'}(r)}))=(b_1, \sigma_{q',2}(b_{\tau^{-q'}(2)}), \ldots, \sigma_{q', r}(b_{\tau^{-q'}(r)})),
	\]
	there exists $i>1$ such that $\sigma_{q', i}(b_{\tau^{-q'}(i)})\neq b_i$.
	%\yuan{The referee suggests it should be $\sigma_{q', i}(b_i)\neq b_{\tau^{-q'}(i)}$, but I think it should be $\sigma_{q', i}(b_{\tau^{-q'}(i)})\neq b_i$} 
	%Agree with Yuan.
	 Therefore, it follows that
	\begin{equation}\label{eq:neq-q'}
		\pi^{q'}(a_1, b_2, \ldots, b_r) \neq (a_1, b_2, \ldots, b_r).
	\end{equation} 
	On the other hand, as $\gcd(d, s^*)=1$, we see $\tau^{s^*}$ sends $1$ to $\tau^{s^*}(1)\neq 1$. Then we obtain $b_{\tau^{s^*}(1)}=\sigma_{s^*, \tau^{s^*}(1)}(b_1)$, because $\pi^{s^*}$ fixes $(b_1, \ldots, b_r)$. Thus we obtain 
	\begin{equation}\label{eq:neq-s*}
		\pi^{s^*}(a_1, b_2, \ldots, b_r) \neq (a_1, b_2, \ldots, b_r)
	\end{equation}
	because the $\tau^{s^*}(1)$-th component of the two sides are $\sigma_{s^*, \tau^{s^*}(1)}(a_1)$ and $b_{\tau^{s^*}(1)}=\sigma_{s^*, \tau^{s^*}(1)}(b_1)$, which are not equal by the assumption $a_1\neq b_1$. By \eqref{eq:neq-q'} and \eqref{eq:neq-s*}, $(a_1, b_2, \ldots, b_r)$ belongs to a $\pi$-cycle of size divisible by $q$ and $s$ since we knew $\pi^{q' s^*}$ fixes $(a_1,b_2,\ldots,b_r)$, which proves Proposition~\ref{prop:wreaths}\eqref{item:eb-1}.
	
	Now suppose $\sigma_{q', 1}(b_1)\neq b_1$. We consider the element $(b_1, a_2, \ldots, a_r)$, which, similarly as above,  we will show belongs to a $\pi$-cycle of size divisible by $q's^*$. Note that
	\begin{equation*}
		\pi^{q'} (b_1, a_2, \ldots, a_r) \neq (b_1, a_2, \ldots, a_r)
	\end{equation*}
	as the first components are $\sigma_{q',1}(b_1)$ and $b_1$ respectively. So in this case it suffices to show
	\begin{equation}\label{eq:neq-s*-2}
		\pi^{s^*} (b_1, a_2, \ldots, a_r) \neq (b_1, a_2, \ldots, a_r).
	\end{equation}
	As $\gcd(q', s^*)=1$, for each $1<i\leq d$, there is a positive integer $j$ such that $s^*j\equiv i-1 \mod{q'}$. If \eqref{eq:neq-s*-2} does not hold, then $\pi^{s^*j}$ fixes $(b_1, a_2, \ldots, a_r)$. Then by studying the $i$-th component we have 
	\begin{equation}\label{eq:express-a}
		\sigma_{s^*j, i}(b_1) = a_i,
	\end{equation}
	and it implies $b_i=a_i$ for each $1<i\leq d$ because $\sigma_{s^* j, i}(b_1)$ is also the $i$-th component of $\pi^{s^*j} (b_1, \ldots, b_r)$ which is $b_i$.  Consider
	\begin{eqnarray*}
		\pi^{q'} (b_1, \ldots, b_r) &=& \pi^{q'} (b_1, a_2, \ldots, a_d, b_{d+1}, \ldots, b_r) \\
		&=& (\sigma_{q', 1}(b_1), a_2, \ldots, a_d, \sigma_{q', d+1}(b_{\tau^{-q'}(d+1)}), \ldots, \sigma_{q', r}(b_{\tau^{-q'}(r)}))
	\end{eqnarray*}
where the second equality uses that $\sigma_{q',i}(a_i) = a_i$ for $1 \leq i \leq d$ since $\pi^{q'}$ fixes $(a_1,a_2,\ldots,a_r)$.  Notice that $\pi^{q'} (b_1, \ldots, b_r)$ belongs to a cycle of size $s^*$.  Then for $j$ with $s^* j \equiv i-1 \mod{q'}$, looking at the $i$th components of $\pi^{s^*j} \circ \pi^{q'}  (b_1,\ldots,b_r)$ we conclude, for each $1< i \leq d$, that
 	\begin{equation}\label{eq:express-a-2}
		a_i =\sigma_{s^* j, i}(\sigma_{q',1}(b_1)),
	\end{equation}
  Then \eqref{eq:express-a} and \eqref{eq:express-a-2} contradict the assumption $\sigma_{q',1}(b_1)\neq b_1$, establishing \eqref{eq:neq-s*-2}.
		
			\item[Case 3] Suppose $d = q^{n_1} s^{n_2}$ with $n_1, n_2 >0$. We denote $q':=\max(q^*, q^{n_1})$ and $s':=\max(s^*, s^{n_2})$. We will show that statement \eqref{item:eb-1} holds in this case.  	
We begin by computing that 
\begin{equation}\label{eq:compute}
\pi^{q'} (b_1,a_2,\ldots a_r) = ( a_1, \ldots, a_{\tau^{q'}(1) -1}, \sigma_{q', \tau^{q'}(1)} (b_1), a_{\tau^{q'}(1) +1}, \ldots , a_r)
\end{equation}		
using \eqref{eq:action} and the hypothesis that $\pi^{q^*} (a_1,\ldots, a_r) = (a_1,\ldots, a_r)$.  As $a_1 \neq b_1$, we see that 
	\begin{equation} \label{eq:firstclaim}
		\pi^{q'}(b_1, a_2, \ldots, a_r) \neq (b_1, a_2, \ldots a_r)
	\end{equation}
	
We next prove that 
\begin{equation} \label{eq:secondclaim}
\pi^{s'}(b_1,a_2,\ldots ,a_r) \neq (b_1, a_2,\ldots, a_r).
\end{equation}
Assume otherwise, that $\pi^{s'} (b_1, a_2, \ldots, a_r) = (b_1, a_2, \ldots, a_r)$; by considering the $\tau^{s'}(1)$-th component of this equality we have that 
	\begin{equation}\label{eq:comp-1}
		\sigma_{s', \tau^{s'}(1)}(b_1)=a_{\tau^{s'}(1)}.
	\end{equation}
	On the other hand, $\pi^{q'}(b_1, a_2, \ldots, a_r)$ is also fixed by $\pi^{s'}$ by our assumption. So, similarly, by \eqref{eq:compute} and considering the $\tau^{s'}(1)$-th component of 
\[
	\pi^{s'}  (\pi^{q'}(b_1, a_2, \ldots, a_r) )= \pi^{q'} (b_1, a_2, \ldots, a_r),
\]
we see that 
	\begin{equation}\label{eq:comp-2}
		\sigma_{s', \tau^{s'}(1)}(a_1)=a_{\tau^{s'}(1)}.
	\end{equation}
	Then \eqref{eq:comp-1} and \eqref{eq:comp-2} contradict the assumption that $a_1\neq b_1$, so we proved \eqref{eq:secondclaim}.  
	
On the other hand, since $\tau^{q' s'}=1$ and $\pi^{q' s'}$ fixes $(a_1,\ldots,a_r)$ and $(b_1,\ldots, b_r)$, we see that $\pi^{q' s'}$ fixes $(b_1,a_2,\ldots , a_r)$.  Then \eqref{eq:firstclaim} and \eqref{eq:secondclaim} show that $(b_1,a_2,\ldots,a_r)$ lies in a cycle of order divisible by $qs$.
\end{description}
\end{proof}

\begin{corollary}\label{cor:embedding}
		Let  $m, r, n$ be positive integers with $n=m^r$, and let $q, s$ be distinct primes. Assume that $G$ is a permutation group acting on $\Delta=\{1, \ldots, m\}$, such that if $g\in G$ has two cycles of sizes divisible by $q$ and $s$ respectively, then $g$ has a cycle of size divisible by $qs$. Consider the embedding $\iota: G \wr S_r \hookrightarrow S_n$ defined by the natural action of $G\wr S_r$ on $\Delta^r$.
		If, for some $\pi \in G\wr S_r$, the image $\iota(\pi)$ has a cycle of size divisible by $q$ and a cycle of size divisible by $s$, then $\iota(\pi)$ also has a cycle of size divisible by $qs$.
	\end{corollary}
	
	\begin{proof}
		Assume that $(a_1, \ldots, a_r)$ and $(b_1, \ldots, b_r)$ belong to $\pi$-cycles of size divisible by $q$ and $s$ respectively. By Proposition~\ref{prop:wreaths}, it's enough to consider the situation in the case \eqref{item:eb-2}. Let $t$ and $i$ be described in Proposition~\ref{prop:wreaths}\eqref{item:eb-2}, and write $\pi=(\sigma, \tau)$. Then for $j \neq i$, $\pi^t$ fixes the $j$-th component of $(a_1, \ldots, a_r)$ and $(b_1, \ldots, b_r)$, and $\tau^t$ fixes $i$. Therefore $a_i$ and $b_i$ belongs to $\sigma_{t,i}$-cycles in $\Delta$ of size divisible by $q$ and $s$ respectively. Then by the assumption in this corollary, we see that $\sigma_{t,i} \in G$ has a cycle of size divisible by $qs$, and hence $\pi^t$ also has a cycle of size divisible by $qs$.
	\end{proof}

\begin{corollary} \label{cor:eliminatepower}
Fix a prime $q$, an integer $A>1$, and an element $\pi =(\sigma,\tau) \in S_m \wr S_r$.  Assume $(a_1,\ldots,a_r)$ and $(b_1,\ldots,b_r)$ belong to $\pi$-cycles of length divisible by $q$ and length exactly $A$ respectively.  
Suppose that for every prime $s \mid A$, $\iota(\pi)$ does not have a cycle of size divisible by $qs$.
Then there exists $i$ such that $\tau(i)=i$ and $\pi(b_1, \ldots, b_r)_j = b_j$ for any $j\neq i$.
\end{corollary}

Here $\pi(b_1, \ldots, b_r)_j$ represents the $j$-th coordinate of $\pi(b_1, \ldots, b_r)$.

\begin{proof}		
Let $\{s_1,s_2,\ldots, s_e\}$ be the prime divisors of $A$, 
and write $\pi= (\sigma,\tau)$.  
By Proposition~\ref{prop:wreaths}, for each $s_j$, we have $t_{s_j}$ and $i_{s_j}$ such that
	\begin{equation} \label{eq:mostlyfixed}
	   \pi^{t_{s_j}}(a_1, \ldots, a_r)_k = a_k, \text{ and }  \pi^{t_{s_j}}(b_1, \ldots, b_r)_k = b_k,  \text{ for any } k\neq i_{s_j}
	\end{equation}
and such that $\tau^{t_{s_j}}(i_{s_j}) = i_{s_j}$.  Notice that this last condition together with \eqref{eq:wreathpower} shows that \eqref{eq:mostlyfixed} continues to hold if we modify the $i_{s_j}$-th entry of $(a_1, \ldots, a_r)$ or of $(b_1, \ldots, b_r)$.  

If $i_{s_j} \neq i_{s_{j'}}$, we would have that $\pi^{t_{s_{j}}t_{s_{j'}}}(a_1, \ldots, a_r)=(a_1, \ldots, a_r)$ which contradicts the fact that $q\nmid t_{s_{j}}t_{s_{j'}}$. So $i_{s_j}$ does not depend on the choice of $j$ and we define $i := i_{s_1}$.  

Since $s_j \nmid t_{s_j}$ for each $j$, the integers $s_1,s_2,\ldots, s_e$ are relatively prime and so by the Chinese Remainder Theorem there exists positive integers 
 $c_1, \ldots, c_{e}$ such that $\sum_{j=1}^{e} c_j t_{s_j} \equiv 1 \mod{A} $; as $(b_1,\ldots, b_r)$ lies in a cycle of length $A$ we conclude that 
\[
\pi^{\sum c_j t_{s_j}} (b_1, \ldots, b_r) = \pi( b_1, \ldots, b_r).
\]
On the other hand, since each $\pi^{t_{s_j}}$ only modifies the $i$th entry of $(b_1,\ldots,b_{i-1}, b, b_{i+1},\ldots,b_r)$ for any $b \in \{1,\ldots, m\}$, we have $\pi^{\sum c_j t_{s_j}} (b_1, \ldots, b_r)_k =b_k$ for $k\neq i$ as desired.
\end{proof}
	
\subsection{The Case \texorpdfstring{$\ell \equiv 1 \mod{4}$}{p=1 mod 4}} \label{sec:p1mod4}
We now give a short proof of the  $\ell \equiv 1 \mod{4}$ case of Theorem~\ref{thm:max-orbit}, following the strategy of \cite[\S3]{mp18}.  The idea is to apply Jordan's Theorem \cite[Thm.~3.1]{mp18}, which states that every primitive permutation group on $n$ elements which contains a cycle of length a prime $p$ with $p \leq n-3$ must contain the alternating group.

The key observation is that $C_1(\pm2) \subset \cO(\ell)$, as $C_1(\pm 2)$ is a single $\overline{Q}_\ell$ orbit and has size $\ell$ (see \cite[Lem.~2.2]{mp18}), while $|Y^*(\ell) - \cO(\ell)| \leq \ell^{1/2}$ since $\ell\geq N_{1/2}$.  By studying the action of $\rot_1$, as in the original it can be deduced that an appropriate power is a $\ell$-cycle $\cO(\ell)$ containing the conic $C_1(\pm2)$.  We also see that $|\cO(\ell)| \geq |Y^*(\ell)|-\ell^{1/2} = \frac{\ell(\ell+3)}{4} - \ell^{1/2} > \ell+3$ since $\ell \geq 13$.  Thus it suffices to check the permutation action is primitive and apply Jordan's theorem.

The proof then proceeds as in \cite[S3]{mp18}, analyzing blocks of $\cO(\ell)$ and again using that $C_1(\pm 2) \subset \cO(\ell)$.

\subsection{Finishing the Proof} 
Finally, we will deduce Theorem~\ref{thm:main_A} from Theorem~\ref{thm:max-orbit}.  Let $\ell$ be an odd prime with the property $\BP(\ell)$.
If $\ell \equiv 1 \mod{4}$ then assume that $\ell \geq \max(N_{1/2},13)$, while if $\ell \equiv 3 \mod{4}$ then $\ell \geq \max(N_{1/2},23)$.
%and the order of $\frac{3+\sqrt{5}}{2}\in \F_{\ell^2}$ is at least $32\sqrt{\ell+1}$.
Let $\cM(\overline{\pi_\ell})$ be the connected component of $\cM(\PSL_2(\F_\ell))^\abs$ containing the $\bQ$-rational point $\ol{\pi_\ell}$ from \S\ref{ss_rational_point}.  In fact, $\ol{\pi_\ell}$ is defined over $\bZ[1/|\SL_2(\F_\ell)|]$ (Remark~\ref{rmk:ringofdef}).  As $\cM(\PSL_2(\F_\ell))^\abs$ is defined over $\ZZ[1/|\SL_2(\F_\ell)|]$, it follows that $\cM(\overline{\pi_\ell})$ is also defined over $\ZZ[1/|\SL_2(\F_\ell)|]$.

Note that $R_3([3,3,3]) = [3,3,6]$, so $[3,3,6] \in \cO(\ell)$.  
Using the identification of Propositions~\ref{prop:identifyfibers} and \ref{prop:fibermarkoff} and the Galois correspondence, we know from \S\ref{ss_rational_point} that $\ol{\pi_\ell}$ corresponds to $[3,3,6]$, that 
the geometric fibers of $\cM(\overline{\pi_\ell})$ are identified with $\cO(\ell)$,
and that the monodromy group is identified with $\ol{Q}_\ell^+$.  From Theorem~\ref{thm:max-orbit}, we know that $\ol{Q}_\ell$ contains the alternating group on $\cO(\ell)$.  Since $\ol{Q}_\ell^+$ is a normal subgroup of $\ol{Q}_\ell$ of index $1$ or $2$ and the alternating group is simple, we deduce that $\ol{Q}_\ell^+$ is the alternating or symmetric group on $\cO(\ell)$.  Then applying Proposition~\ref{prop_tame} completes the proof of Theorem~\ref{thm:main_A}. \qed

\appendix
\section{Explicit ramification calculations} 

\subsection{Details for the Proof of Proposition~\ref{prop_explicit_ramification}}

Note that $X^*(3)$ is empty, and $\Tr$ is not a bijection for $\ell = 2$ (see \ref{prop:fibermarkoff}).

\begin{lemma}\label{lemma_explicit} Let $\ell\ge 5$ be a prime. We have the following descriptions of the automorphisms $r,s,t,\gamma_0,\gamma_{1728}$ and the permutations they induce on $X^*(\ell)$ via $\Tr$ (see \ref{prop:fibermarkoff}):
$$\begin{array}{rcl}
r : (a,b) & \mapsto & (a^{-1},b)\\
s : (a,b) & \mapsto & (b,a) \\
t : (a,b) & \mapsto & (a^{-1},ab) \\
\gamma_0 : (a,b) & \mapsto & (ab^{-1},a) \\
\gamma_{1728} : (a,b) & \mapsto & (b^{-1},a)
\end{array}\quad\stackrel{\Tr_*}{\longrightarrow}\quad
\begin{array}{rcl}
\ol{r} = R_3 : (x,y,z) & \mapsto & (x,y,xy-z) \\
\ol{s} = \tau_{12} : (x,y,z) & \mapsto & (y,x,z) \\
\ol{t} = \tau_{23} : (x,y,z) & \mapsto & (x,z,y) \\
\ol{\gamma}_0 : (x,y,z) & \mapsto & (xy-z,x,x^2y-xz-y) \\
\ol{\gamma}_{1728} : (x,y,z) & \mapsto & (y,x,xy-z)
\end{array}$$

Moreover, we have
\begin{enumerate}
    \item $\gamma_0$ acts on $X^*(\ell)$ and $Y^*(\ell)$ as permutations of order 3 with exactly one fixed point. This fixed point is given by the triple $(3,3,6)$.
    \item $\gamma_{1728}$ acts on $X^*(\ell)$ as a permutation of order 2 with exactly two fixed points if $\ell\equiv 1,7\mod 8$ and with no fixed points otherwise.  When $\ell\equiv 1,7\mod 8$, the fixed points are given by $(\pm 2\alpha,\pm 2\alpha,4)$, where $\alpha$ is a root of $x^2-2$ in $\bF_\ell$. Thus in $Y^*(\ell)$, $\gamma_{1728}$ acts as a permutation of order 2 with a unique fixed point if $\ell\equiv 1,7\mod 8$ and no fixed points otherwise.
\end{enumerate}
\end{lemma}
\begin{proof} The formulas for $r,s,t$ were checked in Proposition~\ref{prop:fibermarkoff}. For $\gamma_0,\gamma_{1728}$, one can check that the following equalities hold in $\Aut(F_2)$:
$$\gamma_0 = s\circ r\circ s\circ t\circ r\circ s,\qquad \gamma_{1728} = s\circ r$$
The descriptions of $\ol{\gamma}_0,\ol{\gamma}_{1728}$ then follow from those of $\ol{r},\ol{s},\ol{t}$, noting that $\Tr$ induces an \emph{anti-homomorphism} $\Tr_* : \Aut(F_2)\rightarrow\Aut(\bA^3)$.

Next, we show that $\gamma_0$ has the desired properties. The image of $\gamma_0$ in $\GL_2(F_2^\ab)$ is the order 6 matrix $\spmatrix{1}{1}{-1}{0}$. However $\gamma_0^3$ is given by $(a,b)\mapsto (a^{-1},b^{-1})$ (up to $\Inn(F_2)$) which visibly acts trivially on $X^*(\ell)$ and hence also on $Y^*(\ell)$. Thus, $\gamma_0$ acts with order 3 on $Y^*(\ell)$, and we wish to show that it has exactly one fixed point.

From the description of $\gamma_0$, we find that the image of $(x,y,z)$ in $Y^*(\ell)$ is a fixed point if and only if
\begin{itemize}
\item[(a)] $xy-z = \epsilon_1x$
\item[(b)] $x = \epsilon_2y$
\item[(c)] $x^2y-xz-y = \epsilon_3z$
\end{itemize}
where $\epsilon_i = \pm1$ and $\epsilon_1\epsilon_2\epsilon_3 = 1$. Moreover $(x,y,z)$ is a fixed point in $X^*(\ell)$ if (a),(b),(c) hold with $\epsilon_i = 1$. Substituting (a) into (c) we find
$$\epsilon_1x^2 - y = \epsilon_3z$$
Thus by (b), we find that the fixed points are of the form
$$P_{\epsilon_1,\epsilon_2,x} := (x,\epsilon_2x,\epsilon_2x^2 - \epsilon_1x)$$
for $x$ arbitrary and $\epsilon_1,\epsilon_2 = \pm1$. We wish to count the number of points of this form which lie on $X^*(\ell)\subset\bF_\ell^3$. Any such point must satisfy
\begin{eqnarray*}
x^2 + x^2 + (\epsilon_2x^2-\epsilon_1x)^2 - x^2(x^2 - \epsilon_1\epsilon_2x) & = & 0 \\
x^2(\epsilon_1\epsilon_2x+3) & = & 0
\end{eqnarray*}
If $\epsilon_1 = \epsilon_2 = 1$ then the only solutions are $x = 0,-3$, corresponding to the points $P_{1,1,0} = (0,0,0)$ and $P_{1,1,-3} = (-3,-3,6)$. The former does not lie on $X^*(\ell)$, so $(-3,-3,6)$ is the unique fixed point in $X^*(\ell)$. If we allow $\epsilon_1,\epsilon_2$ to be arbitrary in $\{\pm 1\}$, then one obtains additionally the points $P_{-1,-1,-3} = (-3,3,-6)$, $P_{1,-1,-3} = (-3,3,-6)$, and $P_{-1,1,-3} = (-3,-3,6)$ which all describe the same point in $Y^*(\ell)$.

Next we consider $\gamma_{1728}$. Again, in $\Out(F_2)$, $\gamma_{1728}$ has order 4, with $\gamma_{1728}^2 = \gamma_0^3$ and hence $\gamma_{1728}$ acts with order 2 on $X^*(\ell)$ and $Y^*(\ell)$.

Thus the image of $(x,y,z)$ in $Y^*(\ell)$ is a fixed point if and only if
\begin{itemize}
\item[(a)] $y = \epsilon_1x$
\item[(b)] $x = \epsilon_2y$
\item[(c)] $xy-z = \epsilon_3z$
\end{itemize}
where again $\epsilon_i = \pm1, \epsilon_1\epsilon_2\epsilon_3 = 1$, and $(x,y,z)$ is a fixed point in $X^*(\ell)$ if (a),(b),(c) hold with $\epsilon_i = 1$. Substituting (a) into (c) we get
$$z + \epsilon_3 z = \epsilon_1 x^2$$
In this case we must have $\epsilon_3 = 1$, or else $x = y = 0$, which would not yield any solutions in $X^*(\ell)$. Thus, any solution in $X^*(\ell)$ has the form:
$$(x,\epsilon_1x,\frac{\epsilon_1}{2}x^2)$$
Again the choice $\epsilon_1$ is irrelevant in $Y^*(\ell)$, so we may take $\epsilon_1 = 1$, in which case $(x,x,\frac{1}{2}x^2)\in X^*(\ell)$ if and only if
$$x\ne 0\quad\text{and}\quad  x^2 + x^2 + \frac{1}{4}x^4 = \frac{1}{2}x^4$$
Rearranging, we get
$$\frac{1}{4}x^2(x^2-8) = 0$$
Since $x\ne 0$, we must have $x = \pm\alpha$ where $\alpha\in\bF_\ell^\times$ is a root of $x^2-8$. Such an $\alpha$ exists if and only if $\left(\frac{2}{\ell}\right) = 1$, which occurs if and only if $\ell\equiv 1,7\mod 8$. In $Y^*(\ell)$ both $x = \alpha, x = -\alpha$ give the same point, whereas in $X^*(\ell)$, we obtain two fixed points.
\end{proof}

\bibliographystyle{amsalpha}
\bibliography{mrc}

\end{document}